\documentclass[11pt, twoside]{article}
\usepackage{bcc-cambridge-class,mathrsfs,verbatim}

\bcctitle{Intersection Problems in Extremal Combinatorics: Theorems, Techniques and Questions Old and New.}
\bccshorttitle{Intersection problems}

\bccname{David Ellis}
\bccshortname{David Ellis} 

\bccaddressa{School of Mathematics \\ University of Bristol \\ Fry Building \\ Woodland Road \\ Bristol \\ BS8 1UG \\ United Kingdom.}
\bccemaila{david.ellis@bristol.ac.uk}
\DeclareMathOperator{\Domain}{\textrm{Domain}}

\begin{document}

\global\long\def\connected{\text{highly connected}}
\global\long\def\f{\mathcal{F}}
\global\long\def\pn{\mathcal{P}\left(\left[n\right]\right)}
\global\long\def\g{\mathcal{G}}
\global\long\def\l{\mathcal{L}}
\global\long\def\s{\mathcal{S}}
\global\long\def\j{\mathcal{J}}
\global\long\def\d{\mathcal{D}}
\global\long\def\Cay{\mathrm{Cay}}
\global\long\def\GL{\mathrm{GL}}
\global\long\def\Inf{}
\global\long\def\Id{\textrm{Id}}
\global\long\def\Tr{\mathrm{Tr}} 
\global\long\def\pn{\mathcal{P}([n])}
\global\long\def\sgn{\textrm{sgn}}
\global\long\def\p{\mathcal{P}}
\global\long\def\Prob{\text{Prob}}
\global\long\def\h{\mathcal{H}}
\global\long\def\n{\mathbb{N}}
\global\long\def\a{\mathcal{A}}
\global\long\def\b{\mathcal{B}}
\global\long\def\j{\mathcal{J}}
\global\long\def\c{\mathcal{C}}
\global\long\def\e{\mathbb{E}}
\global\long\def\x{\mathbf{x}}
\global\long\def\y{\mathbf{y}}
\global\long\def\av{\mathsf{A}}
\global\long\def\chop{\mathrm{Chop}}
\global\long\def\stab{\mathrm{Stab}}
\global\long\def\Span{\mathrm{Span}}
\global\long\def\codim{\mathrm{codim}}
\global\long\def\Var{\mathrm{Var}}
\global\long\def\rank{\mathrm{rank}}
\global\long\def\t{\mathsf{T}}

\makebcctitle

\begin{abstract}
The study of intersection problems in extremal combinatorics dates back perhaps to 1938, when Paul Erd\H{o}s, Chao Ko and Richard Rado proved the (first) `Erd\H{o}s-Ko-Rado theorem' on the maximum possible size of an intersecting family of $k$-element subsets of a finite set. Since then, a plethora of results of a similar flavour have been proved, for a range of different mathematical structures, using a wide variety of different methods. Structures studied in this context have included families of vector subspaces, families of graphs, subsets of finite groups with given group actions, and of course uniform hypergraphs with stronger or weaker intersection conditions imposed. The methods used have included purely combinatorial methods such as shifting/compressions, algebraic methods (including linear-algebraic, Fourier analytic and representation-theoretic), and more recently, analytic, probabilistic and regularity-type methods. As well as being natural problems in their own right, intersection problems have revealed connections with many other parts of combinatorics and with theoretical computer science (and indeed with many other parts of mathematics), both through the results themselves, and the methods used.

In this survey paper, we discuss both old and new results (and both old and new methods), in the field of intersection problems. Many interesting open problems remain; we will discuss several such. For expositional and pedagogical purposes, we also take this opportunity to give slightly streamlined versions of proofs (due to others) of several classical results in the area. This survey is intended to be useful to graduate students, as well as to more established researchers. It is a somewhat personal perspective on the field of intersection problems, and is not intended to be exhaustive; the author apologises for any omissions. It is an expanded version of a paper that will appear in the Proceedings of the 29th British Combinatorial Conference, University of Lancaster, 11th-15th July 2022.
\end{abstract}

\section{Introduction: three proofs of one theorem.}
Intersection problems are of great interest, and importance, in extremal combinatorics and theoretical computer science. Roughly speaking, they take the following form: how large can a family of (mathematical) objects be, if any two of the objects in the family intersect in some specified way? (`Two' can be replaced by an integer greater than two, in some problems, but for now we will try to keep things simple.)
\newline

To make things concrete, let us give a specific example.
\begin{qn}
\label{question:power-set}
If $X$ is an $n$-element set, what is the maximum possible size of a family $\mathcal{F}$ of subsets of $X$, such that $A \cap B \neq \emptyset$ for all $A,B \in \f$?
\end{qn}
This question turns out to be rather easy. Without loss of generality, we may assume that $X = [n] := \{1,2,\ldots,n\}$, the standard $n$-element set. If $\mathcal{F}$ is a family as in Question \ref{question:power-set}, then it can contain at most one of $S$ and $[n] \setminus S$, for any subset $S \subset [n]$. Writing $\mathcal{P}([n])$ for the power-set of $[n]$ (i.e.\ the set of all subsets of $[n]$), and noting that the pairs $\{S,[n]\setminus S\}$ partition $\mathcal{P}([n])$, it follows that $\mathcal{F}$ contains at most half of all the subsets of $[n]$, i.e.\ $|\mathcal{F}| \leq 2^{n-1}$. This upper bound is sharp, as can be seen by taking $\mathcal{F}$ to be all subsets of $[n]$ containing the element $1$. There are many other ways of attaining this upper bound: one can also take $\mathcal{F}$ to be all subsets of $[n]$ of size greater than $n/2$, if $n$ is odd, or one can take $\mathcal{F}$ to be all subsets of $[n]$ containing at least two elements of $\{1,2,3\}$, if $n \geq 3$. \label{triv}

Some more terminology is helpful. A family of sets $\mathcal{F}$ is said to be {\em intersecting} if any two sets in $\mathcal{F}$ have nonempty intersection. So, in the previous paragraph, we proved the simple fact that any intersecting family of subsets of an $n$-element set has size at most $2^{n-1}$.

Let us now turn to a somewhat harder question, first considered by Paul Erd\H{o}s, Chao Ko and Richard Rado in 1938:

\begin{qn}
\label{question:ekr}
For $n,k \in \mathbb{N}$, what is the maximum possible size of an intersecting family of $k$-element subsets of $\{1,2,\ldots,n\}$?
\end{qn}

An aside: if $X$ is a finite set, we write ${X \choose k} : = \{S \subset X:\ |S|=k\}$ for the set of all $k$-element subsets of $X$. A family of sets where all the sets are of the same size is often called a {\em uniform} family; if all the sets have size $k$, it is called a {\em $k$-uniform} family. By contrast, a family of sets where the sets have different sizes is often called a {\em non-uniform family}. Problems about subsets of ${X \choose k}$, for some finite set $X$ and $k \in \mathbb{N}$, are often called `uniform' problems; problems about subsets of $\mathcal{P}(X)$ are often called `non-uniform' problems. So Question \ref{question:ekr} is an example of a `uniform' problem: a problem about uniform set-families. Question \ref{question:power-set} is an example of a `non-uniform' problem.

It makes sense, in Question \ref{question:ekr}, to impose the restriction that $k \leq n/2$, since if $k > n/2$, then any two $k$-element subsets of $[n]$ have nonempty intersection, and therefore the family of all $k$-element subsets of $X$ is an intersecting family: the question is `trivial'. In fact, the question is also easy when $k=n/2$ and $n$ is even: essentially the same pairing argument as in Question 1.1, shows that if $n$ is even, then an intersecting family $\mathcal{F} \subset {[n] \choose n/2}$ has $|\mathcal{F}| \leq \tfrac{1}{2} {n \choose n/2}$; this upper bound can be attained by taking $\mathcal{F}$ to be all $(n/2)$-element subsets of $[n]$ containing the element 1. However, Question \ref{question:ekr} is harder for $k < n/2$. In 1938, Erd\H{o}s, Ko and Rado proved the following theorem \cite{ekr}.

\begin{theorem}[Erd\H{o}s, Ko and Rado, 1938]
\label{thm:ekr}
Let $n,k \in \mathbb{N}$ with $k \leq n/2$. If $\mathcal{F}$ is an intersecting family of $k$-element subsets of $[n]$, then
\begin{equation}
\label{eq:ekr-bound}
|\f| \leq {n-1 \choose k-1}.
\end{equation}
If $k < n/2$, then equality holds in (\ref{eq:ekr-bound}) if and only if there exists $i \in [n]$ such that
$$\f = \{S \subset [n]:\ |S|=k,\ i \in S\}.$$
\end{theorem}

Theorem \ref{thm:ekr} is often called `The' Erd\H{o}s-Ko-Rado theorem, even though there are several theorems of Erd\H{o}s, Ko and Rado that relate to intersecting families.

We remark that the conclusion of Theorem \ref{thm:ekr} illustrates a common (though by no means universal) feature of intersection problems: the `winning' families are often those consisting of all the objects containing a small, fixed sub-object (in this case, a single point) that guarantees what we are looking for in the intersection of each pair of objects (in fact, in the intersection of all of the objects in the family). Many intersection problems take the following form: given two finite families of sets $\mathcal{H}$ and $\mathcal{I}$, determine the maximum possible size of a subset $\f$ of $\mathcal{H}$ such that for any two sets $S,T \in \f$, the intersection $S \cap T$ contains a set in $\mathcal{I}$. Clearly, Question \ref{question:ekr} is of this form, with $\mathcal{H} = {[n] \choose k}$ and $\mathcal{I}$ consisting of all the nonempty sets (or, equivalently, all the singletons). An intersection problem of this form is said to have the {\em Erd\H{o}s-Ko-Rado property} if some family of the form $\{S \in \mathcal{H}:\ I \subset S\}$ (for some $I \in \mathcal{I}$) is an optimal solution to the problem. It is said to have the {\em strict Erd\H{o}s-Ko-Rado property} if all the optimal solutions are of this form. Rephrased in this language, Theorem \ref{thm:ekr} states that for $k \leq n/2$, Question \ref{question:ekr} has the Erd\H{o}s-Ko-Rado property, and that for $k < n/2$, the strict form of the property holds.

Slightly more generally, families of sets consisting of all the sets having specified intersection with some `small' set (of bounded size) are often called {\em juntas}. More precisely, a non-uniform family $\f \subset \mathcal{P}(X)$ is said to be a {\em $j$-junta} if there exists a subset $J \in {X \choose j}$ and a family $\mathcal{J} \subset \mathcal{P}(J)$ such that $S \in \f$ if and only if $S \cap J \in \mathcal{J}$, for all $S \subset X$. Similarly, a $k$-uniform family $\f \subset {X \choose k}$ is said to be a {\em $j$-junta} if there exists a subset $J \in {X \choose j}$ and a family $\mathcal{J} \subset \mathcal{P}(J)$ such that $S \in \f$ if and only if $S \cap J \in \mathcal{J}$, for all $S \in {X \choose k}$. So in the Erd\H{o}s-Ko-Rado theorem, the (unique) winning families are 1-juntas. By contrast, in Question \ref{question:power-set}, the winning family $\{S \subset [n]:\ 1 \in S\}$ is a 1-junta, and the winning family $\{S \subset [n]:\ |S \cap [3]| \geq 2\}$ is a 3-junta, but the winning family $\{S \subset [n]: |S| > n/2\}$ (for $n$ odd) is not a junta. (It is an $n$-junta, trivially, but it is not even an $(n-1)$-junta, and, thinking of $n \to \infty$ as our asymptotic parameter, $n$ does not have bounded size, so an $n$-junta is not `really' a junta.) In many intersection problems (though not all), the winning families turn out to be juntas. (We will shortly see examples of intersection theorems where none of the winning families are juntas --- so in particular, the Erd\H{o}s-Ko-Rado property does not hold.) \label{junta}

The proofs of the upper bound (\ref{eq:ekr-bound}) are (arguably) more interesting and (certainly) more elegant than the proof of the characterisation of the equality case, so in our discussion we will focus our attention on the proofs of the upper bound. 

There are now many proofs of (the upper bound in) Theorem \ref{thm:ekr}, of which we draw attention to three: the original proof of Erd\H{o}s, Ko and Rado (which uses `combinatorial shifting', a.k.a.\ `compressions', together with induction on $n$); the averaging proof of Katona (perhaps the most elegant of all the proofs), and the algebraic (spectral) proof of Lov\'asz. Each of these three proofs contains important ideas that have been used to tackle a range of other interesting and important problems in combinatorics and theoretical computer science, so we will discuss all three of them at some length.
\newline

We start with Katona's proof \cite{katona}, which Paul Erd\H{o}s described as a `Book Proof'. It is probably the shortest proof, and is purely combinatorial. It rests on the following claim.
\begin{claim}
Let $X$ be an $n$-element set, and let $k \leq n/2$. Let $\mathcal{F}$ be an intersecting family of $k$-element subsets of $X$. Let $C$ be any cyclic ordering of the elements of $X$. Then at most $k$ sets in $\mathcal{F}$ are (cyclic) intervals in $C$.
\end{claim}
This claim quickly implies the upper bound (\ref{eq:ekr-bound}). Indeed, for any cyclic ordering $C$ of the elements of $X$, there are exactly $n$ sets of size $k$ that are intervals in $C$, so the above claim says that an intersecting family can contain at most a $(k/n)$-fraction of them (for any cyclic ordering $C$). Averaging over all the cyclic orderings of $X$ immediately implies that
$$|\f| \leq \frac{k}{n}{n \choose k} = {n-1\choose k-1},$$
as required.

We now turn to the proof of the claim. Without loss of generality, we may assume that $X = \{1,2,\ldots,n\}$ and that $C = 123\ldots n1$, the `standard' cyclic ordering of $1,2,\ldots,n$. Let $\mathcal{F}$ be an intersecting family of $k$-element subsets of $\{1,2,\ldots,n\}$. If $\mathcal{F}$ contains no interval in $C$, then we are done, so without loss of generality, we may assume that $\mathcal{F}$ contains the cyclic interval $\{1,2,\ldots,k\}$. Any other set in $\mathcal{F}$ must intersect $\{1,2,\ldots,k\}$, so the only intervals in $C$ that $\mathcal{F}$ can possibly contain, are the $2k-1$ sets
\begin{align*}
&\{n-k+2,n-k+3,\ldots,n,1\},\\
&\{n-k+3,n-k+4,\ldots,n,1,2\},\\
&\ldots\\
&\{k-1,k,\ldots,2k-2\},\\
&\{k,k+1,\ldots,2k-1\}.
\end{align*}
Now we observe that for each $i \in \{2,\ldots,k\}$, $\mathcal{F}$ contains at most one of the cyclic intervals $\{i-k,i-k+1,\ldots,i-1\}$ and $\{i,i+1,\ldots,i+k-1\}$ (addition and subtraction being modulo $n$), since these two sets are disjoint. (Here we use the fact that $i-k+n > i+k-1$, which follows from $k \leq n/2$). It follows that $\mathcal{F}$ contains at most $k$ sets that are intervals in $C$, proving the claim.

\subsubsection*{Averaging over smaller subsets: a general proof-technique} It is worth abstracting one of the ideas of Katona's proof, as it occurs rather often in proofs in extremal combinatorics. Suppose we have a (finite) universe $U$ of mathematical objects, and a property $P$ of subsets of $U$ which is closed under taking subsets (i.e.\ if $\f \subset U$ has the property $P$, and $\g \subset \f$, then $\g$ also has the property $P$). Suppose we wish to prove that any subset $\mathcal{F}$ of $U$ with the property $P$, has $|\f| \leq M$. Now suppose that there is a group $H$ of permutations of $U$ acting transitively on the elements of $U$, such that elements of $H$ (when applied to subsets), preserve the property $P$. Rather than focussing on the entire universe $U$, we can focus on a smaller subset $\mathcal{S} \subset U$; if we can prove that for any subset $\mathcal{G}$ of $\mathcal{S}$ satisfying the property $P$, we have $|\mathcal{G}| \leq (M|\mathcal{S}|)/|U|$, then we are done, by averaging over all translates of $\mathcal{S}$ (by elements of $H$). In the case of Katona's proof, $U$ was ${[n] \choose k}$, the set of all $k$-element subsets of $[n]$, $H$ was the symmetric group $S_{[n]}$, and $\mathcal{S}$ was the set of all $k$-element subsets of $[n]$ that are (cyclic) intervals in the cyclic ordering $123\ldots n1$.
\newline

We now turn to a proof of Theorem \ref{thm:ekr} using combinatorial shifting (a.k.a.\ compressions) and induction on $n$, essentially a streamlined version of the original proof of Erd\H{o}s, Ko and Rado in \cite{ekr}. Like Katona's proof, it is purely combinatorial. A key tool in this proof is the {\em $ij$-compression} operator $C_{ij}$, defined as follows. We assume without loss of generality that $X = \{1,2,\ldots,n\}$, the standard $n$-element set, which we hereafter denote by $[n]$. For $1 \leq i < j \leq n$ and for a set $S \subset [n]$, we define
$$C_{ij}(S) = \begin{cases} (S \cup \{i\}) \setminus \{j\}, \text{ if } S \cap \{i,j\}=\{j\},\\ S, \text{ otherwise.}\end{cases}
$$
For a family $\mathcal{F} \subset \mathcal{P}([n])$, we define
$$C_{ij}(\mathcal{F}) = \{C_{ij}(S):\ S \in \mathcal{F}\} \cup \{S \subset [n]:\ C_{ij}(S) \in \f\}.$$
It is easy to check that $|C_{ij}(\mathcal{F})| = |\mathcal{F}|$ for any family $\mathcal{F} \subset \mathcal{P}([n])$, that if $\mathcal{F}$ is intersecting, then so is $C_{ij}(\mathcal{F})$, and that if all the sets in $\mathcal{F}$ have size $k$, then the same is true of $C_{ij}(\mathcal{F})$ (for any $i < j$). In other words, the $ij$ compression operation $C_{ij}$ preserves both the size of a family, the property of being a $k$-uniform family, and the property of being an intersecting family.

We say a family $\mathcal{F} \subset \mathcal{P}([n])$ is {\em left-compressed} if $C_{ij}(\mathcal{F}) = \mathcal{F}$ for all $i < j$. It is clear that, for an arbitrary family $\mathcal{F} \subset \mathcal{P}([n])$, we can apply some sequence of $C_{ij}$'s to $\mathcal{F}$ so as to produce a left-compressed family, $\mathcal{G}$ say. (If at some stage in this process, we have a family $\mathcal{F}'$ which is not $ij$-compressed for some $i < j$, we replace $\mathcal{F}'$ with $C_{ij}(\mathcal{F}')$, noting that this reduces $\sum_{S \in \mathcal{F}'} \sum_{x \in S} x$, a non-negative quantity, so this process must terminate.) We have $|\mathcal{G}| = |\mathcal{F}|$, and if $\mathcal{F}$ is intersecting, so is $\mathcal{G}$. Finally, if all the sets in $\mathcal{F}$ have size $k$, the same is true of $\mathcal{G}$ (compressions preserve the property of being $k$-uniform). Hence, in proving the bound (\ref{eq:ekr-bound}), we may assume that $\mathcal{F}$ is left-compressed.

The bound (\ref{eq:ekr-bound}) is trivial when $n=1$. Let $n \geq 2$, and assume by induction that (\ref{eq:ekr-bound}) holds when $n$ is replaced by $n-1$. Now let $k \leq n/2$ and let $\mathcal{F} \subset {[n] \choose k}$ be intersecting; by the above argument, we may and shall assume that $\mathcal{F}$ is left-compressed. If $k=n/2$, then we are done, by the easy pairing argument mentioned above; so we assume henceforth that $k < n/2$. Let us consider the two families
$$\mathcal{A}: = \{S \setminus \{n\}:\ S \in \f,\ n \in S\} \subset {[n-1] \choose k-1},\quad \mathcal{B} : = \{S \in \f:\ n \notin S\} \subset {[n-1] \choose k}.$$
We claim that $\mathcal{A}$ is an intersecting family. Indeed, suppose for a contradiction that there exist $A_1,A_2 \in \mathcal{A}$ such that $A_1 \cap A_2 = \emptyset$. Note that $A_1 \cup \{n\},A_2 \cup \{n\} \in \mathcal{F}$, and further that $[n] \setminus (A_1 \cup A_2) \neq \emptyset$ (using the fact that $k \leq n/2$). Pick $i \in [n] \setminus (A_1 \cup A_2)$. Since $\mathcal{F}$ is left-compressed, we must have $C_{in}(A_1 \cup \{n\}) \in \mathcal{F}$, but $C_{in}(A_1 \cup \{n\}) \cap A_2 = \emptyset$, contradicting the fact that $\mathcal{F}$ is intersecting.

Trivially, $\mathcal{B}$ is also an intersecting family, so by induction (noting that $k \leq (n-1)/2$), we have
$$|\mathcal{A}| \leq {n-2 \choose k-2},\quad |\mathcal{B}| \leq {n-2 \choose k-1}.$$
Finally, we have 
$$|\mathcal{F}| = |\mathcal{A}|+|\mathcal{B}| \leq {n-2 \choose k-2}+{n-2 \choose k-1} = {n-1 \choose k-1},$$
completing the inductive step, and proving the bound (\ref{eq:ekr-bound}).

\subsubsection*{Combinatorial shifting/compressions: a general proof-strategy} As with Katona's proof, it is worth abstracting one of the ideas of this proof (viz., compressions/shifting), as it is very useful, both in other parts of combinatorics and in theoretical computer science, but also indeed in other parts of mathematics (notably, geometry). Suppose we have some real-valued function $f$ defined on some universe $U$ of mathematical objects, and some property $P$ possessed by some of the objects in $U$ (formally, $P \subset U$), and we wish to prove that the maximum possible value of $f$, over all objects in $U$ possessing the property $P$, is equal to $M$ (where $M \in \mathbb{R}$). One potential proof-technique is to define a family of `shifting' operations $\{\mathcal{S}_i:\ i \in I\}$, with the properties that
\begin{enumerate}
\item If $A \in U$, then $\mathcal{S}_i(A) \in U$ for all $i \in I$;
\item If $A \in P$, then $\mathcal{S}_i(A) \in P$ for all $i \in I$;
\item $f(\mathcal{S}_i(A)) \geq f(A)$ for all $A \in U$; 
\item For any $A \in U$, we can apply some sequence of the `shifting' operations $\mathcal{S}_i$ to $A$ in such a way that the resulting object $B$ is stable under all of the $\mathcal{S}_i$;
\item It is (fairly) easy to prove that the maximum of $f(B)$, over all objects $B \in U$ that are stable under all of the shifting operations, is equal to $M$.
\end{enumerate}
It is clear that this technique `works', at least, when the sequence in item 4.\ is finite, as is usually the case in combinatorics. (In geometry, the sequence may need to be infinite, in which case one needs to worry about convergence issues, but these issues can often be dealt with successfully, though the convergence issue is sometimes delicate. This is the case, for example, with the proof of the isoperimetric inequality in the plane which uses Steiner symmetrizations; the proof-method was first proposed/sketched by Steiner \cite{steiner} in 1838, but it was only made rigorous more than seventy years later, by Carath\'eodory \cite{cara} in 1909, essentially by finding a sequence of Steiner symmetrizations under which the iterates of a body converge, in an appropriate metric.)

Within combinatorics, the general technique of `shifting', outlined above, has been particularly successful in the field of discrete isoperimetric inequalities. The most elegant proofs of the vertex-isoperimetric inequality and the edge-isoperimetric inequality for the $n$-dimensional hypercube, use shifting (with different compressions, in each case; see \cite{leader-survey}), as does the most elegant proof of the Kruskal-Katona theorem (due to Daykin \cite{daykin-kk}), and as do the only known proofs of the vertex- and edge-isoperimetric inequalities for the discrete $\ell^1$-grid; see \cite{bl-comp,bl-grid}. (Though the proof of the edge-isoperimetric inequality for the discrete $\ell^1$-grid, which is due to Bollob\'as and Leader \cite{bl-grid}, requires other difficult ingredients too.) Daykin \cite{daykin-ekr} observed that the Kruskal-Katona theorem in fact implies the upper bound in the Erd\H{o}s-Ko-Rado Theorem (Theorem \ref{thm:ekr}), and we shall see later that there are several other close connections between isoperimetric inequalities on the one hand, and intersection problems on the other.
\newline

We now turn to the algebraic (spectral) proof of Theorem \ref{thm:ekr}, due to Lov\'asz \cite{lovasz}. For this, we need a little more terminology. If $G = (V,E)$ is a graph, an {\em independent set} in $G$ is a subset $S \subset V$ such that $s_1s_2 \notin E(G)$ for all $s_1,s_2 \in S$. The {\em independence number} of a finite graph $G$ is the maximum size of an independent set in $G$. We denote the independence number of $G$ by $\alpha(G)$.

The upper bound in Theorem \ref{thm:ekr} can be rephrased as a statement about the independence numbers of Kneser graphs. For positive integers $k \leq n$, the {\em Kneser graph} $K_{n,k}$ is the graph with vertex-set ${[n] \choose k}$ in which two $k$-element sets are joined by an edge if they are disjoint. An intersecting family of $k$-element subsets of $[n]$ is precisely an independent set in the Kneser graph $K_{n,k}$, so the upper bound in Theorem \ref{thm:ekr} says precisely that an independent set in $K_{n,k}$ has size at most ${n-1 \choose k-1}$.

There are various well-known upper bounds on the independence number of a finite graph. One such is Hoffman's bound, which bounds the independence number of a finite $d$-regular graph in terms of the maximum and minimum eigenvalue of the adjacency matrix of the graph, and this is (essentially) the bound used in Lov\'asz's proof.

\begin{theorem}[Hoffman, 1974 (unpublished)]
\label{thm:hoffman}
Let $G = (V,E)$ be a finite, $d$-regular graph, and let $A$ be the adjacency matrix of $G$. Let $N = |V(G)|$, and let $d = \lambda_1 \geq \lambda_2 \geq \ldots \geq \lambda_{N}$ be the eigenvalues of $A$, repeated with their multiplicities. Let $S \subset V(G)$ be an independent set of vertices of $G$. Then
$$|S| \leq \frac{-\lambda_{N}}{d-\lambda_{N}} N.$$
Equality holds only if 
$$1_S - \frac{|S|}{N}(1,1,\ldots,1)$$
is a $\lambda_{N}$-eigenvector of $A$, where $1_{S}$ denotes the indicator vector of $S$.
\end{theorem}

Hoffman did not actually publish this result, and it seems to have first appeared in print in Lov\'asz's 1979 paper (which credits Hoffman). It can be proved straightforwardly, as follows.
\begin{proof}(of Theorem \ref{thm:hoffman}).
Equip $\mathbb{R}^{V(G)}$ with the inner product
$$\langle u,v \rangle = \frac{1}{N} \sum_{x \in V(G)} u(x) v(x).$$
Let $S \subset V(G)$ be an independent set, and let $v_1,\ldots,v_N$ be an orthonormal basis of eigenvectors of $A$ with $v_i$ having eigenvalue $\lambda_i$ (for each $i \in [N]$), and with $v_1$ being the all-1's vector $(1,1,\ldots,1): = \mathbf{f}$ (which has eigenvalue $d$). Expand the indicator vector $1_{S}$ in terms of this eigenbasis:
$$1_{S} = \sum_{i=1}^{N} \alpha_i v_i.$$
Note that
$$\alpha_1 = \langle 1_S, \mathbf{f} \rangle = |S|/N = \langle 1_S,1_S \rangle = \sum_{i=1}^{N} \alpha_i^2.$$
Now, 
\begin{equation} \label{eq:ip} 0 = \frac{1}{|V|} \sum_{x,y \in S} A_{x,y} = \langle A1_{S},1_{S} \rangle = \sum_{i=1}^{N} \lambda_i \alpha_i^2 \geq d \alpha_1^2 + \lambda_N \sum_{i=2}^{N} \alpha_i^2 = d \alpha_1^2 + \lambda_N(\alpha_1-\alpha_1^2).\end{equation}
Rearranging yields
$$\alpha_1 \leq \frac{-\lambda_N}{d-\lambda_N},$$
which is precisely Hoffman's bound. Clearly, equality holds in (\ref{eq:ip}) only if $\alpha_i = 0$ whenever $i \geq 2$ and $\lambda_i > \lambda_N$, which yields the equality statement of the theorem.
\end{proof}

Clearly, the Kneser graph $K_{n,k}$ is ${n-k \choose k}$-regular, so to deduce the upper bound in Theorem \ref{thm:ekr} from Hoffman's bound, it suffices to prove that the least eigenvalue of $K_{n,k}$ is $-{n-k-1 \choose k-1}$, for all $k \leq n/2$, and this was indeed proved by Lov\'asz. The Kneser graph $K_{n,k}$ is actually one of the adjacency matrices of the Johnson association scheme. The general theory of association schemes (see \cite{delsarte-73}) implies that $K_{n,k}$ has exactly $k+1$ eigenspaces $V_0,V_1,\ldots,V_k$, where $V_0$ is the 1-dimensional space of constant vectors, and for each $i \in [k]$, $V_i$ is the orthogonal complement of $V_{i-1}$ in
$$U_i: = \Span\{w_T:\ T \in {[n] \choose i}\},$$
where $w_T \in \mathbb{R}[{[n] \choose k}]$ is the indicator vector corresponding to the event that a subset `contains $T$', explicitly,
$$w_T(S) = \begin{cases} 1 & \text{ if } T \subset S,\\ 0 & \text{ otherwise.}\end{cases}$$
Equipped with this knowledge, it is fairly easy to calculate the eigenvalues of $K_{n,k}$ explicitly: the eigenvalue $\lambda_i$ corresponding to $V_i$ is given by
$$\lambda_i = (-1)^i {n - k - i \choose k-i}\quad (0 \leq i \leq k),$$
so in particular, the least eigenvalue is $-{n-k-1 \choose k-1}$, as required. The reader is referred to the paper of Lov\'asz for details of this calculation.

\subsubsection*{The `generalised harmonic analysis' approach} One aspect of Lov\'asz's proof is particularly worth abstracting. The use of Hoffman's bound is a (relatively simple) instance of a very general method (or philosophy, even) that has led to huge progress in extremal combinatorics, analysis and additive number theory, in recent years. Suppose we want to understand the (extremal) properties of those subsets $\mathcal{S}$ of some universe $U$ that satisfy a certain property, $P$. A general method of attack is to consider the indicator function (a.k.a.\ the indicator vector) $1_{\mathcal{S}} \in \mathbb{R}^U$, defined by $1_{\mathcal{S}}(u) = 1$ if $u \in \mathcal{S}$ and $1_{\mathcal{S}}(u) = 0$ if $u \notin \mathcal{S}$, and to expand $1_{\mathcal{S}}$ as a linear combination of functions with `nice' algebraic/analytic properties. One can then hope to use the information that $\mathcal{S}$ satisfies the property $P$, in combination with the above-mentioned expansion, to obtain some extremal information about $1_{\mathcal{S}}$, from which we can hope to deduce the desired (extremal) conclusion about $\mathcal{S}$. Often, as in the case of Lov\'asz's proof sketched above, this will involve applying a linear operator (related to the property $P$) to our expansion, and using spectral methods. Often (though not in the case of Lov\'asz's proof), the expansion can be in terms of the characters of an Abelian group, in which case we are doing harmonic analysis, or in terms of the representations of a non-Abelian group, in which case we are doing `non-Abelian harmonic analysis'. So one can view the general method outlined above, as a `generalisation' of the application of harmonic analysis to extremal problems. We will see several other instances of this general method being used to solve intersection problems (with best-possible bounds), later in this survey paper. Of course, harmonic analysis has been used with outstanding success in extremal problems in analytic number theory, ever since Hardy and Littlewood developed their `circle method'; Roth's seminal 1953 theorem on subsets of integers containing no 3-term arithmetic progressions, is another example, and harmonic analysis is ubiquitous in modern additive combinatorics / discrete analysis. But the bounds proved via these kinds of method, in these areas of Mathematics, are typically not exact: rather, they differ from the (conjectured) truth by very large constants or by logarithmic factors, so the mathematics involved has a slightly different flavour to it (more `analytic' and less `algebraic', perhaps), compared to the search for more `exact' bounds in extremal combinatorics, which we discuss in this survey.

The remainder of this survey paper is structured as follows. In Section \ref{sec:sets}, we discuss intersection problems for families of subsets of an unstructured ground-set. In Section \ref{sec:stab}, we discuss `stability' results that describe the structure of `large' intersecting families. In Section \ref{sec:structure}, we consider intersection problems where extra structure is imposed upon the ground-set, such as the additive structure of the integers, or a graph structure. In Section \ref{sec:more-complicated}, we consider intersection problems for families of more complicated mathematical objects, such as permutations, partitions, vector subspaces and linear maps. In Section \ref{sec:sym}, we consider what happens in intersection problems when additional `symmetry' requirements are imposed upon the intersecting families in question.

This survey is a somewhat personal perspective on the field of intersection problems, and is not intended to be exhaustive; if a discussion of the reader's favourite theorem is omitted, I can only apologise profusely. The reader is referred to the excellent surveys of Borg \cite{borg}, Frankl and Tokushige \cite{frankl-survey}, and Godsil and Meagher \cite{godsil-meagher-survey}, for slightly different perspectives on the field.

\section{Intersection problems for families of subsets of an (unstructured) set}
\label{sec:sets}
In this section, we give a quick (and by no means exhaustive) survey of intersection theorems and intersection problems, for families of subsets of a ground-set where there is no special `structure' on the ground-set.

\subsection{The Complete Intersection Theorem, and some of its predecessors}

In view of the Erd\H{o}s-Ko-Rado theorem (Theorem \ref{thm:ekr}), it is natural to ask what happens if we strengthen the intersection condition, while keeping the `universe' the same (viz., ${[n] \choose k}$). One natural strengthening is as follows. For a positive integer $t$, we say a family of sets is {\em $t$-intersecting} if any two sets in the family have intersection of size at least $t$. This leads to the following generalisation of Question \ref{question:ekr}.
\begin{qn}
\label{question:ak} 
For each triple of positive integers $(n,k,t)$, what is the maximum possible size $M(n,k,t)$ of a $t$-intersecting family of $k$-element subsets of $\{1,2,\ldots,n\}$?
\end{qn}
This question is only interesting for $n > 2k-t$, since for $n \leq 2k-t$, the family ${[n] \choose k}$ is itself $t$-intersecting, so trivially, $M(n,k,t) = {n \choose k}$ for all $n \leq 2k-t$.

In \cite{ekr}, Erd\H{o}s, Ko and Rado gave a short proof that for each $t < k$, there exists $n_0 = n_0(k,t) \in \mathbb{N}$ such that for all $n \geq n_0(k,t)$, we have
\begin{equation}
\label{eq:tumvirate-0} M(n,k,t) = {n-t \choose k-t};
\end{equation}
the extremal families in this case are precisely the families of the form $\{S \in {[n] \choose k}:\ S \supset B\}$, for $B \in {[n] \choose t}$, sometimes called the {\em $t$-umvirates}. However, for small $n$ and $t > 1$, these families are no longer extremal: for example, when $n=2k$ and $t=2$, it is easy to check that
$$\left|\left\{S \in {[2k] \choose k}:\ |S \cap \{1,2,3,4\}| \geq 3\right\}\right| > \left|\left\{S \in {[2k] \choose k}:\ \{1,2\} \subset S\right\}\right|$$
for all $k$ sufficiently large. Frankl conjectured in 1978 that
$$M(n,k,t) = \max_{0 \leq i \leq (n-t)/2} |\mathcal{F}_i|,$$
where
$$\mathcal{F}_i: = \left\{S \in {[n] \choose k}:\ |S \cap [t+2i]| \geq t+i\right\}\quad (0 \leq i \leq (n-t)/2).$$
Frankl's conjecture was finally proved in 1997 by Ahlswede and Khachatrian \cite{AK}; it is now known as the `complete intersection theorem'. It is a remarkable and seminal result. The proof of Ahlswede and Khachatrian in \cite{AK} is entirely combinatorial, and is very intricate (perhaps rather miraculous), making clever use of $ij$-compressions (defined in the previous section), generating families, and taking complements. (We say a family $\mathcal{G} \subset \mathcal{P}([n])$ {\em generates} a $k$-uniform family $\mathcal{F} \subset {[n] \choose k}$ if $\mathcal{F} = \{S \in {[n] \choose k}:\ T \subset S \text{ for some } T \in \mathcal{G}\}$. The relevance of taking complements stems from the observation that if $\mathcal{F} \subset {[n] \choose k}$ is $t$-intersecting if and only if $\overline{\mathcal{F}} : = \{[n] \setminus S:\ S \in \mathcal{F}\} \subset {[n] \choose n-k}$ is $(n-2k+t)$-intersecting.) In \cite{push-pull}, Ahlswede and Khachatrian give a different proof (of the Complete Intersection Theorem), which in a sense is `dual' to the proof in \cite{AK}. 

It is important to note that in 1984, Wilson \cite{wilson} gave a spectral proof of the following special case of the Ahlswede-Khachatrian's Complete Intersection Theorem.
\begin{theorem}[Wilson, 1984]
\label{thm:wilson}
If $n,k,t \in \mathbb{N}$ with $n \geq (k-t+1)(t+1)$, and $\f \subset {[n] \choose k}$ is $t$-intersecting, then
\begin{equation}\label{eq:tumvirate} |\f| \leq {n-t \choose k-t},\end{equation}
and equality holds if and only if $\f$ is a $t$-umvirate, i.e.\ $\f =\{S \in {[n] \choose k}:\ S \supset B\}$ for some $B \in {[n] \choose t}$.
\end{theorem}
The bound $n \geq (k-t+1)(t+1)$ in Wilson's theorem is best-possible, so some thirteen years before the Ahlswede-Khachatrian theorem, Wilson had determed the values of $(n,k,t) \in \mathbb{N}^3$ for which the $t$-umvirates are optimal (i.e.\ the values for which the Erd\H{o}s-Ko-Rado property holds). This proof is very different to the Ahlswede-Khachatrian proof, and the general technique has been used more recently to solve problems where purely combinatorial techniques (such as compressions) seem not to work, so we proceed to discuss Wilson's proof of Theorem \ref{thm:wilson} in some detail.

A natural first attempt at a spectral proof of Theorem \ref{thm:wilson} is to try to mimic Lov\'asz's spectral proof of the Erd\H{o}s-Ko-Rado theorem and to calculate the minimum and maximum eigenvalues of the adjacency matrix $A$ of the graph $K(n,k,<\!\!t)$ with vertex-set ${[n] \choose k}$, where two $k$-element sets are joined by an edge if they have intersection of size less than $t$. (Clearly, a $t$-intersecting family is precisely an independent set of vertices in this graph.) One can then apply Hoffman's bound (Theorem \ref{thm:hoffman}), and see what comes out as the upper bound on the size of an independent set in $K(n,k,<\!\!t)$. Unfortunately, this only yields the desired bound (viz., ${n-t \choose k-t}$ for all $n \geq (k-t+1)(t+1)$) in the case $t=1$. However, one can then observe that Hoffman's bound still holds if $A$ is replaced by a {\em pseudoadjacency matrix} for the graph $G$. (A {\em pseudoadjacency matrix} for a finite graph $G = (V,E)$ is a real symmetric matrix $M$ with rows and columns indexed by $V(G)$, where all the row-sums are equal and positive, and where $M_{x,y} = 0$ whenever $xy \notin E(G)$.) The proof of Hoffman's bound, given above, clearly works when the adjacency matrix $A$ is replaced by a pseudoadjacency matrix $M$ --- an observation that goes back to Delsarte --- and in fact, the graph $G$ need not be regular. We therefore immediately obtain the following.

\begin{theorem}[The Delsarte-Hoffman bound] 
\label{thm:delsarte-hoffman}
Let $G = (V,E)$ be a finite graph, let $N = |V(G)|$ and let $M$ be a pseudoadjacency matrix of $G$. Let $\lambda_1$ be the eigenvalue of $M$ corresponding to the eigenvector $\mathbf{f}: = (1,1,\ldots,1)$, and let $\lambda_{\min}$ be the minimum (i.e.\ the `most negative') eigenvalue of $M$. Let $S \subset V(G)$ be an independent set of vertices of $G$. Then
\begin{equation}
\label{eq:dhbound} |S| \leq \frac{-\lambda_{\min}}{\lambda_1-\lambda_{\min}} N.\end{equation}
Equality holds only if 
$$1_S - \frac{|S|}{N}(1,1,\ldots,1)$$
is a $\lambda_{\min}$-eigenvector of $M$, where $1_{S}$ denotes the indicator vector of $S$.
\end{theorem}

The next step (and most of the work) of Wilson's proof is to find a pseudoadjacency matrix $M$ (for the graph $K(n,k,<\!\!t)$) that has the appropriate maximum and minimum eigenvalues to imply the bound (\ref{eq:tumvirate}) (for all $n \geq (k-t+1)(t+1)$). Given the symmetries of ${[n] \choose k}$, it is natural to seek such a matrix $M$ where $M_{S,T}$ depends only upon $|S \cap T|$, and where $M_{S,T} = 0$ for all $|S \cap T| \geq t$. This condition can easily be seen to be equivalent to taking
$$M_{S,T} = \sum_{j=0}^{t-1} c_j A_j,$$
for some real coefficients $c_0,c_1,\ldots,c_{t-1}$, where $A_j$ is the adjacency matrix of the graph $K(n,k,=\!\!j)$ with vertex-set ${[n] \choose k}$, where two sets are joined by an edge if they have intersection of size exactly $j$. The $A_j$'s (which are in fact defined for all $0 \leq j \leq k$) are precisely the adjacency matrices of the {\em Johnson association scheme}, and it follows from the general theory of association schemes that they are simultaneously diagonalisable, with (common) eigenspaces $V_0,V_1,\ldots,V_k$ (defined in the previous section). The corresponding eigenvalues are also known (they were given by Delsarte in 1973 \cite{delsarte-73}).

We are left with the task of choosing the coefficients $c_0,c_1,\ldots,c_{t-1}$ so as to ensure the correct maximum and minimum eigenvalues to imply the desired upper bound. This might be a rather hard task, but for the fact that we have an important clue: equality holds in (\ref{eq:dhbound}) only if the `shifted characteristic vector' $1_S - \frac{|S|}{|V(G)|}(1,1,\ldots,1)$ is a $\lambda_{\min}$-eigenvector of $M$. We know that equality must hold in (\ref{eq:tumvirate}) whenever $\f$ is of the form $\{S \in {[n] \choose k}:\ B \subset S\}$ for some $B \in {[n] \choose t}$, and it is easy to check that the shifted characteristic vectors of these families span $V_1 \oplus V_2 \oplus \ldots \oplus V_{t-1}$. Hence, if $\lambda_i$ is the eigenvalue of $M$ corresponding to $V_i$ (for each $i \in [k] \cup \{0\}$), then the $\lambda_i$ must satisfy
$$\frac{-\lambda_i}{\lambda_0-\lambda_i} = \frac{{n-t \choose k-t}}{{n \choose k}} \quad \forall 1 \leq i \leq t-1.$$
These conditions determine $c_0,c_1,\ldots,c_{t-1}$ up to a constant factor (and rescaling all the $c_i$ by a constant factor simply rescales all the $\lambda_i$ by a constant factor, which makes no difference to the Delsarte-Hoffman bound). It turns out that, under this choice, the eigenvalues $\lambda_t,\lambda_{t+1},\ldots,\lambda_k$ are all at least $\lambda_{t-1} = \ldots = \lambda_1$ (provided $n \geq (k-t+1)(t+1)$); this is checked by Wilson, though it involves some non-trivial calculation. Hence, this choice of $M$ yields Wilson's theorem.

In fact, the $t$-intersection question turns out to be non-trivial even in the `easier' setting of subfamilies of $\mathcal{P}([n])$. This problem was resolved by Katona \cite{katona-t} in 1964.
\begin{theorem}[Katona, 1964.]
\label{thm:katona}
Let $t \in \mathbb{N}$. If $\mathcal{F} \subset \mathcal{P}([n])$ is $t$-intersecting, then
$$|\mathcal{F}| \leq \begin{cases} \sum_{i \geq (n+t)/2}{n \choose i} & \text{ if } n+t \text{ is even;}\\ {n-1 \choose (n+t-1)/2}+\sum_{i \geq (n+t+1)/2} {n \choose i}& \text{ if }n+t \text{ is odd.}\end{cases}$$
\end{theorem}

Theorem \ref{thm:katona} says, for example, that when $n+t$ is even, the $t$-intersecting family $\{S \subset [n]:\ |S| \geq (n+t)/2\}$ `wins'. Note that this family is far from being a junta; moreover, it is easy to check that any winning family (in fact, any family with size `close' to the maximum) is `far' from being a junta, in a well-defined sense. In particular, the Erd\H{o}s-Ko-Rado property certainly does not hold, here.

As observed by Wang \cite{wang}, Theorem \ref{thm:katona} can be given a short proof using $ij$-compressions, which were introduced in the previous section. We sketch Wang's proof from \cite{wang}. We may assume, in proving Theorem \ref{thm:katona}, that $\mathcal{F}$ is left-compressed. We claim that if $\mathcal{F}$ is left-compressed and $t$-intersecting, then
$$\mathcal{A}: = \{S \subset [n] \setminus \{1\}:\ S \cup \{1\} \in \mathcal{F}\} \subset \mathcal{P}([n] \setminus \{1\})$$
is $(t-1)$-intersecting, and
$$\mathcal{B}: = \{S \subset [n] \setminus \{1\}:\ S \in \mathcal{F}\} \subset \mathcal{P}([n] \setminus \{1\})$$
is $(t+1)$-intersecting. The first statement is clear, so let us prove the second. Let $B_1,B_2 \in \mathcal{B}$. Choose $j \in B_1 \cap B_2$. Since $\mathcal{F}$ is $1j$-compressed, we have $(B_1 \cup \{1\}) \setminus \{j\} \in \mathcal{F}$, so using the fact that $\mathcal{F}$ is $t$-intersecting, we have $|B_1 \cap B_2| = |((B_1 \cup \{1\}) \setminus \{j\}) \cap B_2|+1 \geq t+1$, as required.

Given the above claim, it is straightforward to prove Theorem \ref{thm:katona} by induction on $n$. Indeed, the theorem is easy for $t=1$ (see Question \ref{question:power-set}), and easy to check for $n=1$, so we may assume that $t \geq 2$. Let $n \geq 2$, and assume by induction that the theorem holds when $n$ is replaced by $n-1$. Let $\mathcal{F} \subset \mathcal{P}([n])$ be $t$-intersecting; then $|\mathcal{F}| = |\mathcal{A}|+|\mathcal{B}|$, and applying the inductive hypothesis to $\mathcal{A}$ and to $\mathcal{B}$, with $t-1$ and $t+1$ in place of $t$, respectively, yields the bound in the theorem.

\subsection{Forbidding just one intersection}

There is another natural way of generalising Question \ref{question:ekr}, which is to forbid just one intersection-size.
\begin{qn}
For each triple of non-negative integers $(n,k,\ell)$, what is the maximum possible size $N(n,k,\ell)$ of a family of $k$-element subsets of $[n]$ such that no two sets in the family have intersection of size $\ell$?
\end{qn}
When $\ell=0$, the answer is given by the Erd\H{o}s-Ko-Rado theorem, but for many values of $(n,k,\ell)$, the exact answer remains elusive. Some remarkable results are known, however, on this problem.
\newline

\subsubsection*{The case when $n$ is large compared to $\ell$} In 1985, Frankl and F\"uredi \cite{ff} proved that $N(n,k,\ell) = {n-\ell-1 \choose k-\ell-1}$ for all $k \geq 2\ell+2$ and all $n \geq n_0(k)$; the extremal families in this case are precisely those of the form $\{S \in {[n] \choose k}:\ T \subset S\}$, for $T \in {[n] \choose \ell+1}$. Their proof relies on an intricate application of the `delta-system method', which was first developed and used by Frankl in the 1970s (see e.g.\ \cite{frankl-77}, or \cite{frankl-78-int} for a particularly transparent application), and first described in full by Deza, Erd\H{o}s and Frankl in 1978 \cite{def}.

Recently, two results were obtained (one by Keller and Lifshitz \cite{keller-lifshitz}, another by Keller, Lifshitz and the author \cite{ekl-entropy}) which together imply that $N(n,k,\ell) = M(n,k,\ell+1)$ for all $k \geq 2\ell+2$ and all $n \geq n_0(\ell)$ (see Question \ref{question:ak} for the definition of $M(n,k,t)$). The proofs of these two results use the `junta method', together with appropriate notions of pseudorandomness/regularity. We will discuss these methods, in a slightly easier context, in Section \ref{subsection:perms}.

\subsection*{The Frankl-Wilson theorem(s)}

The fact that $N(n,k,\ell) = M(n,k,\ell+1)$ for all $k \geq 2\ell+2$ and all $n$ sufficiently large depending on $\ell$, may be somewhat surprising: it says that for $n$ large, banning intersections of size exactly $\ell$ has the same effect (on the maximum size of a $k$-uniform family) as banning intersections of sizes $0,1,2,\ldots,\ell-1$ and $\ell$. However, the following theorem of Frankl and Wilson is (arguably) even more surprising (indeed, spectacular).

\begin{theorem}[Frankl-Wilson \cite{fw81}, 1981]
\label{thm:fw81}
If $p$ is prime, then
$$N(4p,2p,p) \leq 2{4p \choose p-1}.$$
\end{theorem}
We note that ${4p \choose p-1} < e^{-p/4} \cdot 2^{4p}$ (using e.g.\ a standard Chernoff bound such as (\ref{eq:chernoff}), below), so the bound in Theorem \ref{thm:fw81} is exponentially small as a fraction of ${4p \choose 2p}$ (recall that ${4p \choose 2p} \gtrsim 2^{4p}/\sqrt{p}$). Theorem \ref{thm:fw81} is a very remarkable result: it says that just banning one intersection-size can force a family to be exponentially small. It is a consequence of the following theorem, also from \cite{fw81}.

\begin{theorem}[Frankl-Wilson, 1981]
\label{thm:fw81gen}
Let $p$ be a prime, and let $k,n \in \mathbb{N}$ with $k \leq n$. Let $\mathcal{F} \subset {[n] \choose k}$, and let $\lambda_1,\ldots,\lambda_s \in [p-1] \cup \{0\}$ with $\lambda_i \not\equiv k \mod p$ for all $i \in [s]$. Suppose that for any two distinct sets $S,T \in \mathcal{F}$, there exists $i \in [s]$ such that $|S \cap T| \equiv \lambda_i \mod p$. Then
$$|\f| \leq {n \choose s}.$$
\end{theorem}

To deduce Theorem \ref{thm:fw81} from Theorem \ref{thm:fw81gen}, we let $\mathcal{F} \subset {[4p] \choose 2p}$ with $|S \cap T| \neq p$ for all distinct $S,T \in \f$. For each pair $\{S,[4p] \setminus S\} \subset \f$, we remove one of $S$ and $[4p] \setminus S$ from $\f$, producing a family $\f'$ that satisfies $|S \cap T| \not\equiv 0$ (mod.\ $p$) for all distinct $S,T \in \f'$. By applying Theorem \ref{thm:fw81gen} to $\f'$ with $n=4p$, $k=2p$, $s=p-1$ and $\{\lambda_1,\ldots,\lambda_s\} = [p-1]$, we see that $|\f'| \leq {4p \choose p-1}$, and therefore $|\f|\leq 2|\f'| \leq 2{4p \choose p-1}$, as required.

\subsubsection*{The `linear independence method' in extremal combinatorics}
Frankl and Wilson's proof of Theorem \ref{thm:fw81gen} uses a rather ingenious variant of what may be called the `linear algebraic method' or `linear independence method' in extremal combinatorics. The general idea behind this method is as follows. Suppose $U$ is a universe of mathematical objects, and $P$ is a property of subsets of $U$, and we wish to show that any subset $\mathcal{S} \subset U$ satisfying $P$, has $|\mathcal{S}| \leq M$. One possible way of doing this is to associate, to each element $u$ of $U$, a vector $\phi(u)$ in some vector space $V$ of dimension at most $M$, and to show that, if $\mathcal{S} \subset U$ satisfies the property $P$, then $\phi(u_1) \neq \phi(u_2)$ for all distinct $u_1,u_2 \in S$ and $\{\phi(u):\ u \in \mathcal{S}\}$ is a linearly independent set; it then follows immediately that $|\mathcal{S}| = |\phi(\mathcal{S})| \leq M$.

This general method is more easily illustrated with the following slight generalisation of a (very) special case of Theorem \ref{thm:fw81gen}, due to Berlekamp \cite{berkelamp}.

\begin{prop}
\label{prop:oddtown}[The `oddtown' theorem]
Let $\f \subset \mathcal{P}([n])$ such that $|S|$ is odd for all $S \in \f$, and $|S \cap T|$ is even for all distinct $S,T \in \f$. Then $|\f| \leq n$.
\end{prop}
\begin{proof}
For each $S \in \f$, consider the characteristic vector $\chi_S \in \mathbb{F}_2^n$, defined by $\chi_S(i) = 1$ if $i \in S$ and $\chi_S(i)=0$ if $i \notin S$. Let
$$\langle u,v \rangle := \sum_{i=1}^{n} u_i v_i\quad (u,v \in \mathbb{F}_2^n)$$
denote the standard inner product in $\mathbb{F}_2^n$. Since $|S|$ is odd for all $S \in \f$, we have $\langle \chi_S,\chi_S \rangle = 1$ for all $S \in \f$, and since $|S \cap T|$ is even for all distinct $S,T \in \f$, we have $\langle \chi_S,\chi_T \rangle = 0$ for all distinct $S,T \in \f$. Hence, $\{\chi_S:\ S \in \f\}$ is orthonormal, and therefore linearly independent in $\mathbb{F}_2^n$, a vector space of dimension $n$. The proposition follows.
\end{proof}

We remark that the bound in Proposition \ref{prop:oddtown} is very strong indeed: the condition that $|S \cap T|$ is even for all distinct $S,T \in \f$ forces a polynomial (indeed, linear) upper bound, compared to the exponential $|\{S \subset [n]:\ |S| \text{ is odd}\}|=2^{n-1}$. Proposition \ref{prop:oddtown} is clearly best possible, as is evidenced by the family of singletons $\{\{i\}:\ i \in [n]\}$.

Results such as Theorem \ref{thm:fw81gen} and Proposition \ref{prop:oddtown}, supplying upper bounds on the size of a family of sets where intersection-sizes are specified modulo $m$ for some integer $m$, are often referred to as `oddtown'-type results. A wide variety of such results have been proven over the last 40 years; the reader is referred e.g.\ to \cite{babai-frankl,ly} for a sample of such results. For a survey of a great many other applications of what we have called the `linear independence method', in combinatorics and combinatorial geometry, the reader is referred to the book of Babai and Frankl \cite{babai-frankl}.

\subsection{Geometric consequences of the Frankl-Wilson theorem(s)}

We now discuss three beautiful and important geometric consequences of the aforementioned theorems of Frankl and Wilson.

\subsubsection*{The Kahn-Kalai counterexample to Borsuk's conjecture} Theorem \ref{thm:fw81} was used by Kahn and Kalai in 1993 \cite{kahn-kalai-borsuk} to disprove (in an extremely strong sense), a hitherto widely believed conjecture of Borsuk in geometry, namely,
\begin{conj}[Borsuk]
Any bounded subset $B \subset \mathbb{R}^n$ can be partitioned into $n+1$ sets, each of which has strictly smaller diameter than $B$.
\end{conj}
This is true for $n \leq 3$; the regular simplex in $\mathbb{R}^n$ shows that, if Borsuk's conjecture were true, then it would be best-possible. Unfortunately, it is wildly false in general.

We call a decomposition of a bounded subset $B \subset \mathbb{R}^n$ into sets of strictly smaller diameter than $B$, a {\em Borsuk decomposition} of $B$. Kahn and Kalai exhibited, for each $n \in \mathbb{N}$, a (discrete) subset $B_n$ of $\mathbb{R}^n$ requiring at least $c^{\sqrt{n}}$ sets in a Borsuk decomposition, where $c>1$ is an absolute constant. In fact, $B_n \subset \{0,1\}^n$ for each $n$. To explain their construction, we identify $\{0,1\}^n$ with $\mathcal{P}([n])$ in the natural way, and we assume that $n={4p \choose 2}$ where $p$ is prime, so that we can identify $\mathcal{P}([n])$ with the set of all edge-sets of labelled graphs on the vertex-set $[4p]$. For each subset $S \in {4p \choose 2p}$, we consider the graph $G_S$ which is the complete bipartite graph with vertex-classes $S$ and $[4p] \setminus S$. We take $B_n = \{G_S:\ S \in {4p \choose 2p}\}$. Note that for any $S,T \in {4p \choose 2p}$, we have
\begin{align*} |E(G_S) \Delta E(G_T)| & = |E(G_S)|+|E(G_T)| - 2|E(G_S) \cap E(G_T)|\\
&= 8p^2 - 2(|S \cap T|^2 + (2p-|S \cap T|)^2),
\end{align*}
which is maximized when $|S \cap T| = p$. Hence, any subset of $B_n$ with diameter strictly smaller than $B_n$, must not contain $G_S$ and $G_T$ with $|S \cap T| = p$, and therefore has cardinality at most $2{4p \choose p-1}$, by Theorem \ref{thm:fw81}, so the number of pieces needed in a Borsuk decomposition of $B_n$ is at least
$$\frac{\tfrac{1}{2}{4p \choose 2p}}{2 {4p \choose p-1}} \gtrsim e^{p/4}/\sqrt{p} \geq c^{\sqrt{n}},$$
where $c>1$ is an absolute constant.

Writing $b(n)$ for the maximum possible number of sets required in a Borsuk decomposition of a bounded subset of $\mathbb{R}^n$, it is known that $b(n) \leq C^n$ for all $n \in \mathbb{N}$, where $C$ is an absolute constant (this upper bound in due to Schramm \cite{schramm}). The best-known lower bound on $b(n)$ is still of the form $c^{\sqrt{n}}$.

\subsubsection*{The chromatic number of Euclidean space}

A classical problem is to determine (or bound) the chromatic number of the unit distance graph on $\mathbb{R}^n$: that is, the graph with vertex-set $\mathbb{R}^n$, where two points are joined if they are at a (Euclidean) distance of one apart; we denote this chromatic number by $\chi(\mathbb{R}^n)$. Even the determination of $\chi(\mathbb{R}^2)$ is a notorious open problem, known as the Hadwiger-Nelson problem; for more than 50 years the best known bounds were $4 \leq \chi(\mathbb{R}^2) \leq 7$ (the lower bound being due to the brothers William and Leo Moser, the upper bound being due to Isbell), until the 2018 breakthrough of Aubrey de Grey \cite{grey} showing that $\chi(\mathbb{R}^2) \geq 5$. In \cite{fw81}, Frankl and Wilson applied Theorem \ref{thm:fw81} to show that $\chi(\mathbb{R}^n) \geq c^n$ for all $n \in \mathbb{N}$, where $c>1$ is an absolute constant. To see this, assume that $n=4p$ for a prime $p$, and embed ${4p \choose 2p} \subset \mathcal{P}([4p])$ into $\{0,1\}^{4p} \subset \mathbb{R}^{4p}$ in the natural way; then two sets $S,T \in {4p \choose 2p}$ with $|S \cap T| = p$ correspond exactly to two points $x,y \in \{0,1\}^{4p}$ with $\|x-y\|_2 = \sqrt{2p}$. Rescaling by a factor of $\sqrt{2p}$, we see that a monochromatic subset of $(1/\sqrt{2p})\cdot \{0,1\}^{4p}$ (in a proper colouring of $\mathbb{R}^{4p}$) has size at most $2{4p \choose p-1}$, so
$$\chi(\mathbb{R}^{4p}) \geq \frac{2^{4p}}{2{4p \choose p-1}} \geq \frac{2^{4p}}{2e^{-p/4} \cdot 2^{4p}} = \tfrac{1}{2}e^{p/4}.$$
Using the classical fact (Bertrand's postulate) that for every $x>1$, there exists a prime $p$ with $x \leq p \leq 2x$, it follows that $\chi(\mathbb{R}^n) > c^n$ for each $n \in \mathbb{N}$, where $c>1$ is an absolute constant.

In the other direction, it is easy to see, by partitioning $\mathbb{R}^n$ into half-open cubes of side-length $1/\sqrt{n}$ and greedily colouring the cubes (giving each cube one colour), that $\chi(\mathbb{R}^n) \leq C^n$ where $C$ is an(other) absolute constant, so the rate of growth of $\chi(\mathbb{R}^n)$ is exponential. Its precise (exponential) rate of growth is unknown. The best-known general lower bound is $\chi(\mathbb{R}^n)\geq (1.239+o(1))^n$, due to Raigordskii \cite{raigordskii}, and the best-known general upper bound is $\chi(\mathbb{R}^n) \leq (3+o(1))^n$, due to Larman and Rogers \cite{larman-rogers}.

\subsubsection*{Witsenhausen's problem on subsets of the sphere}
For each $n \in \mathbb{N}$, let $\mu_n$ denote Lebesgue measure on the $n$-dimensional unit sphere $S^n \subset \mathbb{R}^{n+1}$. An old problem of Witsenhausen is to determine the maximum measure $m(n)$ of a Lebesgue-measureable subset $\f \subset S^{n}$ such that $\langle x,y \rangle \neq 0$ for all $x,y \in \f$. Kalai conjectures that the extremal sets for this problem are the double spherical caps of half-angle $\pi/3$, i.e.\ the sets of the form $\{x \in S^n:\ |\langle x,x_0 \rangle| > 1/2\}$ for $x_0 \in S^n$ (this beautiful conjecture remains open). In \cite{fw81}, Frankl and Wilson applied Theorem \ref{thm:fw81} to show that $m(n) \leq c^n$ for each $n \in \mathbb{N}$, for some absolute constant $c<1$. To see this, take $n=4p-1$ where $p$ is prime, and observe that, embedding ${[4p] \choose 2p} \subset \mathcal{P}([4p])$ into $\{\pm 1\}^{4p}$ in the natural way, two sets $S,T \in {[4p] \choose 2p}$ with $|S \cap T| = p$ correspond exactly to two points $x,y \in \{\pm1\}^{4p}$ with $\langle x,y \rangle = 0$. Rescaling, we have a copy $\mathcal{C}$ of ${[4p] \choose 2p}$ inside $(1/(2\sqrt{p}))\cdot \{\pm 1\}^n \subset S^{4p-1}$ where, again, two sets $S,T \in {[4p] \choose 2p}$ with $|x \cap y| = p$ correspond exactly to two points $x,y \in (1/(2\sqrt{p}))\cdot\{\pm1\}^{4p}$ with $\langle x,y \rangle = 0$. If $\f \subset S^{4p-1}$ is measurable with $\langle x,y \rangle \neq 0$ for all $x,y \in \f$, then by Theorem \ref{thm:fw81}, we have 
$$\frac{|\f \cap \mathcal{C}|}{|\mathcal{C}|} \leq \frac{2{4p \choose p-1}}{{4p \choose 2p}} \leq \frac{2 e^{-p/4} \cdot 2^{4p}}{\frac{1}{4p} \cdot 2^{4p}} \leq c^{4p-1},$$
where $c<1$ is an absolute constant. The same bound clearly holds if $\mathcal{C}$ is replaced by $\sigma(\mathcal{C})$, where $\sigma \in O(4p,\mathbb{R})$; so averaging over all images of $\mathcal{C}$ under elements of $O(4p,\mathbb{R})$, proves that $\mu_{4p-1}(\f) \leq c^{4p-1}$. (Here, as usual, for $N \in \mathbb{N}$, we denote by $O(N,\mathbb{R})$ the group of $N$ by $N$ orthogonal matrices with real entries.)

\subsection{Forbidding just one intersection, in the non-uniform setting}

In contrast to Question \ref{question:ak}, which is easier in the setting of $\mathcal{P}([n])$ (the `non-uniform setting) than in the case of ${[n] \choose k}$ (the `uniform setting'), the forbidden intersection problem appears at least as difficult in the setting of $\mathcal{P}([n])$ as in the case of ${[n] \choose k}$. For $n,\ell \in \mathbb{N}$, we write
$$N(n,\ell) = \max\{|\f|:\ \f \subset \mathcal{P}([n]),\ |S \cap T| \neq \ell\ \text{for all distinct }S,T \in \f\}.$$
Solving a $\$250$ problem of Erd\H{o}s from 1976, Frankl and R\"odl proved in a breakthrough 1987 paper \cite{fr87} that $N(n,\lfloor n/4 \rfloor) \leq 1.99^n$ for all $n \in \mathbb{N}$. This was deduced from the following theorem (together with a calculation of the appropriate constants).

\begin{theorem}[Frankl-R\"odl, 1987]
\label{thm:fronefamily}
For each $0 < \eta < 1/4$, there exists $\epsilon = \epsilon(\eta)>0$ such that for any $\ell \in \mathbb{N}$ with $\eta n \leq \ell \leq (1/2-\eta)n$, if $\f \subset \mathcal{P}([n])$ with $|S \cap T| \neq \ell$ for all distinct $S,T \in \f$, then $|\f| \leq (2-\epsilon)^n$.
\end{theorem}

This in turn was deduced from a two-family version.
\begin{theorem}[Frankl-R\"odl, 1987]
\label{thm:frtwofamily}
For each $0 < \eta < 1/4$, there exists $\epsilon = \epsilon(\eta)>0$ such that for any $\ell \in \mathbb{N}$ with $\eta n \leq \ell \leq (1/2-\eta)n$, if $\f,\g \subset \mathcal{P}([n])$ with $|S \cap T| \neq \ell$ for all $S \in \f$ and $T \in \g$, then $|\f||\g| \leq (4-\epsilon)^n$.
\end{theorem}

Frankl and R\"odl's proof (in \cite{fr87}) of Theorem \ref{thm:frtwofamily} is ingenious, and purely combinatorial, using a density increment argument. We believe it deserves to be more widely known and better understood, particularly as density increment methods (and similar increment methods --- incrementing another parameter such as `entropy' or `energy', appropriately defined) have been very successfully used in combinatorics and other areas of mathematics, over the last 30 years. So we proceed to give a (very detailed) sketch of the Frankl-R\"odl proof.

Frankl and R\"odl begin with two observations that follow from Harper's vertex-isoperimetric inequality (see \cite{harper}) for the discrete cube. Harper's theorem easily implies the following.
\begin{theorem}
\label{thm:harper}
Let $\mathcal{A} \subset \mathcal{P}([n])$ with $|\mathcal{A}| \geq \sum_{i=0}^{a} {n \choose i}$. Then $|N_t(\mathcal{A})| \geq \sum_{i=0}^{a+t} {n \choose i}$.
\end{theorem}
Here, $N_t(\mathcal{A})$ denotes the {\em $t$-neighbourhood} of $\mathcal{A}$, i.e.\ family of sets which are at Hamming distance at most $t$ from $\mathcal{A}$.

The first observation of Frankl and R\"odl is as follows.
\begin{lemma}
\label{lem:fr1}
Let $0 < \beta < 1$. Let $\f,\g \subset \pn$ such that $|F \cap G| > \beta n$ for all $F \in \f$ and $G \in \g$. Then
$$|\f||\g| \leq 2^{2nH_2((1+\beta)/2)},$$
where $H_2(p): = p \log_2(1/p)+(1-p)\log_2(1/(1-p))$ denotes the binary entropy function.
\end{lemma}
Note that for $\beta >0$, the quantity $H_2((1+\beta)/2)$ is bounded away from 1, so the upper bound in Lemma \ref{lem:fr1} is exponentially small compared to $2^{2n}$.

To prove Lemma \ref{lem:fr1}, assume without loss of generality that $|\f| \leq |\g|$. Then choose $a \in \mathbb{N}$ such that
$$\sum_{i=0}^{a} {n \choose i} \geq |\f| > \sum_{i=0}^{a-1}{n \choose i}.$$
Since $F \cap G \neq \emptyset$ for all $F \in \f$ and $G \in \g$, we have $(F \in \f) \Rightarrow ([n] \setminus F \notin \g)$, and therefore $|\f| + |\g| \leq 2^n$, so $|\f| \leq 2^{n-1}$; it follows that $a \leq n/2$. Let $t \in \mathbb{N}$ be maximal such that $|F \cap G| \geq t$ for all $F \in \f$ and $G \in \g$; note that $t > \beta n$. It follows from Theorem \ref{thm:harper} that
$$|N_t(\f)| \geq \sum_{i=0}^{a+t-1}{n \choose i}.$$
Let $\overline{\g}: = \{[n] \setminus G:\ G \in \g\}$. Since $|F \cap G| \geq t$ for all $F \in \f$ and $G \in \g$, we must have $N_t(\f) \cap \overline{\g} = \emptyset$, and therefore $|\g| = |\overline{\g}| \leq \sum_{i=a+t}^{n} {n \choose i}$. Hence,
$$|\f||\g| \leq \left( \sum_{i=0}^{a} {n \choose i}\right)\cdot \left( \sum_{i=a+t}^{n} {n \choose i}\right).$$
Maximising over the choice of $a$, and using the Chernoff bounds
$$\sum_{i=0}^{a}{n \choose i} \leq 2^{H_2(a/n)n} \quad \forall a \leq n/2,\quad \sum_{i=a+t}^{n}{n \choose i} \leq 2^{H_2((a+t)/n)n}\quad \text{for } a+t \geq n/2$$
and using the trivial bound
$$\sum_{i=a+t}^{n} {n \choose i} \leq 2^n$$
in the case where $a+t < n/2$, the conclusion of the lemma follows.

The second observation of Frankl and R\"odl is an easy consequence of the first.
\begin{lemma}
\label{lem:fr2}
Let $0 < \kappa < 1/2$. Let $\f,\g \subset \pn$ such that $|F \cap G| < (1/2-\kappa)n$ for all $F \in \f,\ G \in \g$. Then for any $0 < \lambda < \kappa$, we have
$$|\f||\g| \leq \max\{2\cdot 2^{n(1+H_2(1/2-\lambda))},2\cdot 2^{2nH_2((1+\kappa-\lambda)/2)}\}.$$
\end{lemma}
Note that, as with the previous lemma, both $H_2(1/2-\lambda)$ and $H_2((1+\kappa-\lambda)/2)$ are bounded away from 1 for any $0 < \lambda < \kappa$, so the upper bound in Lemma \ref{lem:fr2} is exponentially small compared to $2^{2n}$.

Lemma \ref{lem:fr2} may be proved as follows. We let
$$\f_s : = \{F \in \f:\ |F| \leq (1/2-\lambda)n\},\quad \f_l: = \f \setminus \f_s,$$
$s$ standing for `small' and $l$ for `large'. Note that $|\f_s| \leq 2^{H_2(1/2-\lambda)n}$, so if $|\f_s| \geq |\f|/2$ then we are done. We may assume, therefore, that $|\f_l| \geq |\f|/2$. Letting $\overline{\g}: = \{[n] \setminus G:\ G \in \g\}$, we have
$$|F\cap H| = |F \cap ([n] \setminus G)| = |F| - |F \cap G| > (1/2-\lambda)n-(1/2-\kappa)n = (\kappa-\lambda)n$$
for all $F \in \f_l$ and all $H \in \overline{\g}$ (here, $G \in \g$), so applying Lemma \ref{lem:fr1} to $\f_l$ and $\overline{\g}$, yields
$$|\f_l||\overline{\g}| \leq 2^{2nH_2((1+\kappa-\lambda)/2)},$$
and therefore
$$|\f||\g| \leq 2\cdot 2^{2nH_2((1+\kappa-\lambda)/2)},$$
proving the lemma.

Equipped with the two preceding lemmas, the Frankl-R\"odl proof proceeds as follows. Roughly speaking, the idea is to show that we can pass to smaller copies of $\mathcal{P}([n])$ in such a way that we either obtain a `large' density increment on the copies, or else we `widen the interval of forbidden intersections' (while approximately preserving the density), in such a way as to reduce either to the case covered by Lemma \ref{lem:fr1} or to that covered by Lemma \ref{lem:fr2}.

To make this precise, for families $\f,\g \subset \pn$ and integers $0 \leq a \leq b \leq n$, we write $(\f,\g) \in \mathcal{P}(n,[a,b])$ if $|F \cap G| \notin [a,b]$ for all $F \in \f$ and $G \in \g$, i.e.\ if intersections in the interval $[a,b]$ are forbidden. Theorem \ref{thm:frtwofamily} (our goal) deals with the case where $a=b$; Lemma \ref{lem:fr1} deals with the case where $a=0$ and $b/n$ is bounded away from zero, and Lemma \ref{lem:fr2} deals with the case where $b=n$ and $a/n$ is bounded from above, away from $1/2$.

The idea of `passing to copies' rests on the following. For $\mathcal{A} \subset \pn$, we define
$$\mathcal{A}_0 = \{S \in \mathcal{A}:\ n \notin \mathcal{A}\} \subset \mathcal{P}([n-1]),\quad \mathcal{A}_1 = \{S \setminus \{n\}:\ n \in S,\ S \in \mathcal{A}\} \subset \mathcal{P}([n-1]).$$
We observe that if $(\f,\g) \in \mathcal{P}(n,[a,b])$, then
\begin{enumerate}
\item[(i)] $(\f_1,\g_1) \in \mathcal{P}(n-1,[a-1,b-1])$;
\item[(ii)] $(\f_0,\g_0 \cup \g_1) \in \mathcal{P}(n-1,[a,b])$;
\item[(iii)] $(\f_1,\g_0 \cap \g_1) \in \mathcal{P}(n-1,[a-1,b])$.
\end{enumerate}
The observations (i) and (ii) will enable us to achieve a density increment within $\mathcal{P}([n-1])$, while preserving the width of the interval of forbidden intersections; when this is not possible, the observation (iii) will enable us to widen the interval of forbidden intersections, while approximately preserving the density.

To keep track of the (correct) densities, for $m \in \mathbb{N}$ and for a family $\mathcal{A} \subset \mathcal{P}([m])$, we write $\mu(\mathcal{A}) = |\mathcal{A}|/2^m$; in other words, $\mu = \mu_m$ denotes the uniform measure on $\mathcal{P}([m])$ (though we suppress $m$ from the notation, as it will be clear from the context).

The proof of Theorem \ref{thm:frtwofamily} is accomplished by the following algorithm.

\begin{enumerate}
\item Set $m=n$, $a=\ell$, $b = \ell$, and fix $\delta = \delta(\eta)>0$ a sufficiently small positive real number (with $\delta \leq 1/10)$.
\item Check whether $a=0$. If yes, terminate; if not, go to (3).
\item Check whether $b=m$. If yes, terminate; if not, go to (4).
\item Check whether $\mu(\f_1)\mu(\g_1) > (1+\delta)\mu(\f)\mu(\g)$. If yes, replace $\f$ by $\f_1$, replace $\g$ by $\g_1$, replace $a$ by $a-1$ and replace $b$ by $b-1$, and go to (8); if not, go to (5).
\item Choose $\f_1$ or $\g_1$ (say $\f_1$) with $\mu(\f_1) \leq \sqrt{1+\delta}\mu(\f)$, and go to (6).
\item Check whether $\mu(\f_0)\mu(\g_0 \cup \g_1)  > (1+\delta) \mu(\g) \mu(\f)$. If yes, replace $\f$ with $\f_0$ and $\g$ by $\g_0 \cup \g_1$, and go to (8); if not, go to (7).
\item Replace $\f$ by $\f_1$, $\g$ by $\g_0 \cap \g_1$ and $a$ by $a-1$, and go to (8).
\item Replace $m$ by $m-1$ and go to (2).
\end{enumerate}

The key observation is that if at some iteration of (steps (2)-(8) of) the algorithm, we have a pair of families $(\f,\g) \in \mathcal{P}(m,[a,b])$ at the start of step (2), then by step (8), they have either been replaced by
\begin{align}
\label{eq:density-inc}
\text{a pair of families }&(\f',\g') \in \mathcal{P}(m-1,[a,b]) \cup \mathcal{P}(m-1,[a-1,b-1])\nonumber \\
\text{with }&\mu(\f')\mu(\g') > (1+\delta)\mu(\f)\mu(\g),\end{align}
or by
\begin{align} \label{eq:interval-widening} \text{a pair of families }&(\f',\g') \in \mathcal{P}(m-1,[a-1,b]) \nonumber \\
\text{with }&\mu(\f')\mu(\g') > (1-\delta-2\delta^2)\mu(\f)\mu(\g).\end{align}
In the first case, we have a density increment; in the second case, we have widened the interval of forbidden intersections while approximately preserving the density. To prove this, observe that we have the required density increment unless we are directed (at step (6)) to go to step (7). In the latter case, just before applying step (7), we have (w.l.o.g.)
$$\mu(\f_1) \leq \sqrt{1+\delta} \mu(\f),\quad \mu(\f_0)\mu(\g_0 \cup \g_1) \leq (1+\delta)\mu(\f)\mu(\g).$$
It is easy to check from these inequalities that $\mu(\f_1) \mu(\g_0 \cap \g_1) \geq (1-\delta-2\delta^2)\mu(\f)\mu(\g)$. Indeed, write
$$\frac{\mu(\f_1)}{\mu(\f)} = 1+y,\quad \frac{\mu(\g_0 \cup \g_1)}{\mu(\g)} = 1+x,$$
where $x,y \in \mathbb{R}$; note that $x \geq 0$. Since
$$\frac{\mu(\f_0)}{\mu(\f)}+\frac{\mu(\f_1)}{\mu(\f)} = 2,$$
we have $\mu(\f_0)/\mu(\f) = 1-y$; since $\mu(\g_0 \cap \g_1) + \mu(\g_0 \cup \g_1) = 2\mu(\g)$, we have $\mu(\g_0 \cap \g_1)/\mu(\g) = 1-x$. Suppose first that $y \leq 0$, i.e.\ that $\mu(\f_1) \leq \mu(\f)$. Since
$$(1+x)(1-y) + (1-x)(1+y) = 2-2xy \geq 2,$$
we then have
\begin{align*} \frac{\mu(\f_1)}{\mu(\f)} \cdot \frac{\mu(\g_0 \cap \g_1)}{\mu(\g)} & = (1-x)(1+y)\\
& \geq 2 - (1+x)(1-y)\\
&= 2-\frac{\mu(\g_0\cup \g_1)}{\mu(\g)} \cdot \frac{\mu(\f_0)}{\mu(\f)}\\
& \geq 2-(1+\delta)\\
& = 1-\delta,\end{align*}
which suffices. Suppose now that $y \geq 0$. Then we have $x,y \geq 0$ with $1+y \leq \sqrt{1+\delta}$ (which implies $y < \delta/2$), and $(1+x)(1-y) \leq 1+\delta$, where $\delta \leq 1/10$; it is easy to check from this that
$$(1-x)(1+y) \geq 1-\delta-2\delta^2,$$
which again suffices, similarly to above.

We now examine what happens when the algorithm terminates. We let $\alpha n$ be the number of steps at which (\ref{eq:density-inc}) holds, and $\beta n$ be the number of steps at which (\ref{eq:interval-widening}) holds. Suppose the algorithm terminates at $m \in \mathbb{N}$, so that it runs for $n-m$ steps; then
\begin{equation}
\label{eq:total}
n-m = (\alpha+\beta)n
\end{equation}
Let $\f^{*}$ and $\g^{*}$ be the families with which the algorithm terminates. We may assume that $\mu(\f) \mu(\g) \geq (1-\delta^2)^{n}$, otherwise we are done; then
$$1 \geq \mu(\f^*)\mu(\g^*) \geq (1+\delta)^{\alpha n} (1-\delta-2\delta^2)^{\beta n} \mu(\f) \mu(\g) \geq (1+\delta)^{\alpha n} (1-\delta-2\delta^2)^{\beta n}(1-\delta^2)^n.$$
Taking logs and dividing by $n$ yields
$$\alpha \ln(1+\delta)+\beta \ln(1-\delta-2\delta^2) + \ln(1-\delta^2) \leq 0;$$
rearranging, we obtain
$$\alpha-\beta \leq \frac{\beta\ln(1/((1+\delta)(1-\delta-2\delta^2))) + \ln(1/(1-\delta^2))}{\ln(1+\delta)}.$$
Using the inequalities $x \geq \ln(1+x) \geq x-x^2/2$ (for all $x \geq 0$), this implies
$$\alpha-\beta \leq \left( \beta\frac{3 +2\delta}{(1+\delta)(1-\delta-2\delta^2)(1-\delta/2)} + \frac{1}{(1-\delta^2)(1-\delta/2)} \right)\delta,$$
which (together with $\alpha \leq 1/2$ and $\delta \leq 1/10$) implies that
\begin{equation}
\label{eq:abrel}
\alpha - \beta \leq 3\delta.
\end{equation}
In other words, the number of density incrementing steps cannot be too much greater than the number of interval-widening steps.

First suppose the algorithm terminates with $a=0$, so that $|F \cap G| > b$ for all $F \in \f^*$ and $G \in \g^*$. Since the width of the interval of forbidden intersections increases by one at each step where (\ref{eq:interval-widening}) holds and remains the same at each step where (\ref{eq:density-inc}) holds, we must have $b = \beta n$. Since the algorithm starts with $a=b=\ell$, and $a$ decreases by at most one at each step, we must clearly have $(\alpha+\beta)n \geq \ell \geq \eta n$, so $\alpha+\beta \geq \eta$. Combining this fact with $\alpha-\beta \leq 3\delta$ implies that
$$\beta \geq \eta/2-3\delta/2.$$
Hence, we have $b \geq (\eta/2-3\delta/2)n$. Consider now the families
\begin{align*} \f^{\dagger} &: = \{F \cup S:\ F \in \f^*,\ S \subset \{m+1,\ldots,n\}\} \subset \pn,\\
 \g^{\dagger}& : = \{G \cup S:\ G \in \g^*,\ S \subset \{m+1,\ldots,n\}\} \subset \pn.
 \end{align*}
We clearly have $|F \cap G| > b$ for all $F \in \f^{\dagger}$ and $G \in \g^{\dagger}$; further $\mu(\f^{\dagger}) = \mu(\f^*)$ and $\mu(\g^{\dagger}) = \mu(\g^*)$. We apply Lemma \ref{lem:fr1} (passing from sizes of families, to measures of families) to obtain
$$\mu(\f^*)\mu(\g^*) = \mu(\f^{\dagger})\mu(\g^{\dagger}) \leq 4^{(H_2(1/2+\beta/2)-1)n} \leq 4^{(H_2(1/2+\eta/4-3\delta/4)-1)n}.$$
Very crudely, we have $\beta \leq 1$ and therefore
\begin{align*} \mu(\f) \mu(\g) & \leq \frac{\mu(\f^*)\mu(\g^*)}{(1-\delta-2\delta^2)^{\beta n}}\\
& \leq \frac{4^{(H_2(1/2+\eta/4-3\delta/4)-1)n}}{(1-\delta-2\delta^2)^{n}}\\
& \leq (1-\delta)^{2n},
\end{align*}
using the fact that
$$4^{(H_2(1/2+\eta/4-3\delta/4)-1)} \leq (1-\delta-2\delta^2)(1-\delta)^2$$
provided $\delta$ is sufficiently small depending on $\eta$ (consider the limits of both sides as $\delta \to 0$, for a fixed $\eta >0$).

The case where the algorithm terminates with $b=m$ is dealt with similarly, except that Lemma \ref{lem:fr2} is applied instead of Lemma \ref{lem:fr1}. We leave the details to the reader.

We remark that in \cite{keevash-long}, Keevash and Long show how to use the method of `dependent random choice' to deduce Theorem \ref{thm:fronefamily} from Theorem \ref{thm:fw81gen}, whose proof is purely algebraic. So there are now two very different proofs of Theorem \ref{thm:fronefamily}.

It is a `folklore' conjecture that for each $\ell, n \in \mathbb{N}$, if $\f \subset \mathcal{P}([n])$ with $|S \cap T| \neq \ell$ for all distinct $S,T \in \f$, then $\f$ is no larger than the family $\{S \subset [n]:\ |S| < \ell \text{ or } |S| > (n+\ell)/2\}$ if $n+\ell$ is odd, and no larger than the family $\{S \subset [n]:\ |S| < \ell \text{ or } |S \cap [n-1]| \geq (n+\ell)/2\}$, if $n+\ell$ is even. (This was proven for $n$ sufficiently large depending on $\ell$, by Frankl and F\"uredi \cite{ff-un}.) An approximate version of this conjecture appears e.g.\ in \cite{mr}. These conjectures would supply very sharp versions of Theorem \ref{thm:fronefamily}. Both are wide open, to the best of our knowledge.

\subsection{Forbidding a matching}
There is another natural way of weakening the intersection condition in the Erd\H{o}s-Ko-Rado theorem: what happens if we demand that among any $s+1$ sets (for some $s \geq 2$), at least two must intersect? In other words, we forbid a matching of size $s+1$. (A {\em matching of size $r$} consists of $r$ pairwise disjoint sets.) For $n,k,s \in \mathbb{N}$, we write $m(n,k,s) := \max\{|\f|:\ \f \subset {[n] \choose k},\ \f \text{ contains no matching of size }s+1\}$.

Clearly, if $n < k(s+1)$, then no $s+1$ sets of size $k$ can be pairwise disjoint, so $m(n,k,s) = {n \choose k}$ for all $n < k(s+1)$. However, for $n > k(s+1)$ the problem is non-trivial. Erd\H{o}s conjectured the following \cite{emc}.
\begin{conj}[Erd\H{o}s Matching Conjecture, 1965]
If $n,k,s \in \mathbb{N}$ with $n \geq (s+1)k$, then
$$m(n,k,s) = \max\left\{{n \choose k}-{n-s \choose k},{k(s+1) -1 \choose k}\right\}.$$
\end{conj}
\noindent The Erd\H{o}s Matching Conjecture says that one of the two families
$$\left\{S \in {[n] \choose k}:\ S \cap [s] \neq \emptyset\right\},\quad {[k(s+1)-1] \choose k}$$
must `win'. It is easy to check that ${n \choose k}-{n-s \choose k} > {k(s+1) -1 \choose k}$ whenever $n \geq (k+1)(s+1)$, i.e.\ the first of the two families above beats the second (in this range), so the Erd\H{o}s Matching Conjecture implies that 
\begin{equation}
\label{eq:large-bound}
m(n,k,s) = {n \choose k} - {n-s \choose k}
\end{equation}
whenever $n \geq (k+1)s$. The bound (\ref{eq:large-bound}) was verified by Erd\H{o}s for all $n \geq n_0(k,s)$, by Bollob\'as, Daykin and Erd\H{o}s \cite{bde} for all $n \geq 2k^3s$, by Huang, Loh and Sudakov \cite{hls} for all $n \geq 3k^2s$, by Frankl \cite{frankl-new} for all $n \geq (2s+1)k-s$, and by Frankl and Kupavskii \cite{fk} for all $n \geq \tfrac{5}{3} ks - \tfrac{2}{3}s$ (provided $s \geq s_0$, where $s_0$ is an absolute constant).

On the other hand, when $n=k(s+1)$, a simple averaging argument shows that $m(n,k,s) \leq {k(s+1)-1 \choose k}$ (see \cite{kleitman-emc}), so the second of the two families above `wins'. Frankl \cite{frankl-other-range} recently proved that $m(n,k,s) = {k(s+1)-1 \choose k}$ for all $k(s+1) \leq n \leq (k+\epsilon)(s+1)$, where $\epsilon = \epsilon(k) >0$ for each $k$. The general case remains open.

\subsection{Covering by intersecting families: Lov\'asz's proof of Kneser's conjecture, and a Boolean analogue.}
Recall that if $G = (V,E)$ is a graph, the {\em chromatic number} $\chi(G)$ of $G$ is the minimum integer $k$ such that $V(G)$ may be partitioned into $k$ independent sets. In 1955, Kneser \cite{kneser} made the following conjecture regarding the chromatic number of the Kneser graph $K_{n,k}$.
\begin{conj}[Kneser, 1955]
\label{conj:kneser}
For $k,n \in \mathbb{N}$ with $k \leq n/2$, we have $\chi(K_{n,k}) = n-2k+2$.
\end{conj}
Kneser's conjecture says that if $k \leq n/2$, then $n-2k+2$ intersecting families are required to cover ${[n] \choose k}$; this may be achieved by taking the covering
$$\left(\bigcup_{i=1}^{n-2k+1} \{S \in {[n] \choose k}:\ i \in S\}\right) \cup {\{n-2k+2,n-2k+3,\ldots,n\} \choose k}.$$
Kneser's conjecture was proved by Lov\'asz \cite{lovasz-kneser} in 1977. His proof is one of the first examples of algebraic topology being used to resolve a problem in extremal combinatorics. Shortly afterwards, B\'ar\'any \cite{barany} gave a shorter proof, also topological, relying on the Borsuk-Ulam theorem.

An attractive variant of Kneser's problem concerning the Boolean cube $\{0,1\}^n$, was recently posed (independently) by Alon \cite{alon-talk} and Long \cite{long-personal}. For each $1 \leq t \leq n$, let us define $G_{n,t}$ to be the graph with vertex-set $\{0,1\}^n$, where two vertices are joined by an edge if their Hamming distance is at least $n-t$ (i.e.\ iff they agree on at most $t$ coordinates). The problem is to find $\chi(G_{n,t})$. It is easy to show that for $t \leq \sqrt{n}$, we have $t+1 \leq \chi(G_{n,t}) \leq O(t^2)$. (The upper bound may be proved using a colouring where the colour-classes consist of $Ct^2$ Hamming balls (or subsets thereof), whose centres are chosen independently at random; this works with high probability, provided $C>0$ is a sufficiently large absolute constant. The lower bound follows from Kneser's conjecture, now a theorem, together with the fact that $G_{n,t}$ contains a copy of $K_{n,\lceil (n-t)/2\rceil}$.) However, for $t \leq \sqrt{n}$, it is unknown whether $\chi(G_{n,t})$ is linear or quadratic in $t$.

\section{The structure of `large' intersecting families}
\label{sec:stab}
Most of the questions we have considered up to now have simply asked for the maximum possible size of a family of mathematical objects that satisfies some property $P$: this is the perhaps the most obvious question to ask, from the point of view of extremal combinatorics. Another natural class of questions asks for a description of the structure of `large' families of objects satisfying a property $P$. Here, `large' does not necessarily mean `of the maximum possible size': it can mean, for example, within a factor of $1-\epsilon$ of the maximum possible size (for a sufficiently small $\epsilon >0$), or it can mean within a factor of $c$ of the maximum possible size, for a fixed positive constant $c>0$ (letting the size of the ground-set tend to infinity). Different notions of `large' typically lead to structural results of different flavours. Sometimes, it so happens that `large' families share some structural features of the extremal families (those of the maximum possible size): this phenomenon is sometimes known as `stability'. Sometimes, it happens that the structure of `large' families can differ wildly from that of the extremal families: a phenomenon we may call `instability'.

One of the first `stability' results in the area was obtained by Hilton and Milner \cite{hilton-milner} in 1967: this strengthens the Erd\H{o}s-Ko-Rado theorem.
\begin{theorem}[Hilton-Milner, 1967]
\label{thm:hm}
Let $3\leq k < n/2$, and let $\f \subset {[n] \choose k}$ be an intersecting family such that $\cap_{S \in \f}S = \emptyset$. Then
\begin{equation} \label{eq:hm} |\f| \leq {n-1 \choose k-1}-{n-k-1 \choose k-1}+1.\end{equation}
If equality holds, then either (i) there exists $i \in [n]$ and $T \in {[n] \setminus \{i\} \choose k}$ such that
$$\f = \{T\} \cup \{S \in {[n] \choose k}:\ i \in S,\ S \cap T \neq \emptyset\},$$
or else (ii) $k=3$ and there exists $Y \in {[n] \choose 3}$ such that
$$\f = \{S \in {[n] \choose 3}:\ |S \cap Y| \geq 2\}.$$
\end{theorem}

This is a beautiful and strong (in fact, exact) result. It says that a rather strong form of stability occurs: either an intersecting family $\f$ is intersecting for `trivial' reasons (viz., because there exists an element of $[n]$ contained in all of the members of $\f$), or else it has size significantly smaller than the maximum. We note that
$${n-1 \choose k-1}-{n-k-1 \choose k-1} < k{n-2 \choose k-2} = \frac{k(k-1)}{n-1} {n-1 \choose k-1},$$
as can be seen from a simple union bound (the extremal families of type (i) in the Hilton-Milner theorem, are contained within $\cup_{j \in T} \{S \in {[n] \choose k}:\ \{i,j\} \subset S\}$), and the right-hand side is $o({n-1 \choose k-1})$ whenever $k = o(\sqrt{n})$. So whenever $k = o(\sqrt{n})$, the maximum possible size of a `non-trivially' intersecting family is an $o(1)$-fraction of the maximum possible size of an intersecting family, and the bound (\ref{eq:hm}) is very strong. On the other hand, we have
$$\frac{{n-k-1 \choose k-1}}{{n-1 \choose k-1}} = \frac{n-k-1}{n-1}\cdot \frac{n-k-2}{n-2}\cdot \ldots \cdot \frac{n-2k+1}{n-k+1} \leq \left(1-\frac{k}{n-1}\right)^{k-1} \leq e^{-k(k-1)/(n-1)},$$
which is $o(1)$ when $\sqrt{n} = o(k)$. So whenever $\sqrt{n} = o(k)$, the maximum possible size of a `non-trivially' intersecting family is within a $(1-o(1))$-fraction of the maximum possible size of an intersecting family, and the bound (\ref{eq:hm}) is not perhaps so strong.

We note that when $k$ is close to $n/2$, the Erd\H{o}s-Ko-Rado theorem exhibits what we may call `instability'. Indeed, the Erd\H{o}s-Ko-Rado theorem itself tells us that when $n=2k+1$, the unique extremal families are those consisting of all the $k$-element sets containing a fixed point, but the intersecting family
$$\f = \{S \subset [2k+1]:\ |S \cap [k]| > k/2\}$$
has $|\f| = (1-O(1/\sqrt{k})){n-1 \choose k-1}$, and yet is very far in structure from the extremal families.

It is natural to ask for structural information about intersecting families which have size below the bound (\ref{eq:hm}), but which are still `large' to some extent. Such information is provided by a beautiful 1987 theorem of Frankl. To state it in full, we need some more definitions.

If $\f \subset \p([n])$, we define $\deg(\f) : = \max_{j \in [n]} |\{F \in \f:\ j \in F\}|$ to be the maximum degree of $\f$ (considering $\f$ as a hypergraph). For $2 \leq k \leq n-1$ and $3 \leq i \leq k+1$, we define
\begin{align*}
\g_i &:= \left\{S \in {[n] \choose k}: 1 \in S \text{ and } S \cap \{2,3,\ldots,i\} \neq \emptyset\right\}\\
& \cup \left\{S \in {[n] \choose k}: 1 \not \in S \text{ and }\{2,3,\ldots,i \} \subset S\right\}.
\end{align*}
Clearly, each $\g_i$ is an intersecting family. 
\begin{theorem}[Frankl, 1987]\label{thm:frankl87}
Let $n,k,i \in \mathbb{N}$ with $k < n/2$ and $3 \leq i \leq k+1$. Let $\f \subset {[n] \choose k}$ be an intersecting family with $\mathrm{deg}(\f) \leq \mathrm{deg}(\g_i)$. Then
\begin{equation}
\label{eq:frankl-bound}
|\f| \leq |\g_i|.
\end{equation}
If equality holds in (\ref{eq:frankl-bound}), then either $\f$ is isomorphic to $\g_i$ or else $i=4$ and $\f$ is isomorphic to $\g_3$.
\end{theorem}
(Here, `isomorphic' means `equal up to permutations of $[n]$'.) This theorem is clearly sharp, as is evidenced by the $\g_i$ themselves. We note that $\g_{k+1}$ is precisely an extremal family of type (i) in the Hilton-Milner theorem, so Frankl's theorem strengthens the Hilton-Milner theorem. (Indeed, if $\f \subset {[n] \choose k}$ is intersecting with $\cap_{F \in \f} F = \emptyset$, then for any $i \in [n]$ there exists $S \in \f$ such that $i \notin S$, and there are exactly $d(\g_{k+1})$ elements of ${[n] \choose k}$ that intersect $F$ and contain $i$, so $\deg(\f) \leq \deg(\g_{k+1})$, and therefore by Frankl's theorem, $|\f| \leq |\g_{k+1}|$.) However, unlike the Hilton-Milner theorem, Frankl's theorem also provides structural information when $|\f| \geq c {n-1 \choose k-1}$ and $\sqrt{n} = o(k)$ (provided $c \geq 3k/n$). For example, the $i=3$ case of Frankl's theorem implies that if $\f \subset {[n] \choose k}$ is intersecting with $|\f| = 3{n-2 \choose k-2}-{n-3 \choose k-3}$, then $\deg(\f) \geq 2{n-2 \choose k-2} - {n-3 \choose k-3}$, so at least (roughly) two-thirds of the members of $\f$, contain some fixed $i \in [n]$. 

Frankl's proof of Theorem \ref{thm:frankl87} is very elegant, and purely combinatorial, relying on $ij$-compressions, and also using the Kruskal-Katona theorem.

It is natural to ask similar questions about $t$-intersecting families. The following theorem of Ahlswede and Khachatrian \cite{ak-nontrivial} gives an exact analogue of the Hilton-Milner theorem, for $t$-intersecting families.

\begin{theorem}[Ahlswede-Khachatrian, 1996]
\label{thm:aknontriv}
Let $n > (t+1)(k-t+1)$, and let $\f \subset {[n] \choose k}$ be $t$-intersecting with $|\cap_{F \in \f}F| < t$. If $k \leq 2t+1$, then
\begin{equation}
\label{eq:bound-1}
|\f| \leq |\f_1|,
\end{equation}
where
$$\f_1: = \left\{S \in {[n] \choose k}:\ |S \cap [t+2]| \geq t+1\right\},$$
and equality holds in (\ref{eq:bound-1}) only if $\f$ is isomorphic to $\f_1$. If $k > 2t+1$, then
\begin{equation}
\label{eq:bound-2}
|\f|\leq \max\{|\f_1|,|\mathcal{H}|\},
\end{equation}
where
$$\mathcal{H} = \left\{S \in {[n] \choose k}:\ [t] \subset S,\ S \cap \{t+1,\ldots,k+1\} \neq \emptyset\} \cup \{[k+1] \setminus \{i\}:\ i \in [t]\right\},$$
and equality holds in (\ref{eq:bound-2}) only if $\f$ is isomorphic to $\f_1$ or to $\mathcal{H}$.
\end{theorem}
The proof of Theorem \ref{thm:aknontriv} uses the same methods as that of Ahlswede and Khachatrian's complete intersection theorem, discussed above. 

 Interestingly, an exact analogue, for $t$-intersecting families, of Frankl's theorem (Theorem \ref{thm:frankl87}), is not known. However, a number of `approximate' structure theorems exist. Interestingly, several of these results use (in their proofs) techniques from the analysis of Boolean functions.
 
 Note that the family $\mathcal{H}$ in Theorem \ref{thm:aknontriv} satisfies
 $$|\mathcal{H}| = {n-t \choose k-t} - {n-k-1 \choose k-t} + t,$$
 and
 $$\frac{{n-k-1 \choose k-t}}{{n-t \choose k-t}} \leq \left(1-\frac{k-t+1}{n-t}\right)^{k-t} \leq \exp(-(k-t+1)(k-t)/(n-t)),$$
 which is $o(1)$ (as $n \to \infty$) if $k > 2t+1$ and $\sqrt{n} = o(k)$. Hence, for $k > 2t+1$ and $\sqrt{n} = o(k)$, we have $|\mathcal{H}| = (1-o_{n \to \infty}(1)){n-t \choose k-t}$, so in this case (and in particular, in the case where $k = \Theta(n)$), Theorem \ref{thm:aknontriv} only yields structural information about $t$-intersecting families with size within a $(1-o(1))$-factor of the maximum possible size.
 
 The following theorem of Friedgut \cite{friedgut-measure} was the first structural result concerning $t$-intersecting families of size a constant fraction of the maximum possible size, for $t>1$ and $k = \Theta(n)$.
 
\begin{theorem}[Friedgut, 2008]
\label{thm:frt}
For any $t \in \mathbb{N}$ and $\eta >0$, there exists $C=C(t,\eta)>0$ such that the following holds. Let $\eta n < k < (1/(t+1)-\eta)n$ and let $\epsilon \geq \sqrt{(\log n) / n}$. If $\f \subset {[n] \choose k}$ is a $t$-intersecting family with $|\f| \geq (1-\epsilon){n - t \choose k-t}$, then there exists $B \in {[n] \choose t}$ such that
$$|\{F \in \f:\ B \not\subset F\}| \leq C \epsilon {n-t \choose k-t}.$$
\end{theorem}

We proceed to sketch the proof. One key idea of Friedgut's proof is to work with the {\em $p$-biased measure} on $\pn$, where $p \approx k/n$, rather than the uniform (counting) measure on ${[n] \choose k}$, and then translate results from the former setting to the latter. This strategy makes sense, as the $p$-biased measure on $\pn$ has nicer analytic and algebraic properties than the uniform measure on ${[n] \choose k}$.

The $p$-biased measure $\mu_p$ on $\pn$ is defined as follows: for $S \subset [n]$, we define
$$\mu_p(\{S\}) = p^{|S|}(1-p)^{n-|S|},$$
and for a family $\f \subset \pn$, we define
$$\mu_p(\f) = \sum_{S \in \f} \mu_p(\{S\}).$$
Hence, $\mu_p(\f)$ is the probability that if a subset $S \subset [n]$ is chosen at random, by placing each point of $[n]$ in $S$ independently with probability $p$, then the resulting set $S$ lies in $\f$.

For a family $\f \subset \pn$, we define $\f^{\uparrow}$ to be its up-closure, i.e. $\f^{\uparrow} : = \{T \subset [n]:\ S \subset T \text{ for some }S \in \f\}$. The following lemma (proved in a slightly different form, by Friedgut in \cite{friedgut-measure}) states that if $p$ is a little larger than $k/n$, and $\f \subset {[n] \choose k}$, then $\mu_p(\f^{\uparrow})$ cannot be much smaller than $|\f|/{n \choose k}$.
\begin{lemma}
\label{lem:chernoff}
Let $n,k \in \mathbb{N}$ and suppose that $0 < p,\phi < 1$ satisfy
\[p \ge \frac{k}{n} + \frac{\sqrt{2n \log (1/\phi)}}{n}.\]
Then for any family $\f \subset {[n] \choose k}$, we have
\[\mu_p(\f^{\uparrow}) > (1-\phi) \frac{|\f|}{\binom{n}{k}}.\]
\end{lemma}

We provide a proof, for completeness. The proof rests on the local LYM inequality. For any family $\mathcal{A} \subset {[n] \choose k}$, we write
$$\partial^+ \mathcal{A}: = \{S \in {[n] \choose k+1}:\ S \supset T \text{ for some }T \in \mathcal{A}\}$$
for the {\em upper shadow} of $\mathcal{A}$, and 
$$\partial^{+(j)}(\mathcal{A}) := \{S \in {[n] \choose k+j}:\ S \supset T \text{ for some }T \in \mathcal{A}\} = \partial^{+}(\partial^{+(j-1)}\mathcal{A})$$
for its $j$th iterate (for each $j \in \mathbb{N}$ with $j\leq n-k$). The Local LYM inequality (see e.g.\ \cite{bollobas}, \S 3) states that for any integers $1 \leq k < n$ and any family $\mathcal{A} \subset {[n] \choose k}$, we have
$$\frac{|\partial^{+} \mathcal{A}|}{{n \choose k+1}}\geq \frac{|\mathcal{A}|}{{n \choose k}}.$$
Iterating the local LYM inequality yields
\begin{equation}\label{eq:ll-it} \frac{|\partial^{+(j)} \mathcal{A}|}{{n \choose k+j}}\geq \frac{|\mathcal{A}|}{{n \choose k}}\end{equation}
for all $j \leq n-k$.

\begin{proof}(of Lemma \ref{lem:chernoff}.)
Let $\f \subset {[n] \choose k}$, let $\delta := |\f|/\binom{n}{k}$ and let $X \sim \text{Bin}(n,p)$. 
We will use the Chernoff bound 
\begin{equation}\label{eq:chernoff} \Prob(X < (1-\eta)np) < \exp(-\eta^2np/2) \quad \forall \eta >0.\end{equation}
Observe that (\ref{eq:ll-it}) implies 
\[\frac{|\f^{\uparrow} \cap {[n] \choose l} |}{ \binom{n}{l}} \geq \frac{|\f|}{\binom{n}{k}} = \delta \quad \forall k \leq l \leq n.\]
Hence,
\begin{align*}
\mu_{p}\left(\f^{\uparrow}\right) & = \sum_{l=k}^{n}p^{l}\left(1-p\right)^{n-l} \left|\f^{\uparrow} \cap \binom{[n]}{l}\right| \\
& \geq \sum_{l=k}^{n}p^{l}\left(1-p\right)^{n-l}\binom{n}{l}\delta\\
&= \Prob(X \geq k) \cdot \delta \\ 
&> (1-\phi) \cdot \delta,
\end{align*}
where the last inequality above follows from setting $k = (1-\eta)np$ and using the Chernoff bound above.
\end{proof}

Friedgut's next step is to prove the following via a spectral method.
\begin{lemma}
\label{lem:friedgut-spectral}
If $\f \subset \pn$ is $t$-intersecting and $0 < p < 1/(t+1)$, then $\mu_p(\f) \leq p^t$, and equality holds if and only if there exists $B \in {[n] \choose t}$ such that $\f = \{S \subset [n]:\ S \supset B\}$. Moreover, if $\mu_p(\f) \geq (1-\epsilon)p^t$ then there exists $B \in {[n] \choose t}$ such that $\mu_p(\f \Delta \{S \subset [n]:\ S \supset B\}) = O_p(\epsilon)$.
\end{lemma}
For brevity, we will sometimes refer to families of the form $\{S \subset [n]:\ S \supset B\}$ (for $B \in {[n] \choose t}$) as the $t$-umvirates; note that we previously used this term for their $k$-uniform analogues, i.e.\ families of the form $\{S \in {[n] \choose k}:\ S \supset B\}$ (for $B \in {[n] \choose t}$).

The first part of Lemma \ref{lem:friedgut-spectral} was already known; indeed, it follows easily from Ahlswede and Khachatrian's complete intersection theorem (see Section \ref{sec:sets}), as was observed by Dinur and Safra \cite{dshardness}. However, the second (stability) part was new. Friedgut's spectral proof yields stability, whereas the Ahlswede-Khachatrian machinery does not seem to do so.
 
Friedgut's spectral proof of Lemma \ref{lem:friedgut-spectral} relies on some deep machinery from the analysis of Boolean functions, together with the following generalisation of Theorem \ref{thm:delsarte-hoffman}. Let $G = (V,E)$ be a finite graph, and let $\mu$ be a probability measure on $V(G)$. We say that a matrix $M \in \mathbb{R}^{V(G)^2}$ is a {\em pseudoadjacency matrix of $G$ with respect to the measure $\mu$} if $M_{x,y} = 0$ whenever $xy \notin E(G)$, $M$ has all its row-sums equal and positive, and $M$ is symmetric with respect to the inner product on $\mathbb{R}^{V(G)}$ induced by $\mu$, i.e.\ the inner product
$$\langle u,v \rangle: = \sum_{x \in V(G)} \mu(x) u(x) v(x) \quad \forall u,v \in \mathbb{R}^{V(G)}.$$
The symmetry condition means that $\langle M u,v \rangle = \langle u, Mv \rangle$ for all $u,v \in \mathbb{R}^{V(G)}$, or equivalently, that
 $$\mu(x)M_{x,y} = \mu(y) M_{y,x}\quad \forall x,y \in V(G).$$
It is easy to see that the proof of Hoffman's bound (Theorem \ref{thm:hoffman}) generalises to imply the following `measure-theoretic' version.
 
 \begin{theorem}
 \label{thm:biased-hoff}
 Let $G = (V,E)$ be a finite graph, let $\mu$ be a probability measure on $V(G)$, and let $M$ be a pseudoadjacency matrix of $G$ with respect to $\mu$. Let $\lambda_1$ be the eigenvalue of $M$ corresponding to the eigenvector $\mathbf{f}: = (1,1,\ldots,1)$, and let $\lambda_{\min}$ be the minimum (i.e.\ the `most negative') eigenvalue of $M$. Let $S \subset V(G)$ be an independent set of vertices of $G$. Then
$$\mu(S) \leq \frac{-\lambda_{\min}}{\lambda_1-\lambda_{\min}}.$$
Equality holds only if 
$$1_S - \frac{|S|}{|V(G)|}(1,1,\ldots,1)$$
is a $\lambda_{\min}$-eigenvector of $M$, where $1_{S}$ denotes the indicator vector of $S$.
\end{theorem}

Let $K(n,<\!\!t)$ denote the graph with vertex-set $\mathcal{P}([n])$, where two sets are joined by an edge if their intersection has size less than $t$. Friedgut's strategy is essentially to construct a pseudoadjacency matrix of $K(n,<\!\!t)$ with respect to the measure $\mu_p$, which has appropriate maximum and minimum eigenvalues so that applying Theorem \ref{thm:biased-hoff} to it, implies the upper bound in Lemma \ref{lem:friedgut-spectral}. Friedgut explains how one is naturally led to this construction (we do not repeat his explanation, due to lack of space). The pseudoadjacency matrix constructed by Friedgut has the $p$-biased characters of $\pn$ as an orthonormal basis of eigenvectors (orthonormal, that is, with respect to the inner product induced by $\mu_p$). The $p$-biased characters of $\pn$ are the functions $\{\chi_S^{(p)}:\ S \subset [n]\}$ defined by
$$\chi_S^{(p)}:\pn \to \mathbb{R};\ \chi_S^{(p)}(A) = \left(-\sqrt{(1-p)/p}\right)^{|S \cap A|} \left(\sqrt{p/(1-p)}\right)^{|S|-|S \cap A|}\ \forall A \subset [n].$$
The term `character' is a slight abuse of terminology, since these are only characters of a group in the case $p=1/2$ (where they are characters of $\mathbb{Z}_2^n$, under the natural identification of $\mathbb{Z}_2^n$ with $\pn$), but they share some of the useful (for us) properties of the `genuine' characters $\chi_{S}^{(1/2)}$. Specifically, they are the unique set of functions (up to changes of sign) that form an orthonormal basis of $\mathbb{R}^{\pn}$ with respect to $\mu_p$, such that for all $S \subset [n]$, the function $\chi_S$ depends only upon the coordinates in $S$. As such, they are of crucial importance in the Analysis of Boolean Functions, an important (and rapidly growing) field connecting combinatorics, discrete analysis and theoretical computer science. The reader is referred to \cite{odonnell} for more background on the $p$-biased characters, and their importance.

Since the $p$-biased characters form a basis of $\mathbb{R}[\pn]$, any function $f:\pn \to \mathbb{R}$ has a unique expression in the form
$$f = \sum_{S \subset [n]} \alpha^{(p)}_S \chi_S^{(p)};$$
this is known as the {\em $p$-biased Fourier expansion} of $f$. We write
$$W^{(p)}_{>r}(f) := \sum_{S \subset [n]: |S| > r} (\alpha_S^{(p)})^2;$$
this quantity can be viewed as the `Fourier weight' of the function $f$ on `high frequencies' (i.e.\ on the characters corresponding to sets of size greater than $r$).

An examination of (\ref{eq:ip}), together with the eigenvalues of the matrix constructed by Friedgut, shows that a $t$-intersecting family $\f \subset \pn$ with $\mu_p(\f) \geq (1-\epsilon)p^t$, has characteristic function $1_{\f}$ satisfying $W^{(p)}_{>t}(1_{\f}) = O_p(\epsilon)$. The next step is to apply the following deep theorem of Kindler and Safra \cite{kindler-safra}, with $f = 1_{\f}$.

\begin{theorem}[Kindler-Safra, 2003]
\label{thm:ks}
For $0 < p < 1$ and $t \in \mathbb{N}$, there exist $C = C(p,t) >0$ and $K = K(p,t)>0$ such that the following holds. Let $\delta >0$, and let $f:\pn \to \{0,1\}$ such that $W_{>t}^{(p)}(f) < \delta$. Then there exists a $K$-junta $g:\pn \to \{0,1\}$ such that $\|f-g\|^2_2 < C\delta$. (Here, $\| \cdot \|_2$ is the norm induced by the $\mu_p$-inner product.) 
\end{theorem}
Here, a Boolean function $g: \pn \to \{0,1\}$ is said to be a {\em $K$-junta} if $\{S \subset [n]:\ g(S)=1\}$ is a $K$-junta, or equivalently, identifying $\pn$ with $\{0,1\}^n$, $g$ is a $K$-junta if it depends upon at most $K$ coordinates. Theorem \ref{thm:ks} says that a Boolean function whose Fourier weight is concentrated on low frequencies, is close (in $L^2$-norm) to a Boolean junta. Friedgut shows that, in fact, the Boolean function $g$ that arises when Theorem \ref{thm:ks} is applied to $1_{\f}$, is of the form $A \mapsto 1_{B \subset A}$, for some $B \in {[n] \choose t}$. This implies the conclusion of Lemma \ref{lem:friedgut-spectral}.

It is not too hard to deduce Theorem \ref{thm:frt} from Lemmas \ref{lem:chernoff} and \ref{lem:friedgut-spectral}; we omit the details.

In \cite{ekl}, Keller, Lifshitz and the author obtain the following.
\begin{theorem}[E.-Keller-Lifshitz, 2019] 
\label{thm:stability-uniform-t}
For any $t \in \mathbb{N}$ and $\eta >0$, there exists $\delta_{0}=\delta_0(\eta,t)>0$ such that the following holds. Let $n,k \in \mathbb{N}$ with $k \leq (\tfrac{1}{t+1}-\eta)n$, and let $d\in\mathbb{N}$. Let $\mathcal{F}\subset {[n] \choose k}$ be a $t$-intersecting family with
\begin{equation} \label{eq:wilson-stability-1}
\left|\mathcal{F}\right| > \max\left\{ \binom{n-t}{k-t}\left(1-\delta_{0}\right),\binom{n-t}{k-t}-
\binom{n-t-d}{k-t}+\left(2^{t}-1\right)\binom{n-t-d}{k-t-d+1}\right\}.
\end{equation}
Then there exists a set $B \in {[n] \choose t}$ such that
\begin{equation}\label{eq:wilson-stability-2}
\left|\mathcal{F} \setminus \{S \subset [n]:\ S \supset B\} \right| \leq \left(2^{t}-1\right)\binom{n-t-d}{k-t-d+1}.
\end{equation}
\end{theorem}

Theorem \ref{thm:stability-uniform-t} improves significantly over Theorem \ref{thm:frt}. As well as applying for smaller $k$, it implies the following in the case where $k = \Theta(n)$.

\begin{cor}
\label{corr:rough}
Let $n,k,t \in \mathbb{N}$ with $\eta n \leq k \leq (\tfrac{1}{t+1}-\eta)n$, let $\epsilon >0$, and let $\mathcal{F}\subset {[n] \choose k}$ be a $t$-intersecting family with
$$\left|\mathcal{F}\right| \geq (1-\epsilon)\binom{n-t}{k-t}.$$
Then there exists a set $B \in {[n] \choose t}$ such that
$$\left|\mathcal{F} \setminus \{S \subset [n]:\ S \supset B\} \right| = O_{t,\eta}(\epsilon^{\log_{1-k/n}(k/n)}) \binom{n-t}{k-t}.$$
\end{cor}

The $\epsilon$-dependence in Corollary \ref{corr:rough} is tight up to a constant factor depending upon $t$ and $\eta$ alone. Moreover, for $d$ sufficiently large (as a function of $t$ and $\eta$), Theorem~\ref{thm:stability-uniform-t} is tight (even for $k=o(n)$), up to
replacing $2^t-1$ with $t$ in the inequalities (\ref{eq:wilson-stability-1}) and (\ref{eq:wilson-stability-2}), as evidenced by the families $(\mathcal{F}_{t,s})_{t,s \in \mathbb{N}}$, defined by
\begin{align*}
\mathcal{F}_{t,s} & :=\left\{ A \in {[n]\choose k}\,:\,\left[t\right]\subset A,\,\left\{t+1,\ldots,t+s\right\}\cap A\ne\emptyset\right\} \\
 & \cup\left\{ A \in {[n] \choose k}\,:\,|\left[t\right]\cap A|=t-1,\,\left\{t+1,\ldots,t+s\right\}\subset A\right\}.
\end{align*}

 In fact, in \cite{ekl}, we present a general strategy for proving `stability' versions of extremal theorems, where the following conditions hold. Suppose $P$ is a property of subsets of $\pn$ that is preserved under taking the up-closure: meaning, if $\f \subset \pn$ has the property $P$, then so does $\f^{\uparrow}$. (If $\f = \f^{\uparrow}$, then we call $\f$ an {\em up-set}.) Suppose there exists $p_0 \in (0,1)$ such that 
 $$\max\{\mu_{p_0}(\f):\ \f \subset \pn,\ \f \text{ has the property }P\}$$
 is attained by $t$-umvirates for some $t \in \mathbb{N}$, i.e.\ by families of the form $\{S \subset [n]:\ S \supset B\}$ for $B \in {[n] \choose t}$, so that 
 $$\max\{\mu_{p_0}(\f):\ \f \subset \pn,\ \f \text{ has the property }P\} = p_0^t.$$
 Then we are able to show not only that 
 $$\max\{\mu_{p}(\f):\ \f \subset \pn,\ \f \text{ has the property }P\} = p^t$$
 for all $p < p_0$ (a statement that was already known; see e.g.\ \cite{grimmett} Theorem 2.38), but also that if $\epsilon >0$ and $p < p_0$ is bounded away from $p_0$, then any up-set $\f \subset \pn$ with the property $P$ and with $\mu_p(\f) \geq (1-\epsilon)p^t$, must be close in structure (with the closeness depending on $\epsilon$) to a $t$-umvirate. 
 
Our proof of this is based on isoperimetric inequalities for the hypercube (the graph with vertex-set $\{0,1\}^n$, where two 0-1 vectors are joined by an edge if they differ in just one coordinate); unlike the proof of Friedgut's theorem, our proof does not use Fourier analysis. We proceed as follows. Given an up-set $\f \subset \mathcal{P}([n])$, we view it as a subset of the hypercube $\{0,1\}^n$, and we compare the measures $\mu_p(\f)$ for different values of $p$. A well-known lemma of Russo \cite{russo} states that the function $f: p \mapsto \mu_p(\f)$ satisfies $\frac{df}{dp} = \mu_{p}(\partial \f)$, where $\partial \f$ denotes the edge boundary of $\f$. (If $S \subset \{0,1\}^n$, the {\em edge boundary} $\partial S$ of $S$ is defined to be the set of edges of the hypercube that join an element of $S$ to an element of $\{0,1\}^n \setminus S$. We define $\mu_p(xy) := \mu_p(x)+\mu_p(y)$, for any hypercube edge $xy$.) Our assumptions on $\f$ supply us with two values $p_1<p_0$ such that $\mu_{p_0}(\f)$ is not much larger than $\mu_{p_1}(\f)$. By applying the mean value theorem to the function $f$, it follows that there exists $p_2 \in (p_1,p_0)$ such that the edge boundary of $\f$ is small with respect to $\mu_{p_2}$, i.e.\ $\mu_{p_2}(\partial \f)$ is small. 

On the other hand, the biased version of the edge-isoperimetric inequality on the hypercube asserts that for any  $p \in (0,1)$ and any up-set $\mathcal{F} \subset \p([n])$, we have
$$p \mu_p(\partial \f) \geq \mu_{p}(\f) \log_{p}(\mu_{p}(\f)),$$ 
the minimum being attained (for any $p$) only by sets of the form $\{S \subset [n]: S \supset B\}$. Moreover, the following stability version of this isoperimetric inequality (due to the author, Keller and Lifshitz \cite{ekl-iso}) implies that if $\mu_p(\partial \f)$ is small, then $\f$ is close to a set of the form $\{S \subset [n]: S \supset B\}$ (with respect to $\mu_p$).

\begin{theorem}[E.-Keller-Lifshitz, 2019]
\label{thm:mon-iso-stability} For any $\eta>0$, there exist $C_{1}=C_{1}(\eta)$,
$c_{0}=c_{0}(\eta)>0$ such that the following holds. Let $0<p\leq1-\eta$,
and let $\epsilon\leq c_{0}/\ln(1/p)$. Let $\f \subset \pn$ be an up-set such that
\[
p\mu_{p}(\partial \f)\leq\mu_{p}(\f)\left(\log_{p}(\mu_{p}(\f))+\epsilon\right).
\]
Then there exists $B \subset [n]$ such that
\begin{equation}
\mu_{p}(\f\Delta\{S \subset [n]:\ S \supset B\})\leq C_1 \frac{\epsilon \ln(1/p)}{\ln\left(1/(\epsilon \ln(1/p)\right)}\mu_{p}(\f).\label{eq:conc}
\end{equation}
\end{theorem}

Our proof of Theorem \ref{thm:mon-iso-stability} is purely combinatorial, though rather intricate. Applying Theorem \ref{thm:mon-iso-stability} with $p=p_2$, where $p_2$ is obtained from the `mean value theorem' argument above, we see that the original up-set $\f$ is close to a $t$-umvirate with respect to $\mu_{p_2}$. A monotonicity argument then implies that $\f$ is close to the same $t$-umvirate with respect to $\mu_{p_1}$.

The argument sketched above yields a stability result in the biased-measure setting. By applying an argument analogous to Friedgut's (based on Lemma \ref{lem:chernoff}), we can obtain from this a stability result for uniform families, i.e.\ for subsets of ${[n] \choose k}$ for $k/n$ bounded away from $p_0$ (from above).

\section{Imposing extra structure on the ground set}
\label{sec:structure}
The problems we have considered so far, are intersection problems about `unstructured' ground-sets. (Though we often used the ground-set $[n] = \{1,2,\ldots,n\}$, for notational convenience, any $n$-element set would have sufficed.) It is natural to ask what happens in an intersection problem when we impose some extra structure on the ground-set.

One of the most natural structures to impose is the additive structure of the integers. This leads to several attractive problems, some solved and some still open.

As usual, we let $\mathbb{Z}_n$ denote the cyclic group of the integers modulo $n$, under addition. For a set $B \subset \mathbb{Z}_n$, we say a family $\f \subset \mathcal{P}(\mathbb{Z}_n)$ is {\em $B$-translate-intersecting} if for any two sets $S,T \in \f$, there exists $x \in \mathbb{Z}_n$ such that $B +x \subset S \cap T$, i.e.\ $S$ and $T$ intersect on a (cyclic) translate of $B$. For $n \in \mathbb{N}$ and $B \subset \mathbb{Z}_n$, we let $m_n(B)$ denote the maximum possible size of a $B$-translate-intersecting family of subsets of $\mathbb{Z}_n$. Chung, Frankl, Graham and Shearer \cite{cfgs} made the following conjecture.

\begin{conj}[Chung, Frankl, Graham and Shearer, 1986]
\label{conj:cfgs-translate}
For any $n \in \mathbb{N}$ and any $B \subset \mathbb{Z}_n$, $m_n(B) = 2^{n-|B|}$.
\end{conj}
Conjecture \ref{conj:cfgs-translate} says that one cannot do better than to take the `trivial' family 
$$\f = \{S \subset \mathbb{Z}_n:\ B \subset S\};$$
in other words, the Erd\H{o}s-Ko-Rado property holds for this problem. Chung, Frankl, Graham and Shearer prove this in the case of $B$ a (cyclic) interval. Their proof-method relies on the following observation. If  $B \subset \mathbb{Z}_n$, we say a family $\f \subset \mathcal{P}(\mathbb{Z}_n)$ is {\em $B$-translate-agreeing} if for any $S,T \in \f$, there exists $x \in \mathbb{Z}_n$ such that $(S \Delta T) \cap (B + x) = \emptyset$, i.e.\ the sets $S$ and $T$ have exactly the same intersection with the set $B+x$. Chung, Frankl, Graham and Shearer observed that the maximum possible size of a $B$-translate-agreeing family of subsets of $\mathbb{Z}_n$ is the same as the maximum possible size of a $B$-translate-intersecting family of subsets of $\mathbb{Z}_n$. Since the relation of $S$ and $T$ being `$B$-translate-agreeing' is preserved when we apply the same translation to $S$ and $T$, this observation opens up the possibility that a partitioning proof will work, and this is indeed the case. The original partitioning proof of Chung, Frankl, Graham and Shearer was somewhat indirect, but Russell \cite{russell} more recently gave a direct partitioning proof. Russell proves directly that for any interval $B \subset \mathbb{Z}_n$, we may partition $\mathcal{P}(\mathbb{Z}_n)$ into $2^{n-|B|}$ parts such that no two distinct sets in the same part agree on any (cyclic) translate of the interval $B$. Since any $B$-translate-intersecting family can contain at most one set in each part, this immediately implies Conjecture \ref{conj:cfgs-translate} in the case where $B$ is an interval.

The general case of Conjecture \ref{conj:cfgs-translate} remains open.

A different notion was considered by Simonovits and S\'os in 1976. For $k \in \mathbb{N}$ with $k \geq 3$, we say a family $\f \subset \pn$ is {\em $k$-AP-intersecting} if for any $S,T \in \f$, there exists a $k$-term arithmetic progression $P$ with nonzero common difference, such that $P \subset S \cap T$. Simonovits and S\'os conjectured the following.
\begin{conj}[Simonovits-S\'os, 1976]
If $\f \subset \pn$ is 3-AP-intersecting, then $|\f| \leq 2^{n-3}$.
\end{conj}
This conjecture is completely open; indeed, somewhat embarrassingly, no upper bound of the form $(1/2-c) \cdot 2^{n}$ (for $c$ a positive absolute constant) is yet known, as far as we are aware. (Note that an upper bound of $2^{n-1}$ is trivial, as a 3-AP-intersecting family is certainly 1-intersecting.) Needless to say, the analogous problem for $k$-AP-intersecting families (for $k > 3$) is also completely open.
\newline

Another natural structure to impose on the ground set is a graph structure, viz.\ the edge-set of the complete graph. A family of labelled graphs with vertex-set $[n]$ is naturally identified with a subset of $\mathcal{P}({[n] \choose 2})$ (by identifying a labelled graph with its edge-set, which is a subset of ${[n] \choose 2}$). If $H$ is a fixed, unlabelled graph, we say a family $\f$ of labelled graphs with the (common) vertex-set $[n]$ is {\em $H$-intersecting} if any two of the graphs in $\f$ share a copy of $H$. For example, a family $\f$ of labelled graphs on a common vertex-set is {\em triangle-intersecting} if any two graphs in $\f$ share some triangle. Simonovits and S\'os made the following conjecture in 1976.

\begin{conj}[Simonovits-S\'os, 1976]
\label{conj:sim-sos-conj}
If $\f \subset \mathcal{P}({[n] \choose 2})$ is a triangle-intersecting family of labelled graphs on the vertex-set $[n]$, then $|\f| \leq 2^{{n \choose 2}-3}$.
\end{conj}

Conjecture \ref{conj:sim-sos-conj} says that one cannot do better than to take all graphs containing a fixed triangle.

We note that a triangle-intersecting family of graphs is clearly intersecting (meaning that any two graphs in the family share some edge), so the upper bound of $|\f| \leq 2^{{n \choose 2}-1}$ is trivial. The first improvement on this was due to Chung, Frankl, Graham and Shearer \cite{cfgs}, who used Shearer's entropy lemma to obtain an upper bound of $|\f| \leq 2^{{n \choose 2}-2}$: `halfway' between the trivial bound and the conjectured bound. We sketch the proof. Shearer's entropy lemma can be (re-)stated as follows.

\begin{lemma}[Shearer's lemma, restatement]
\label{lem:shearer}
Let $X = (X_1,\ldots,X_N)$ be a random vector. Let $S \subset [N]$ be a random subset with $\Pr[i \in S] \geq p$ for all $i \in [N]$. Then
$$p \cdot H[X] \leq \mathbb{E}_SH[X_S],$$
where $X_S: = (X_i)_{i \in S}$, and $H[Y]$ denotes the entropy of the random variable $Y$.
\end{lemma}

We apply Lemma \ref{lem:shearer} with $N = {n \choose 2}$. We fix an ordering of the edges of $K_n$ and we let $X = (X_1,\ldots,X_{{n \choose 2}})$ be the indicator vector of the edge-set of a uniform random graph in the triangle-intersecting family $\f$ (using this ordering). Then $H[X] = \log_2(|\f|)$. We now take $S$ to be a random set of the form ${T \choose 2} \cup {[n] \setminus T \choose 2}$, where $T \subset [n]$ is chosen uniformly at random. Clearly, any edge of $K_n$ has probability exactly $1/2$ of appearing in the random set $S$. We note that for any set $S$ of the form ${T \choose 2} \cup {[n] \setminus T \choose 2}$, the vector $X_S$ is the indicator vector of the `projected' family $\{F \cap ({T \choose 2} \cup {[n] \setminus T \choose 2}):\ F \in \f\}$, which is an intersecting family and therefore has size at most $2^{{|T| \choose 2}+{n-|T| \choose 2}-1} = 2^{|S|-1}$; therefore $H(X_S) \leq \log_2(2^{|S|-1}) = |S|-1$. Applying Lemma \ref{lem:shearer}, we obtain
$$(1/2) \cdot \log_2(|\f|) \leq \mathbb{E}_S[|S|-1] = {n \choose 2}/2-1,$$
so rearranging, $|\f| \leq 2^{{n\choose 2}-2}$, as required.

In 2012, Filmus, Friedgut and the author \cite{eff} proved Conjecture \ref{conj:sim-sos-conj} using a spectral method. We use the `generalised harmonic analysis' approach, outlined in our Introduction (Section 1). Consider the graph $G = G_n$ whose vertex-set is $\mathcal{P}({[n] \choose 2})$, i.e.\ the set of all labelled graphs on $[n]$, where we join two labelled graphs by an edge of $G_n$ if and only if they share no triangle in common. A triangle-intersecting family of labelled graphs on $[n]$ is precisely an independent set in $G_n$. Our strategy is to construct a pseudoadjacency matrix $M_n$ for $G_n$ that has appropriate maximum and minimum eigenvalues such that, when the Delsarte-Hoffman bound (Theorem \ref{thm:delsarte-hoffman}) is applied to $G_n$, we obtain the desired upper bound on the size of an independent set in $G_n$, viz.\ $2^{n-3}$. In fact, we work with a sparser graph than $G_n$ (one with more symmetries): we consider the graph $G_n'$ whose vertex-set is again $\mathcal{P}({[n] \choose 2})$, but where we join two labelled graphs $H$ and $H'$ by an edge of $G_n'$ if and only if $E(H) \Delta E(H')$ intersects every triangle. Since $G_n'$ is a subgraph of $G_n$ (on the same vertex-set), it suffices to show that any independent set in $G_n'$ has size at most $2^{n-3}$ (also, any pseudoadjacency matrix for $G_n'$ is also a pseudoadjacency matrix for $G_n$). However, $G_n'$ has the advantage of being a Cayley graph\footnote{Recall that if $(\Gamma,+)$ is an Abelian group, and $S \subset \Gamma$ with $\text{Id} \notin S$, the {\em Cayley graph of $\Gamma$ with generating set $S$} is the graph with vertex-set $\Gamma$, where $g$ is joined to $g+s$ for all $g \in \Gamma$ and $s \in S$.} of the (Abelian) group $\mathbb{Z}_2^{X}$, where $X : = {[n] \choose 2}$; indeed, its generating set is precisely the set of all complements of triangle-free graphs on $[n]$. The characters of the group $\mathbb{Z}_2^X$ therefore form an orthonormal basis of eigenvectors of the adjacency matrix of $G_n'$, and the same is true for any Cayley subgraph of $G_n'$ (this fact is well-known, and goes back to Frobenius). Specifically, if $\Gamma = \Cay(\mathbb{Z}_2^X,\mathcal{R})$ is the Cayley graph of $\mathbb{Z}_2^X$ with generating set $\mathcal{R}$, then for each $A \subset X$, the character
$$\chi_A:\ \mathbb{Z}_2^{X} \to \{\pm 1\};\quad \chi_A(S) = (-1)^{|A \cap S|} \quad (S \subset X)$$
is an eigenvector of $\Gamma$ with eigenvalue
$$\lambda_A : = \sum_{R \in \mathcal{R}} \chi_A(R).$$
(Here, we identify $\mathbb{Z}_2^{X}$ with $\mathcal{P}(X)$ in the natural way.) Our strategy is to choose $M_n$ to be an appropriate linear combination of adjacency matrices of Cayley subgraphs of $G_n'$. In other words, we choose
$$M_n = \sum_{i} c_i A(\Cay(\mathbb{Z}_2^X,\{R_i\}))$$
where $c_i \in \mathbb{R}$ for each $i$, $R_i$ is the complement of a triangle-free graph for each $i$. (Recall that, for a graph $G$, we denote by $A(G)$ the adjacency matrix of $G$.) In fact, it suffices to take each $R_i$ to be the complement of a bipartite graph, for each $i$. The eigenvalues of $M_n$ are then given by
$$\lambda_A = \sum_{i} c_i \chi_A(R_i)\quad (A \subset X).$$
Writing $R_i = \overline{B_i}$ for each $i$, where $B_i$ is bipartite, we have
$$\lambda_A = (-1)^{|A|} \sum_i c_i \chi_A(B_i) = (-1)^{|A|} \sum_i c_i \chi_{B_i}(A) \quad (A \subset X).$$
We note that the functions of the form 
$$A \mapsto \sum_i c_i \chi_{B_i}(A)$$
form a linear subspace ($W$, say) of $\mathbb{R}^{\mathbb{Z}_2^X}$: $W$ is precisely the span of the `bipartite characters', i.e.\ the characters of the form $\chi_B$ where $B$ is (the edge-set of) a bipartite graph on $[n]$. Our task reduces to finding a function $f \in W$ such that the function
$$g: A \mapsto (-1)^{e(A)}f(A)$$
satisfies $g(\emptyset) = 1$ and
\begin{equation}
\label{eq:cond} g(A) \geq -1/7 \quad \forall A \subset X:\ A \neq \emptyset.
\end{equation}
A key observation is that for any $i \in \mathbb{N} \cup \{0\}$, the `random cut statistic'
$$q_i(A) = \Prob[\text{a random bipartition of }A \text{ has exactly }i \text{ edges}]$$
lies in the subspace $W$. It also has the convenient property that, for any fixed $i \in \mathbb{N} \cup \{0\}$, $q_i(A)$ tends to zero rather rapidly as $|A|$ increases. Hence, if we take $f$ to be a bounded linear combination of the $q_i$, we only need to worry about boundedly many of the conditions (\ref{eq:cond}). Luckily, we are able to find a bounded linear combination of the $q_i$ with the properties we want; the same linear combination works for all $n \in \mathbb{N}$.

A similar strategy was used recently by Berger and Zhao \cite{bz} to prove an analogous result for $K_4$-intersecting families of graphs.
\begin{theorem}[Berger-Zhao, 2021+]
If $\f \subset \mathcal{P}({[n] \choose 2})$ is a $K_4$-intersecting family of labelled graphs on the vertex-set $[n]$, then $|\f| \leq 2^{{n \choose 2}-6}$. Equality holds only if $\f$ consists of all graphs on $[n]$ containing a fixed $K_4$.
\end{theorem}

The following conjecture of the author, Filmus and Friedgut (from \cite{eff}) remains open for all $t > 4$, however.
\begin{conj}[E.-Filmus-Friedgut, 2012]
\label{conj:eff}
Let $t \geq 3$. If $\f \subset \mathcal{P}({[n] \choose 2})$ is a $K_t$-intersecting family of labelled graphs on the vertex-set $[n]$, then $|\f| \leq 2^{{n \choose 2}-{t \choose 2}}$. Equality holds only if $\f$ consists of all graphs on $[n]$ containing a fixed $K_t$.
\end{conj}
We believe that new ideas (other than those in \cite{eff} and \cite{bz}) will be required to prove Conjecture \ref{conj:eff}.

The following beautiful conjecture of Alon and Spencer \cite{alon-spencer} also remains open.
\begin{conj}[Alon-Spencer, 1990]
\label{conj:as}
There exists an absolute constant $c>0$ such that if $\f \subset \mathcal{P}({[n] \choose 2})$ is a $P_3$-intersecting family of labelled graphs on the vertex-set $[n]$, we have $|\f| \leq (1/2-c)2^{{n \choose 2}}$.
\end{conj}

Here, $P_3$ denotes the path with three edges. Conjecture \ref{conj:as} is appealing because an affirmative answer would determine precisely the graphs $H$ for which the measure of an $H$-intersecting family of graphs can be (uniformly) bounded away from $1/2$: such graphs would be precisely those that are not disjoint unions of stars. (Indeed, it is easy to see, by considering the family of graphs on the vertex-set $[n]$ in which the degrees of at least $(T+t)/2$ of the vertices $1,\ldots,T$ are at least $(n+t-1)/2$, for an appropriate choice of $T$, that for any $t \in \mathbb{N}$ there exists an $H$-intersecting family of $(1/2-o(1))2^{{n \choose 2}}$ labelled graphs on the vertex-set $[n]$, where $H$ is a disjoint union of $t$ stars each with $t$ rays. All graphs that are not disjoint unions of stars, either contain a triangle or a path with three edges.)

We note that Chung, Frankl, Graham and Shearer conjectured in \cite{cfgs} that if $\f \subset \mathcal{P}({[n] \choose 2})$ is a $P_3$-intersecting family of labelled graphs on the vertex-set $[n]$, then $|\f| \leq 2^{{n \choose 2}-3}$. This was disproved in 2008 by Christofides \cite{christofides-personal}, who constructed (for each $n \geq 6$) a $P_3$-intersecting family of labelled graphs on the vertex-set $[n]$, with size $\tfrac{17}{128} 2^{{n \choose 2}} > 2^{{n \choose 2}-3}$. We briefly describe Christofides' construction. It suffices to exhibit a graph $G$ with six vertices and seven edges, together with a $P_3$-intersecting family $\mathcal{G}$ of 17 subgraphs of $G$. (Given such a graph $G$ and such a family $\mathcal{G}$ of subgraphs of $G$, and given an integer $n \geq 6$, fix a copy $G_1$ of $G$ in $K_{[n]}$, and consider the $P_3$-intersecting family $\f$ of all subgraphs $F$ of $K_{[n]}$ such that $F \cap G_1 \in \mathcal{G}$; it is clear that $\f$ is $P_3$-intersecting with $|\mathcal{F}| = \tfrac{17}{128} \cdot 2^{{n\choose 2}}$.) The graph $G$ is constructed as follows. Let $X$ and $Y$ be disjoint sets with $|X|=2$ and $|Y|=3$, and let $G$ be the graph produced by taking the complete bipartite graph $K_{X,Y}$ and adding to it a new vertex $v$ which is joined by an edge ($e$, say), to one vertex of $X$ but to no other vertex of $K_{X,Y}$. Let $\mathcal{A}$ be the family of all subgraphs of $K_{X,Y}$ with five or six edges, and let $\mathcal{B}$ be the family of all subgraphs of $K_{X,Y}$ that are isomorphic to a cycle of length four. Clearly, $|\mathcal{A}| = 7$ and $|\mathcal{B}| = 3$. It is also easy to see that $\mathcal{A}$ is a $P_3$-intersecting family, that any graph $A \in \mathcal{A}$ shares a $P_3$ with any graph $B \in \mathcal{B}$, and that any two distinct graphs $B,B' \in \mathcal{B}$ intersect on a two-edge path whose two endpoints are the two vertices of $X$. It follows that the family
$$\mathcal{G} = \mathcal{A} \cup \{A \cup \{e\}:\ A \in \mathcal{A}\} \cup \{B \cup \{e\}:\ B \in \mathcal{B}\}$$
is a $P_3$-intersecting family of subgraphs of $G$. Clearly, we have $|\mathcal{G}| = 2|\mathcal{A}|+|\mathcal{B}| = 2 \cdot 7 + 3 = 17$, so we are done.

It would be of great interest to characterise the set of unlabelled graphs $H$ such that for any $n \in \mathbb{N}$, a maximum-sized $H$-intersecting family of labelled graphs on $[n]$ can be obtained by taking all graphs containing a fixed copy of $H$: in other words, to characterise the set of unlabelled graphs $H$ such that the Erd\H{o}s-Ko-Rado property holds, for the $H$-intersecting problem. The results above imply that this is the case when $H$ is an edge, a triangle or a $K_4$; it is not the case when $H$ is a path of three edges, or a disjoint union of stars. Conjecture \ref{conj:eff} would imply that it is the case when $H = K_t$ for all $t \geq 5$. In 2012, the authors of \cite{eff} raised the question of whether the Erd\H{o}s-Ko-Rado property holds whenever $H$ is 2-connected, but a construction of Balogh and Linz \cite{linz-personal} shows that when $H = K_{s,t}$ and $t > 2^{2s}-2s-1$, the Erd\H{o}s-Ko-Rado property does not hold for the $H$-intersection problem, so even the condition of $H$ being $s$-connected is insufficient to guarantee the Erd\H{o}s-Ko-Rado property (for any $s \in \mathbb{N}$). 

\subsection{`Graph-intersecting' families of sets}
If $\f \subset \pn$ and $G$ is a graph with vertex-set $[n]$, we say $\f$ is {\em $G$-edge-intersecting} if for any $S,T \in \f$, there exists an edge $e \in E(G)$ such that $e \cap S \neq \emptyset$ and $e \cap T \neq \emptyset$: in other words, either $S$ and $T$ intersect, or else there is an edge of $G$ with one vertex in $S$ and the other in $T$. This is another natural weakening of the condition of being an intersecting family, introduced by Bohman, Frieze, Ruszink\'o and Thoma in \cite{bfrt}. In \cite{bfrt}, the maximum possible size of a $G$-edge-intersecting family $\f \subset {[n] \choose k}$ is determined in several natural cases. For example, in the case where $G = C_n$, i.e.\ the cycle $123\ldots n1$, a family $\f \subset \pn$ is $C_n$-edge-intersecting if and only if for any $S,T \in \f$, there exist $s \in S$ and $t \in T$ such that the cyclic distance between $s$ and $t$ is at most one, or more succinctly, if for any $S,T \in \f$ we have $d(S,T) \leq 1$, where $d(S,T)$ denotes the cyclic distance between the sets $S$ and $T$. For $n \geq ck^4$, where $c$ is an absolute constant, Bohman, Frieze, Ruszink\'o and Thoma determine the maximum possible size of a $C_n$-edge-intersecting family of $k$-element subsets of $[n]$; it is best to take all $k$-element sets containing $2$ or $3$ or both $1$ and $4$. This was later shown to hold for all $n \geq ck^2$ by Bohman and Martin \cite{bm}, and finally for all $n \geq ck$ by Raynaud and the author \cite{er}. Many open questions remain, however; the reader is referred to \cite{jt} for some of these, as well as several elegant results.         

\section{Intersection problems for families of more complicated mathematical objects}
\label{sec:more-complicated}

\subsection{Permutations}
\label{subsection:perms}
As well as considering (uniform or non-uniform) set families, it is natural to pose intersection problems concerning families of more complicated mathematical objects. In 1977, Deza and Frankl \cite{df} considered families of permutations, introducing the following definition. 
\begin{Def}
Let $S_n$ denote the symmetric group on $[n]$, and let $t \in \mathbb{N}$. A family of permutations $\f \subset S_n$ is said to be {\em $t$-intersecting} if for any two permutations $\sigma,\pi \in \f$, we have $|\{i \in [n]:\ \sigma(i)=\pi(i)\}| \geq t$.
 \end{Def}
 In other words, a family of permutations is said to be $t$-intersecting if any two permutations in the family agree on at least $t$ points. A family of permutations is said to be {\em intersecting} if it is 1-intersecting, i.e.\ if any two permutations in the family agree on at least one point.
 
In \cite{df}, Deza and Frankl gave a short proof that for any $n \in \mathbb{N}$, an intersecting family $\f \subset S_n$ has size at most $(n-1)!$. This is best-possible, since $\{\sigma \in S_n:\ \sigma(1)=1\}$ is an intersecting family of this size. Their proof is a partitioning argument: they observe that the cyclic group $H: = \langle \rho \rangle$ generated by any $n$-cycle $\rho \in S_n$, has the property that any two permutations in $H$ disagree everywhere, so $H$ can contain at most one permutation from an intersecting family $\f$. The same is true of any left coset of $H$. The $(n-1)!$ left cosets of $H$ partition $S_n$; since each contains at most one permutation from an intersecting family $\f$, it follows that $|\f| \leq (n-1)!$.
 
 Somewhat unusually, it took a while before the maximum-sized intersecting families of permutations were characterized, but this was done by Cameron and Ku \cite{ck} and independently Larose and Malvenuto \cite{lm} in 2003; they proved that an intersecting family of permutations $\f \subset S_n$ has $|\f| = (n-1)!$ iff $\f = \{\sigma \in S_n:\ \sigma(i)=j\}$ for some $i,j \in [n]$, i.e.\ iff $\f$ is a coset of the stabilizer of a point. The proofs of Cameron and Ku and Larose and Malvenuto are mainly combinatorial; a more algebraic proof was later given by Godsil and Meagher \cite{gm}.

For $t$-intersecting families of permutations, Deza and Frankl conjectured the following, in \cite{df}.
\begin{conj}[Deza-Frankl, 1977]
For any $t \in \mathbb{N}$, if $n$ is sufficiently large depending on $t$, then any $t$-intersecting family $\f \in S_n$ satisfies $|\f| \leq (n-t)!$.
\end{conj}
 
In the case where there exists a sharply $t$-transitive subgroup of $S_n$, essentially the same partitioning argument as that of Deza and Frankl above, implies that a $t$-intersecting subfamily of $S_n$ has size at most $(n-t)!$. Unfortunately, such a subgroup exists only in a small number of cases:
\begin{itemize}
\item when $t=1$ (for any $n \in \mathbb{N}$);
\item when $t=2$ and $n$ is a prime power;
\item when $t=3$ and $n$ is one more than a prime power;
\item when $t=4$ and $n=11$;
\item when $t=5$ and $n=12$;
\item when $t=n-2$;
\item when $t=n$.
\end{itemize}
The Deza-Frankl conjecture remained open in essentially all other cases, until it was proved by the author, and independently and simultaneously by Friedgut and Pilpel, in 2009. Our proofs were very similar indeed, and we wrote a joint paper \cite{efp}. The high-level strategy is analogous to Wilson's proof of Theorem \ref{thm:wilson}: for each $t \in \mathbb{N}$ and each $n \geq n_0(t)$, we construct a pseudoadjacency matrix $M = M_{n,t}$ for the Cayley graph $G_{n,t}$ on $S_n$ generated by $\{\sigma \in S_n:\ \sigma \text{ has less than }t \text{ fixed points}\}$, with eigenvalues that are such as to imply the desired upper bound when the Delsarte-Hoffman bound (Theorem \ref{thm:delsarte-hoffman}) is applied to $M$. (Observe that a $t$-intersecting family $\f \subset S_n$ is precisely an independent set in $G_{n,t}$.) However, unlike in the proof of Wilson's theorem, there is more than just one natural choice of such a pseudoadjacency matrix $M$ (and this makes the construction harder).

To make it easier to analyse the eigenvalues of our matrix $M$, we take $M$ to be a linear combination of adjacency matrices of normal Cayley subgraphs of $G_{n,t}$; this enables us to use tools from non-Abelian Fourier analysis (a.k.a.\ representation theory) to analyse the eigenvalues. Recall that if $\Gamma$ is a group, and $S \subset \Gamma$ with $S^{-1} = S$ and $\text{Id} \notin S$, the {\em Cayley graph of $\Gamma$ with generating set $S$} is the graph with vertex-set $\Gamma$, where $g$ is joined to $gs$ for all $g \in \Gamma$ and $s \in S$; it is denoted by $\Cay(\Gamma,S)$. A Cayley graph $\Cay(\Gamma,S)$ is said to be {\em normal} if $S$ is conjugation-invariant, i.e.\ $gsg^{-1} \in S$ for all $s \in S$ and $g \in G$. It is a well-known fact, due originally to Frobenius, that if $\Gamma$ is a finite group, $\mathcal{R}$ is a complete set of inequivalent irreducible complex representations of $\Gamma$, and $G = \Cay(\Gamma,S)$ is a normal Cayley graph of $\Gamma$, then the eigenvalues of the adjacency matrix of $G$ are given by
\begin{equation}
\label{eq:frobenius}
\lambda_{\rho} = \frac{1}{\dim(\rho)} \sum_{s \in S} \chi_\rho(s)\quad (\rho \in \mathcal{R}),
\end{equation}
where $\chi_\rho$ denotes the character of the representation $\rho$. We use this, together with an intricate analysis of the representations of $S_n$, to engineer a matrix with the appropriate eigenvalues to prove the Deza-Frankl conjecture. 

Our construction only works for $n \geq n_0(t)$ where $n_0(t)$ is doubly exponential in $t$, and it would be of great interest to determine the maximum-sized $t$-intersecting families in $S_n$, for smaller $n$. In \cite{efp}, we conjectured the following, which remains open.
\begin{conj}[E.-Friedgut-Pilpel, 2011]
\label{conj:2t}
For any $n,t \in \mathbb{N}$, a maximum-sized \(t\)-intersecting family in \(S_{n}\) must be a double translate of one of the families
\[\mathcal{F}_{i} := \{\sigma \in S_{n}:\ \sigma \textrm{ has at least } t+i \textrm{ fixed points in } [t+2i]\}\ (0 \leq i \leq (n-t)/2),\]
i.e.\ it must be of the form \(\pi \mathcal{F}_i \tau\), for some \(\pi,\tau \in S_n\).
\end{conj}
Conjecture \ref{conj:2t} would imply that the conclusion of the Deza-Frankl conjecture holds for all \(n > 2t\); this also remains open. (The proof of Theorem \ref{thm:extremal-forbidden}, below, implies that the conclusion of the Deza-Frankl conjecture holds for all $n > e^{C t \log t}$, where $C>0$ is an absolute constant; this is a slight improvement on the doubly exponential bound mentioned above, but is still likely very far from the truth.)

In \cite{el}, Lifshitz and the author study the forbidden intersection problem for permutations; this is a natural analogue of the well-studied forbidden intersection problem for families of sets (see Section \ref{sec:sets}). We prove the following strengthening of the Deza-Frankl conjecture.
 \begin{theorem}[E.-Lifshitz, 2021+]
\label{thm:extremal-forbidden}
If $t \in \mathbb{N}$, $n$ is sufficiently large depending on $t$, and $\f \subset S_n$ contains no two permutations agreeing on exactlty $t-1$ points, then $|\f| \leq (n-t)!$, with equality only if $\f$ consists of a coset of the pointwise-stabilizer of a $t$-element set.
\end{theorem}

Our main tool for proving Theorem \ref{thm:extremal-forbidden} is a structural result, concerning the approximate structure of large families of permutations with a forbidden intersection. To state it, we need some more notation and terminology. If $A,B \subset \{1,2,\ldots,n\}$ with $|A|=|B|$, and $\pi: A \to B$ is a bijection, the \emph{$\pi$-star in $S_n$} is the family of all permutations in $S_n$ that agree with $\pi$ pointwise on all of $A$. An {\em $s$-star} is a $\pi$-star such that $\pi$ is a bijection between sets of size $s$. (Note that an $s$-star is precisely a coset of the pointwise-stabilizer of an $s$-element set.) If for each $i \in [l]$, $A_i,B_i \subset [n]$ and $\pi_i:A_i \to B_i$ is a bijection, we define
$$\langle \pi_1,\ldots,\pi_l \rangle := \{\sigma \in S_n:\ (\exists i \in [l])(\forall j \in A_i)(\sigma(j)=\pi(j))\},$$
i.e.\ $\langle \pi_1,\ldots,\pi_l \rangle$ is the set of all permutations in $S_n$ that agree everywhere with at least one of the bijections $\pi_i$. We say that $\j \subset S_n$ is a {\em $C$-junta} if $\j = \langle \pi_1,\ldots,\pi_l\rangle$ for some bijections $\pi_i :A_i \to B_i$, where $l \leq C$ and $|S_i| \leq C$ for all $i \in [l]$. We may think of $C$ as (an upper bound on) the `complexity' of the junta $\j$. We note that this definition of a junta is a natural analogue of the definition of a junta of subsets of $[n]$ (see page \pageref{junta}).

We can now state our `junta approximation' result.

\begin{theorem}[E.-Lifshitz, 2021+]
\label{thm:junta-approximation}
For any $r,t \in \mathbb{N}$, there exists $C=C(r,t) \in \mathbb{N}$ such that if $\f \subset S_n$ is $(t-1)$-intersection-free, there exists a $t$-intersecting $C$-junta $\j \subset S_n$ such that $|\f \setminus \j| \leq Cn!/n^r$.
\end{theorem}

Informally, this theorem says that any $(t-1)$-intersection-free family is `almost' contained within a $t$-intersecting junta of bounded complexity. Its use here is a good example of the `junta method', which has proven very useful in extremal combinatorics and theoretical computer science over the last 30 years. This method was first introduced into extremal combinatorics by Dinur and Friedgut \cite{dinur-friedgut} in 2008, and was further developed significantly by Keller and Lifshitz in \cite{keller-lifshitz}. The high-level idea of the method is as follows. Suppose we wish to prove an extremal theorem concerning families of mathematical objects satisfying a certain property, $P$ say. We take such a family, and we first show that it can be approximated by a `junta' of bounded complexity (where the notion of `junta' depends on the problem, but means roughly a family depending upon only a bounded number of coordinates). We then obtain an extremal result about such a junta (this is usually easy), and then finally we use a (usually combinatorial) `perturbation' argument to obtain the desired result about general families possessing our property $P$ (which, by the our junta approximation result, must be close to a junta).

Our proof of Theorem \ref{thm:junta-approximation} employs a mixture of combinatorial, probabilistic and algebraic techniques. Specifically, it relies on (i) a weak regularity lemma for families of permutations (which outputs a junta whose stars are intersected by $\f$ in a weakly pseudorandom way), (ii) a combinatorial argument that `bootstraps' the weak notion of pseudorandomness into a stronger one, and finally (iii) a spectral argument for pairs of highly-pseudorandom fractional families (this spectral argument being significantly shorter than the spectral argument in \cite{efp}, though still non-trivial). Our proof employs four different notions of pseudorandomness, three being combinatorial in nature, and one being algebraic. We believe the connection we demonstrate between these combinatorial and algebraic notions of pseudorandomness may find further applications.

We note that arguments involving pseudorandomness (or quasirandomness), in various forms, have had a huge impact on combinatorics, theoretical computer science and number theory, ever since Szemer\'edi proved his celebrated regularity lemma for graphs, in 1978. The common theme of such arguments is that many mathematical structures can be partitioned into a bounded number of large pieces, together with a small `leftover' piece, such that any $r$ of the large pieces induce a structure that is `random-like' (pseudorandom, or quasirandom), in an appropriate sense. (The right notion of pseudorandomness, and the right value of $r$, depends upon the problem.) For surveys of applications of pseudorandomness and regularity methods in combinatorics, theoretical computer science and additive number theory, the reader is referred for example to \cite{kro,ksss,ks,luczak}.

Theorem \ref{thm:junta-approximation} (together with a short combinatorial argument) also quickly implies the following Hilton-Milner type result for $t$-intersecting families of permutations, first proved by the author \cite{ellis-dfs} in 2009.
\begin{theorem}[E., 2009]
If \(n\) is sufficiently large depending on \(t\), and $\f \subset S_n$ is a $t$-intersecting family of permutations which is not contained within a coset of the pointwise-stabiliser of a $t$-element set, then $|\f| \leq |\mathcal{H}|$, where
\begin{align*}
\mathcal{H} & = \{\sigma \in S_n: \sigma(i)=i \text{ for all } i \leq t,\ \sigma(j)=j \textrm{ for some } j > t+1\} \\
&\cup \{(1\ t+1),(2\ t+1),\ldots,(t\ t+1)\}.
\end{align*}
Equality holds if only if there exist $\sigma,\tau \in S_n$ such that $\f = \sigma \mathcal{H} \tau$.
\end{theorem}
This is a natural analogue, for permutations, of Theorem \ref{thm:aknontriv}. 

\subsection{Intersection problems for more general group actions}
If $\alpha:G \times X \to X$ is a transitive action of a finite group $G$ on a finite set $X$, we say that a subset $\mathcal{F} \subset G$ is {\em $\alpha$-intersecting} if for any two elements $\sigma,\tau \in \mathcal{F}$, there exists $x \in X$ such that $\alpha(\sigma,x) = \alpha(\tau,x)$. One can ask, for each action $\alpha:G \times X \to X$, what is the maximum possible size of an $\alpha$-intersecting subset of $G$. The Deza-Frankl problem discussed in the previous section, is clearly of this form, with $G$ being $S_n$ and $\alpha$ being the natural action of $S_n$ on ordered $t$-tuples. Several other well-studied intersection problems are also of this form. In fact, by quotienting out by the kernel of the action $\alpha$, one can reduce to the case where $G$ is a subgroup of $S_n$ (i.e.\ a permutation group of degree $n$), and the action is the natural action of $G$ on $[n]$; but this reformulation may be less natural than the original problem, in many cases. A well-studied set of problems comes from taking $G$ to be a group of matrices over $\mathbb{F}_q$ (or a quotient thereof), such as $\text{GL}(n,\mathbb{F}_q)$, $\text{SL}(n,\mathbb{F}_q)$ or $\text{PGL}(n,\mathbb{F}_q)$, and taking $\alpha$ to be the natural action of $G$ on $d$-dimensional subspaces or $d$-dimensional projective subspaces, for $d < n$. For example, we say a family $\f \subset \text{PGL}(n+1,\mathbb{F}_q)$ is {\em point-intersecting} if for any $\sigma,\tau \in \f$ there exists a projective point $p \in \text{PG}(n,\mathbb{F}_q)$ such that $\sigma(p) = \tau(p)$. Meagher and Spiga \cite{mspiga,mspiga2} proved the following.
\begin{theorem}[Meagher-Spiga, 2011]
If $q$ is a prime power, and $\f \subset \text{PGL}(2,\mathbb{F}_q)$ is point-intersecting, then $|\f| \leq q(q-1)$; equality holds iff $\f$ is a coset of the stabilizer of a projective point. 
\end{theorem}
\begin{theorem}[Meagher-Spiga, 2014]
If $q$ is a prime power, and $\f \subset \text{PGL}(3,\mathbb{F}_q)$ is point-intersecting of maximum size, then $\f$ is either a coset of the stabilizer of a projective point or a coset of the stabilizer of a projective line.
\end{theorem}
They conjecture the following (in \cite{mspiga}). 
\begin{conj}[Meagher-Spiga, 2011]
\label{conj:ms}
If $q$ is a prime power and $\f \subset \text{PGL}(n+1,\mathbb{F}_q)$ is point-intersecting of maximum size, then $\f$ is either a coset of the stabilizer of a projective point or a coset of the stabilizer of an $(n-1)$-dimensional projective hyperplane.
\end{conj}
This remains open, to the best of our knowledge, for all $n > 2$.

If $V$ is a finite-dimensional vector space, we let $\text{GL}(V)$ denote the general linear group over $V$, i.e.\ the group of all invertible linear maps from $V$ to itself. We say a family $\f \subset \text{GL}(V)$ is {\em $(t-1)$-intersection-free} if for any $\sigma,\tau \in \f$, the subspace $\{v \in V:\ \sigma(v)=\tau(v)\}$ does not have dimension $t-1$. Using similar techniques to in the proof of Theorem \ref{thm:extremal-forbidden}, together with a hypercontractivity result for Boolean functions on spaces of linear maps, Kindler, Lifshitz and the author recently proved the following \cite{ekl-prep}.
\begin{theorem}[E.-Kindler-Lifshitz, 2021+]
For any $t \in \mathbb{N}$ and any prime power $q$, there exists $n_0 = n_0(q,t) \in \mathbb{N}$ such that the following holds. If $n \in \mathbb{N}$ with $n \geq n_0$, $V$ is an $n$-dimensional vector space over $\mathbb{F}_q$, and $\f \subset \text{GL}(V)$ is $(t-1)$-intersection-free, then
$$|\f| \leq \prod_{i=1}^{n-t}(q^n - q^{i+t-1}).$$
Equality holds only if there exists a $t$-dimensional subspace $U$ of $V$ on which all elements of $\f$ agree, or a $t$-dimensional subspace $A$ of $V^*$ on which all elements of $\{\sigma^*:\ \sigma \in \f\}$ agree.
\end{theorem}

The following elegant general result was obtained by Meagher, Spiga and Tiep in 2015 \cite{mst-2-trans}. 
\begin{theorem}[Meagher-Spiga-Tiep, 2015]
\label{thm:mst}
Let $H \leq S_n$ be a 2-transitive permutation group of degree $n$, and let $\f \subset H$ be an intersecting family of permutations in $H$. Then $|\f| \leq |H|/n$.
\end{theorem}
Theorem \ref{thm:mst} says that the Erd\H{o}s-Ko-Rado property holds for intersecting subsets of 2-transitive permutation groups. The proof uses the Delsarte-Hoffman bound, applied to the derangement graph of $H$ (the normal Cayley graph of $H$ generated by the derangements of $H$), together with (\ref{eq:frobenius}), and an intricate analysis of the character theory of 2-transitive groups.

It would be of interest to determine other (similarly general) sufficient conditions on permutation groups, for the conclusion of Theorem \ref{thm:mst} to hold. It is easy to see that it does not hold in general for transitive primitive permutation groups, e.g.\ by considering the $t$-intersecting family of permutations
$$\{\sigma \in S_n:\ \sigma \text{ has at least }t+i \text{ fixed points in } [t+2i]\},$$
which appears in Conjecture \ref{conj:2t}, for appropriate $n,t,i \in \mathbb{N}$, and viewing $S_n$ as a transitive subgroup of $S_{n(n-1)\ldots(n-t+1)}$ via the action of $S_n$ on $t$-tuples of distinct points. Examples of transitive primitive permutation groups where the Erd\H{o}s-Ko-Rado property (for intersecting subsets) fails by a larger multiplicative factor can be found in \cite{li}. In \cite{li}, it is proved that transitive permutation groups of prime power degree (or of prime power order) satisfy the Erd\H{o}s-Ko-Rado property for intersecting subsets; this resolved a conjecture of Meagher, Razafimahatratra and Spiga \cite{mrs}. It is an interesting open problem, raised in \cite{mrs}, to determine the behaviour of the function
\begin{align*} I(n): = \max\{n|\f|/|H|:\ & H \text{ is a transitive permutation group of degree }n,\\
& \f \subset H \text{ is intersecting}\},\end{align*}
i.e.\ to determine the largest possible multiplicative factor by which the conclusion of Theorem \ref{thm:mst} can fail, if the hypothesis of 2-transitivity is weakened to mere transitivity. It was conjectured in 2020 by Li, Song and Pantangi \cite{lsp-orig} that $I(n) < \sqrt{n}$ all $n$; counterexamples to this (for $n=18,30$) were then given by Meagher, Razafimahatratra and Spiga in \cite{mrs}, after which Li, Song and Pantangi showed in \cite{li} that $I(n) \geq n^{\ln 10 / \ln 30} > n^{0.67699}$ for infinitely many $n$. It would be interesting to determine
$$\limsup_{n \to \infty} \frac{\ln I(n)}{\ln n}.$$

The following elegant conjecture was posed by J\'anos K\"orner in 2009, regarding the natural action of the symmetric group on $t$-element subsets.

\begin{conj}
Let $t \in \mathbb{N}$ and let $\f \subset S_n$ be a family of permutations such that for any $\sigma,\pi \in \f$, there exists a $t$-element set $T \in {[n] \choose t}$ such that $\sigma(T)=\pi(T)$. Then $|\f| \leq (n-t)!$.
\end{conj}
This was proved by the author in 2012 \cite{ellis-setwise} for all $n$ sufficiently large depending on $t$, using a similar strategy to in the proof of the Deza-Frankl conjecture in \cite{efp}. It was also recently proved for $t=2$ (and all $n$) by Meagher and Razafimahatratra \cite{mr}. It remains open in full generality.

\subsection{Partitions}
In this section, a {\em partition of $[n]$} (into $k$ sets) is a family of (exactly $k$) nonempty sets $\{S_1,\ldots,S_k\}$ such that $S_i \subset [n]$ for all $i$, $S_i \cap S_j = \emptyset$ for all $i \neq j$, and $S_1 \cup S_2 \cup \ldots \cup S_k = [n]$. The sets $S_i$ are called the {\em parts} of the partition.

If $t \in \mathbb{N}$, we say a family of partitions of $[n]$ is {\em $t$-intersecting} if any two partitions in the family have at least $t$ parts in common. If $n \in \mathbb{N}$, we write $\mathcal{B}(n)$ for the set of all partitions of $[n]$; recall that $B_n:= |\mathcal{B}(n)|$ is the $n$th {\em Bell number}.

If $k,n \in \mathbb{N}$ with $k \leq n$, we write $P_{n}^{k}$ for the set of all partitions of $[n]$ into $k$ sets, and for $t \in \mathbb{N}$ with $t \leq k$ we let
$$\mathcal{P}(n,k,t) := \{P \in P_n^k:\ \{\{1\},\{2\},\ldots,\{t\}\} \subset P\}.$$
If $k,n \in \mathbb{N}$ with $k \mid n$, we let $U_n^k$ denote the set of all partitions of $[n]$ into $k$ equal-sized sets, each of size $n/k$, and writing $c=n/k$, for $t \leq k$ we let
$$\mathcal{Q}(n,k,t) := \{P \in U_n^k:\ \{[c],\{c+1,\ldots,2c\},\ldots,\{(t-1)c+1,\ldots,tc\}\} \subset P\}.$$
An element of $U_n^k$ (for some $n,k$) is called a {\em uniform set partition}.

Ku and Renshaw proved the following in \cite{kr}.
\begin{theorem}[Ku-Renshaw, 2008]
\label{thm:kr1}
Let $n \geq 2$. Suppose $\f \subset \mathcal{B}(n)$ is 1-intersecting. Then $|\f| \leq |\mathcal{B}(n-1)|$, and equality holds iff $\f$ consists of all partitions of $[n]$ containing a fixed singleton.
\end{theorem}

The proof of Theorem \ref{thm:kr1} relies on a strategy of combinatorial shifting/compressions; see our Introduction (Section 1), for a description of this general strategy.

For $t$-intersecting families, Ku and Renshaw \cite{kr} proved the following.

\begin{theorem}[Ku-Renshaw, 2008]
\label{thm:krt}
Let $t \geq 2$ and $n \geq n_0(t)$. Suppose $\f \subset \mathcal{B}(n)$ is $t$-intersecting. Then $|\f| \leq |\mathcal{B}(n-t)|$, and equality holds iff $\f$ consists of all partitions containing $t$ fixed singletons.
\end{theorem}
It should be noted that some condition of the form $n \geq n_0(t)$ in Theorem \ref{thm:krt} is necessary, as can be seen by considering the family of all partitions of $[t+4]$ that have at least $t+2$ singletons, for $t \geq 2$ (there are $(t^2+7t+14)/2 \geq 16$ such, whereas $|\mathcal{B}(4)|=B_4=15$). For $t \geq 2$ and small $n$ (depending on $t$), very little seems to be known.

For partitions into a fixed number of sets, P.L. Erd\H{o}s and L.A. Sz\'ekely \cite{es} proved the following.
\begin{theorem}[P.L. Erd\H{o}s, L.A. Sz\'ekely, 1998]
Let $n,k,t \in \mathbb{N}$ with $k \leq n$ and $n \geq n_0(k,t)$, and let $\f \subset P_n^k$ be an intersecting family of partitions of $[n]$ into $k$ sets. Then $|\f| \leq |\mathcal{P}(n,k,t)|$.
\end{theorem}
As with the case of Theorem \ref{thm:krt}, very little seems to be known for small $n$.

For families of uniform set-partitions, Meagher and Moura proved the following in \cite{mm}.
\begin{theorem}[Meagher-Moura, 2005]
Let $n,k \in \mathbb{N}$ with $k \mid n$, and let $\f \subset U_n^k$ be an intersecting family of partitions of $[n]$ into $k$ equal-sized sets. Then $|\f| \leq |\mathcal{Q}(n,k,1)|$, with equality iff $\f$ is equal to $\mathcal{Q}(n,k,1)$ up to a permutation of $[n]$.
\end{theorem}

For $t$-intersecting families of uniform set partitions, they proved the following \cite{mm}.

\begin{theorem}[Meagher-Moura, 2005]
Let $t \in \mathbb{N}$ If $k \mid n$, $c: = n/k$ and $n \geq n_0(k,t)$ or ($n \geq k(t+2)$ and $n \geq n_1(t,c)$) and $\f$ is a $t$-intersecting family of partitions of $[n]$ into $k$ disjoint sets, then $|\f| \leq |\mathcal{Q}(n,k,t)|$, with equality iff $\f$ is equal to $\mathcal{Q}(n,k,t)$ up to a permutation of $[n]$.
\end{theorem}

They also pose (in \cite{mm}) a conjecture for $t$-intersecting families of uniform set partitions which is a natural analogue of Ahlswede and Khachatrian's complete intersection theorem. This conjecture remains completely open, to the best of our knowledge.

A weaker (but very natural) notion of intersecting partitions was considered by Czabarka, P.L. Erd\H{o}s and L.A. Sz\'ekely (see \cite{es,mm}). For $t \in \mathbb{N}$, we say two partitions $P_1$ and $P_2$ of $[n]$ {\em partially $t$-intersect} if there exist parts $S_1 \in P_1$ and $S_2 \in P_2$ such that $|S_1 \cap S_2| \geq t$. (Note that any two partitions of a nonempty set are partially 1-intersecting, so this notion is only interesting for $t \geq 2$.) For $t \geq 2$, we say a family of partitions is {\em partially $t$-intersecting} if any two partitions in the family partially $t$-intersect. Meagher and Moura conjecture the following.
\begin{conj}[Meagher-Moura, 2005]
Let $n,k,t \in \mathbb{N}$ with $t \geq 2$, $k \mid n$ and $n \geq kt$, and let $\f \subset U_n^k$ be a partially $t$-intersecting family of partitions of $[n]$ into $k$ equal-sized sets. Then
$$|\f| \leq |\{P \in U_n^k:\ [t] \text{ is contained in some part of } P\}|.$$
with equality iff $\f$ is equal to the above family up to a permutation of $[n]$.
\end{conj}
This conjecture also remains completely open, to the best of our knowledge.

\subsection*{Families of subspaces of a vector space}
If $q$ is a prime power, and $V$ is an $n$-dimensional vector space over the field $\mathbb{F}_q$, we say a family $\f$ of subspaces of $V$ is {\em $t$-intersecting} if $\dim(S \cap S') \geq t$ for all $S,S' \in \f$. Frankl and Wilson \cite{fwvs} considered the question of the maximum possible size of such a family; they gave a complete answer in the following theorem.

\begin{theorem}[Frankl-Wilson, 1986]
\label{thm:fwvs}
Let $q$ be a prime power, let $n,k,t \in \mathbb{N}$ with $n \geq 2k-t$, let $V$ be an $n$-dimensional vector space over $\mathbb{F}_q$, and let $\f$ be a $t$-intersecting family of $k$-dimensional subspaces of $V$. Then
$$|\f| \leq \max\left\{{n-t \brack k-t}_q,{2k-t \brack k}_q\right\}.$$
\end{theorem}

\noindent (Here, for integers $0\leq r \leq m$ and $q$ a prime power, we define
$${m \brack r}_q: = \frac{(q^m-1)(q^{m-1}-1) \cdots (q^{m-r+1}-1)}{(q^r -1)(q^{r-1}-1) \cdots (q-1)};$$
these are known as the {\em Gaussian binomial coefficients}.)

We note that for $n \leq 2k-t$, any two $k$-dimensional subspaces of an $n$-dimensional vector space $V$ have intersection of dimension at least $t$, so the question is trivial for this range. In fact, a short argument of Frankl and Wilson shows that to prove Theorem \ref{thm:fwvs}, it suffices to consider the case $n \geq 2k$, since if $2k-t \leq n \leq 2k$ and $\f \subset {V \brack k}$ is $t$-intersecting, then $\{S^{\perp}:\ S \in \f\} \subset {V \brack n-k}$ is $(n-2k+t)$-intersecting, since for any $S,T \in \f$, we have
\begin{align*} \dim(S^{\perp} \cap T^{\perp}) &= \dim((S+T)^{\perp})\\
& = n-\dim(S+T)\\
&= n-\dim(S)-\dim(T)+\dim(S\cap T)\\
& \geq n-2k+t.
\end{align*}

The proof of Theorem \ref{thm:fwvs} (for $n \geq 2k$), follows exactly the same strategy as Wilson's proof of Theorem \ref{thm:wilson}: the proof is spectral, relying on the construction of an appropriate pseudoadjacency matrix. In fact, the calculations are easier than in the proof of Wilson's theorem, due to the rapid growth of the Gaussian binomial coefficients.

\subsection{Triangulations, and a conjecture of Kalai and Meagher}
Let $\mathscr{P}_n$ be a fixed, convex, $n$-vertex polygon. We say a family $\f$ of triangulations of $\mathscr{P}_n$ is {\em intersecting} if any two triangulations in $\f$ share a diagonal. We let $C_n = {2n \choose n}/(n+1)$ denote the $n$th Catalan number; it is well-known that for each $n \geq 3$, the number of triangulations of $\mathscr{P}_n$ is $C_{n-2}$. Kalai and Meagher (independently) conjectured the following the following in 2012.
\begin{conj}[Kalai, Meagher, 2012]
\label{conj:kalai}
Let $\mathscr{P}_n$ be a fixed, convex, $n$-vertex polygon, and let $\f$ be an intersecting family of triangulations of $\mathscr{P}_n$. Then $|\f| \leq C_{n-3}$, with equality iff $\f$ consists of all triangulations containing some fixed diagonal that forms a triangle with two consecutive edges of $\mathscr{P}_n$.
\end{conj} 
This conjecture remains completely open. It is attractive because, although the conjectured extremal families are those of the form `all triangulations containing a fixed diagonal' (so, conjecturally, the Erd\H{o}s-Ko-Rado property holds), fixing different types of diagonal leads to different-sized intersecting families; a short calculation demonstrates that diagonals of the type in the Kalai-Meagher conjecture are the best ones to fix.

It is easy to show that an intersecting family $\f$ of triangulations of $\mathscr{P}_n$ has size at most $C_{n-2}/2$. (Take $\mathscr{P}_n$ to be a regular polygon, and observe that a triangulation $T$ shares no diagonal in common with $\sigma(T)$, where $\sigma$ is the rotation by $2\pi/n$ about the centre of $\mathscr{P}_n$.) This is roughly twice the conjectured bound, since $C_{n-3} = (1/4+O(1/n))C_{n-2}$. To our knowledge, no significant improvement on this `trivial' bound, is known.

Several conjectures generalising Conjecture \ref{conj:kalai} are posed by Olarte, Santos, Spreer and Stump, in \cite{olarte}.

\subsection{Down-sets, and conjectures of Chv\'atal and Kleitman}

 A very well-known open problem comes from considering intersecting families in down-sets. We say a family $\mathcal{B} \subset \pn$ is a {\em down-set} if it is closed under taking subsets, i.e.\ if whenever $B \in \mathcal{B}$ and $B' \subset B$, we have $B' \in \mathcal{B}$. A family $\f \subset \pn$ is said to be a {\em star} if $\cap_{S \in \f}S \neq \emptyset$. A celebrated conjecture of Chv\'atal \cite{chvatal} from 1974 states that an intersecting subfamily of a down-set is no larger than the largest star of the down-set; in other words, the Erd\H{o}s-Ko-Rado property holds for intersecting families in down-sets.
 
 \begin{conj}[Chv\'atal, 1974]
 Let $n \in \mathbb{N}$, and let $\mathcal{B} \subset \pn$ be a down-set. If $\f \subset \mathcal{B}$ is intersecting, then there exists $i \in [n]$ such that
 $$|\f| \leq |\{B \in \mathcal{B}:\ i \in B\}|.$$
 \end{conj}

This conjecture remains completely open, though some partial results are known. Berge \cite{berge} proved that if $\mathcal{B} \subset \pn$ is a down-set, then either $\mathcal{B}$ or $\mathcal{B} \setminus \{\emptyset\}$ may be partitioned into pairs of disjoint sets, which implies that if $\f \subset \mathcal{B}$ is intersecting, then $|\f| \leq |\mathcal{B}|/2$. (A particularly elegant proof of this was given by Daykin, Hilton and Mikl\'os in \cite{dhm}.) Chv\'atal himself verified his conjecture in the case where $\mathcal{B}$ is left-compressed, and Snevily \cite{snevily} proved it in the case where there exists $i \in [n]$ such that $\mathcal{B}$ is $ij$-compressed for all $j \neq i$ (making use of the partitioning result of Berge, above).

Kleitman \cite{kleitman} posed the following strengthening of Chv\'atal's conjecture, which is attractive because it does not (explicitly) involve down-sets. To state it, we need to recall some more definitions.

\begin{Def}
If $f,g : \mathcal{P}([n]) \to \mathbb{R}_{\geq 0}$, we say $f$ {\em flows down to} $g$ if there exists $v: \pn \times \pn \to \mathbb{R}_{\geq 0}$ such that
\begin{itemize}
\item For any $A \subset [n]$, we have $\sum_{B \subset [n]} v(A,B) = f(A)$;
\item For any $B \subset [n]$, we have $\sum_{A \subset [n]} v(A,B) = g(B)$;
\item For any $B \not\subset A$, we have $v(A,B)=0$.
\end{itemize}
\end{Def}

Equivalently, via max-flow min-cut, $f$ flows down to $g$ if and only if $\sum_{A \subset [n]} f(A) = \sum_{A \subset [n]}g(A)$ and for any down-set $\mathcal{B} \subset \pn$ we have
$$\sum_{B \in \mathcal{B}}f(B) \leq \sum_{B \in \mathcal{B}} g(B).$$

\begin{Def}
A family $\mathcal{A} \subset \pn$ is said to be an {\em up-set} if it is closed under taking supersets, i.e.\ if whenever $A \in \mathcal{A}$ and $A \subset A'$, we have $A' \in \mathcal{A}$.
\end{Def}

\begin{Def} A family $\mathcal{F} \subset \pn$ is said to be {\em antipodal} if $\mathcal{F}$ contains exactly one of $S$ and $[n] \setminus S$, for all $S \subset [n]$.
\end{Def}

\begin{Def}
A Boolean function $f:\pn \to \{0,1\}$ is said to be a {\em dictatorship} if there exists $i \in [n]$ such that $f(S) = 1$ iff $i \in S$.
\end{Def}

\begin{conj}[Kleitman, 1979]
If $\f \subset \pn$ is a maximal intersecting family, then there exist $\lambda_1,\ldots,\lambda_n \geq 0$ with $\sum_{i=1}^{n} \lambda_i=1$, such that $1_{\f}$ flows down to $\sum_{i=1}^{n} \lambda_i 1_{\{S \subset [n]:\ i \in S\}}$.
\end{conj}

It is easy to see, and well-known, that a maximal intersecting subfamily of $\pn$ is precisely an antipodal up-set, so Kleitman's conjecture can be restated as saying that (the characteristic function of) an antipodal up-set flows down to a convex linear combination of dictatorships.

It is clear that Kleitman's conjecture implies Chv\'atal's conjecture. In a more recent work \cite{fkkk}, Friedgut, Kahn, Kalai, and Keller give several elegant correlation inequalities that imply Chv\'atal's or Kleitman's conjecture; unfortunately, all of these conjectures remain unproven, though some partial results are given in \cite{fkkk}.

\section{Imposing extra symmetry constraints}
\label{sec:sym}
As we have seen, the extremal families in intersection problems are often (though not always) junta, in the sense that they depend only upon a bounded number of coordinates; sometimes the near-extremal families are also junta-like. In particular, they are far from being symmetric. It is therefore natural to ask what happens in intersection problems when we impose the additional requirement that the family be symmetric. (We say a family of subsets $\f \subset \pn$ is {\em symmetric} if its automorphism group is a transitive subgroup of $S_n$; recall that for a family of subsets $\f \subset \pn$, we define its automorphism group to be the group of all permutations of $[n]$ that preserve the family, i.e.\ $\text{Aut}(\f) := \{\sigma \in S_n: \sigma(S) \in \f \text{ for all } S \in \f\}$.)

A first natural question to ask, in this vein, is the following.
\begin{qn}
\label{question:sym0}
 For $n \in \mathbb{N}$, what is the maximum possible size of a symmetric intersecting family of subsets of $[n]$?
\end{qn}
For $n$ odd, this is rather easy: the family $\{S \subset [n]:\ |S| > n/2\}$ is a symmetric intersecting family of size $2^{n-1}$, and no intersecting family can be larger than this, as we saw on page \pageref{triv}. For $n$ even, we have
$$|\{S \subset [n]: |S| > n/2\}| = 2^{n-1} - \tfrac{1}{2}{n \choose n/2} = (1-O(1/\sqrt{n})) 2^{n-1},$$
so the answer is certainly $(1-o(1))2^{n-1}$, but for most even values of $n$, the exact answer is unknown. Indeed, even the set
$$A: = \{n \in \mathbb{N}:\ \exists \text{ a symmetric intersecting family } \f \subset \pn \text{ with } |\f| = 2^{n-1}\}$$
has not been fully characterised. It contains infinitely many even numbers (see \cite{pjc}), as well as all the odd numbers. Isbell conjectured in 1960 that there exists a function $f: \mathbb{N} \to \mathbb{N}$ such that for $a,b \in \mathbb{N}$ with $b$ odd and $a \geq f(b)$, we have $2^a \cdot b \in A$, but this is still open (see \cite{pjc,isbell}).

Still, the answer to Question \ref{question:sym0} is (asymptotically) not very different to that of Question \ref{question:power-set}, even for $n$ even. Frankl considered  the following variation on Question \ref{question:sym0}, where the situation is much less clear. For an integer $r \geq 2$, we say a family of sets is {\em $r$-wise-intersecting} if any $r$ of the sets in the family have nonempty intersection. Frankl posed the following question: for $n \in \mathbb{N}$, what is the maximum possible size of a symmetric $3$-wise intersecting family of subsets of $[n]$? He made the following.
\begin{conj}[Frankl \cite{frankl-sym}, 1981]
\label{conj:frankl-81}
If $\f \subset \pn$ is a symmetric, 3-wise intersecting family of subsets of $[n]$, then $|\f| = o(2^n)$.
\end{conj}
To motivate Conjecture \ref{conj:frankl-81}, note that $\{S \subset [n]: 1 \in S\}$ is 3-wise-intersecting but very far from being symmetric, whereas $\{S \subset [n]: |S| > 2n/3\}$ is 3-wise-intersecting (just), but has exponentially small size (as a fraction of $2^n$).

In \cite{frankl-sym} Frankl proved his conjecture under the stronger hypothesis of $\f$ being 4-wise-intersecting, but until recently, it remained open in general. In 2017, Conjecture \ref{conj:frankl-81} was proved by Narayanan and the author \cite{en}. The proof is very short indeed, and is perhaps one of the shortest applications of the $p$-biased measure on $\pn$ being used to solve an extremal problem whose statement does not mention any biased measure. The idea as follows. Let $\f \subset \pn$ be a symmetric, 3-wise-intersecting family of the maximum possible size; then $\f$ must be an up-set, meaning that it is closed under taking supersets. Consider now the function $p \mapsto \mu_p(\f)$. It is well-known that this is a monotone non-decreasing function of $p$ (for any up-set $\f \subset \pn$); moreover, the following classical result of Friedgut and Kalai (based upon the celebrated Kahn-Kalai-Linial theorem \cite{kkl} on the influences of Boolean functions) says that this function has a `sharp threshold', in the sense that it jumps from near-zero to near-one, over a short interval.
\begin{theorem}[Friedgut-Kalai, 1996]
\label{thm:fk}
There exists a universal constant $c_0>0$ such that the following holds for all $n \in \N$. Let $0 < p,\epsilon < 1$ and let $\mathcal{F} \subset \pn$ be a symmetric up-set. If $\mu_p(\mathcal{F})> \epsilon$, then $\mu_q(\f) > 1-\epsilon$, where
\[q = \min\left\{1,p + c_0\left(\frac{\log(1/2\epsilon)}{ \log n}\right)\right\}.\]
\end{theorem}

We now observe two facts. Firstly, letting $\mathcal{I}(\f): = \{S \cap T: S,T \in \f\}$ denote the family of all pairwise intersections of sets in $\f$, we observe that if $\f$ is 3-wise-intersecting, then $\f$ and $\mathcal{I}(\f)$ are cross-intersecting, meaning that for any $A \in \f$ and $B \in \mathcal{I}(\f)$, we have $A \cap B \neq \emptyset$. It follows that $\mathcal{I}(\f)$ is contained within the dual family $\f^{*}: = \{[n] \setminus S:\ S \notin \f\}$, and therefore
$$\mu_p(\mathcal{I}(\f)) \leq \mu_p(\f^*) = 1-\mu_{1-p}(\f) \quad \forall p \in [0,1].$$
In particular, setting $p = 1/4$ we have
\begin{equation}
\label{eq:dual}
\mu_{3/4}(\f)+\mu_{1/4}(\mathcal{I}(\f)) \leq 1.
\end{equation}
Secondly, for all $p \in [0,1]$ we have
\begin{align*}
\mu_{p^2}(\mathcal{I}(\f)) &= \Pr_{A_1,A_2 \sim \mu_{p}}[A_1 \cap A_2 \in \mathcal{I}(\f)]\\
& \geq \Pr_{A_1,A_2 \sim \mu_{p}}[A_1 \in \f,\ A_2 \in \f]\\
& = \left(\Pr_{A_1 \sim \mu_{p}}[A_1 \in \f]\right)^2\\
& = (\mu_{p}(\f))^2,
\end{align*}
where the notation $A_1,A_2 \sim \mu_{p}$ means that $A_1$ and $A_2$ are chosen independently at random according to the $p$-biased probability measure on $\pn$. In particular, setting $p=1/2$ we have
\begin{equation}\label{eq:int-prod} \mu_{1/4}(\mathcal{I}(\f)) \geq (\mu_{1/2}(\f))^2.\end{equation}
Combining (\ref{eq:dual}) and (\ref{eq:int-prod}), we have
\begin{equation}
\label{eq:comb}
\mu_{3/4}(\f) + (\mu_{1/2}(\f))^2 \leq 1.
\end{equation}
Combining this with Theorem \ref{thm:fk}, it is easy to see that we must have $|\f|/2^n = \mu_{1/2}(\f) < n^{-1/(8c_0)}$, proving Conjecture \ref{conj:frankl-81}. (Intuitively, if $\mu_{1/2}(\f)$ were greater than $n^{-1/(8c_0)}$, then the `sharp jump' guaranteed by Theorem \ref{thm:fk} takes place just after $p=1/2$, which would contradict (\ref{eq:comb}).)

Our proof of Frankl's conjecture gives $|\f| \leq 2^n/n^{c}$, for $c>0$ an absolute constant. It is likely that this is far from the truth. We make the following conjecture in \cite{en}.
\begin{conj}[E.-Narayanan, 2017]
\label{conj:best}
If $\f \subset \pn$ is a symmetric $3$-wise intersecting family, then 
\[|\f| \le 2^{n - cn^{\delta}},\]
where $c,\delta>0$ are universal constants. 
\end{conj}
This would be best-possible up to the values of $c$ and $\delta$, as evidenced by the following construction communicated to us by Oliver Riordan. Let $k$ be an odd integer and let $n=k^2$, partition $[n]$ into $k$ `blocks' $B_1, B_2, \ldots,B_k$ each of size $k$, and take $\f \subset \pn$ to be the family of all those subsets of $[n]$ that contain more than half the elements in each block and all the elements in some block; in other words, 
\[\f = \{S \subset [n]: |S \cap B_i| > k/2 \text{ for all } i \in [k], \text{ and } B_j \subset S \text{ for some }j \in [k]\}.\]
It is easy to see that $\f$ is symmetric and $3$-wise intersecting, and that
\[\log_2 |\f| = n - 2\sqrt{n} + o(\sqrt{n}).\] 
It is fairly straightforward to generalise Riordan's construction to show that, for any integer $r \geq 3$, there exists a symmetric $r$-wise intersecting family $\f \subset \pn$ with
\[\log_2 |\f| = n - (r-1)n^{(r-2)/(r-1)} + o\left(n^{(r-1)/r}\right),\]
for infinitely many $n \in \mathbb{N}$. It would be very interesting to determine, for each integer $r \geq 3$, the asymptotic behaviour of the function $f_r$, defined by
\[f_r(n) = \max\{|\f|:\ \f \subset \pn,\ \f \text{ is symmetric and }r\text{-wise intersecting}\}.\]

It is also natural to consider intersection problems about $k$-uniform families, under additional symmetry requirements. For positive integers $n$ and $k$ with $k \leq n/2$, we write
$$s(n,k) := \max\{|\f|:\ \f \subset {[n] \choose k},\ \f \text{ is symmetric and intersecting}\}.$$
The determination of $s(n,k)$ is the `symmetric' equivalent of Question \ref{question:ekr}. In \cite{ekn}, it is proved that
\begin{equation}
\label{eq:upper-sym}
s(n,k) \le \exp\left(-\frac{c(n-2k)\log n}{k( \log n - \log k)} \right) \binom{n}{k},
\end{equation}
where $c >0$ is an absolute constant. Our proof proceeds by approximating $|\f|/{n \choose k}$ by the $p$-biased measure of the up-closure $\mu_p(\f^{\uparrow})$, where $p \approx k/n$ (using Lemma \ref{lem:chernoff}), then applying the Friedgut-Kalai `sharp threshold' theorem (Theorem \ref{thm:fk}, above) to $\f^{\uparrow}$, and then finally using the fact that $\mu_{1/2}(\f^{\uparrow}) \leq 1/2$ (since $\f^{\uparrow} \subset \pn$ is an intersecting family).

We also give a construction showing that for $\sqrt{n} \log n \leq k \leq n/2$, we have
\begin{equation}
\label{eq:lower-sym} s(n,k) \geq \exp\left(-(1+C/\log n)\left(\frac{\log n - \log k}{\log n - \log (n-k)}\right) \log n + \log n \right) \binom{n}{k},\end{equation}
where $C>0$ is an absolute constant. The upper bound (\ref{eq:upper-sym}) and the lower bound (\ref{eq:lower-sym}) together imply that if $k = k(n) \le n/2$, then as $n \to \infty$,
$$s(n,k) = o\left(\binom{n-1}{k-1}\right) \quad \text{iff}  \quad  \frac{1}{1/2 - k/n} = o(\log n).$$
This in turn determines roughly the threshold at which the imposition of the symmetry requirement, forces an upper bound which is $o(1)$-fraction of the bound in the Erd\H{o}s-Ko-Rado theorem. However, there is still a significant gap between the upper and lower bounds above, and it would be interesting to narrow this gap. Even the asymptotic behaviour of the function
$$g(n): = \min\{k:\ s(n,k) >0\}$$
is not known, though it follows from known results that $\sqrt{n} < g(n) \leq 1.1527 \sqrt{n}$ for all $n \in \mathbb{N}$. It would be interesting to determine whether $g(n) = (1+o(1))\sqrt{n}$. As we outline in \cite{ekn}, this problem is connected to a problem in additive combinatorics, raised e.g.\ in \cite{banakh}. We say a subset $S \subset \mathbb{Z}_n$ is a {\em difference cover for $\mathbb{Z}_n$} if $S-S = \mathbb{Z}_n$. (Recall that for a set $S \subset \mathbb{Z}_n$, we define its {\em difference set} by $S-S : = \{s-t:\ s,t \in S\}$.) For $n \in \mathbb{N}$, we write $h(n) := \min\{|S|: S \subset \mathbb{Z}_n,\ S \text{ is a difference cover for }\mathbb{Z}_n\}$. It is easy to see that $g(n) \leq h(n)$ for all $n \in \mathbb{N}$, so if $h(n) = (1+o(1))\sqrt{n}$ then it would follow that $g(n) = (1+o(1))\sqrt{n}$. It is not yet known, however, whether $h(n) =(1+o(1))\sqrt{n}$. We believe that the asymptotic determination of $h(n)$ is an interesting problem in its own right.

\thankyou{We thank Yuval Filmus, Ehud Friedgut, Dylan King, Imre Leader and Eoin Long for very helpful discussions and comments. We thank an anonymous reviewer, and the editors of the Proceedings of the 29th BCC, for their careful reading of the paper, and for their helpful comments and suggestions, which we have incorporated.}

\myaddress



\begin{thebibliography}{99}
\bibitem{push-pull} R. Ahlswede and L. H. Khachatrian. A Pushing-Pulling Method: New Proofs of Intersection Theorems. {\em Combinatorica} \textbf{19} (1999), 1--15.
\bibitem{AK} R. Ahlswede and L. H. Khachatrian. The complete intersection theorem for systems of finite
sets, {\it Eur. J. Combin.} \textbf{18} (1997), 125--136.
\bibitem{ak-nontrivial} R. Ahlswede and L. H. Khachatrian. The complete nontrivial-intersection theorem for systems
of finite sets. {\it J. Combin. Theory, Series A} \textbf{76} (1996), 121--138.
\bibitem{alon-talk} N. Alon. {\em Bill Tutte, and the Global Nature of Graph Colouring}. Talk given at the Tutte Centenary Conference, Trinity College, Cambridge, 10th-14th July 2017.
\bibitem{alon-spencer} N. Alon and J. Spencer. {\em The Probabilistic Method}. Wiley, New York, 1990.
\bibitem{babai-frankl} L. Babai and P. Frankl. {\em Linear Algebra Methods in Combinatorics.} Department of Computer Science, University of Chicago, 2020.\\
\textit{https://people.cs.uchicago.edu/$\sim$laci/CLASS/HANDOUTS-COMB/BaFrNew.pdf}.
\bibitem{linz-personal} J. Balogh and W. Linz. Short proofs of three results about intersecting systems. Preprint. arXiv:2104.00778.
\bibitem{banakh} T. O. Banakh and V. M. Gavrylkiv. Difference bases in cyclic groups. {\em J. Algebra Appl.} \textbf{18} (2019), 1950081.
\bibitem{barany} I. B\'ar\'any. Note: A Short Proof of Kneser's Conjecture. {\em J. Combin. Theory, Series A} \textbf{25} (1978), 325--326.
\bibitem{berge} C. Berge. A theorem related to the Chv\'atal conjecture. In: {\em Proceedings of the 5th British Combinatorial Conference} (University of Aberdeen, Aberdeen, 1975), Congressus Numerantium, vol.\ XV, {\em Utilitas Math.}, Winnipeg, Man., 1976, pp. 35--40.
\bibitem{bz} A. Berger and Y. Zhao. $K_4$-intersecting families of graphs. Preprint, April 2021. arXiv:2103.12671.
\bibitem{berlekamp} E. R. Berlekamp. On subsets with intersections of even cardinality. {\em Canad. Math. Bull.} \textbf{12} (1969), 471--477.
\bibitem{bfrt} T. Bohman, A. Frieze, M. Ruszink\'o and L. Thoma. \textit{G}-intersecting families. {\em Combin. Probab. Comput.} \textbf{10} (2001), 376--384.
\bibitem{bm} T. Bohman and R. Martin. A note on \textit{G}-intersecting families. {\em Discrete Math.} \textbf{260} (2003), 183--188.
\bibitem{bollobas} B. Bollob\'as. {\em Combinatorics: Set Systems, Hypergraphs, Families of Vectors and Combinatorial Probability}. Cambridge University Press, Cambridge, 1986.
\bibitem{bde} B. Bollob\'as, D.E. Daykin and P. Erd\H{o}s. Sets of independent edges of a hypergraph. {\em Quart. J. Math. Oxford (Series 2)}, \textbf{21} (1976), 25--32.
\bibitem{bl-comp} B. Bollob\'as and I. Leader. Compressions and isoperimetric inequalities. {\em J. Combin. Theory, Series A} \textbf{56} (1991), 47--62. 
\bibitem{bl-grid} B. Bollob\'as and I. Leader. Edge-isoperimetric inequalities in the grid. {\em Combinatorica} \textbf{11} (1991), 299--314.
\bibitem{borg} P. Borg. Intersecting families of sets and permutations: a survey. {\em International Journal of Mathematics, Game Theory and Algebra} \textbf{21} (2012), 543--559.
\bibitem{pjc} P. Cameron,  P. Frankl, and W. M. Kantor. Intersecting families of finite sets and fixed-point-free 2-elements. {\em Europ. J. Combin.} \textbf{10} (1989), 149--160.
\bibitem{ck} P. Cameron and C.Y. Ku. Intersecting families of permutations. {\em Europ. J. Combin.} \textbf{24} (2003), 881--890.
\bibitem{cara} C. Carath\'eodory and E. Study. Zwei Beweise des Satzes, dass der Kreis unter alien Figuren gleichen Umfanges den gr\"ossten Inhalt hat. {\em Math. Annalen} \textbf{68} (1909), 133--140.
\bibitem{christofides-personal} D. Christofides. Denser families of graphs intersecting in paths of length three. Manuscript, 2011.
\bibitem{cfgs} F. R. K. Chung, P. Frankl, R. L. Graham and J. B. Shearer. Some intersection theorems for ordered sets and graphs. {\em J. Combin. Theory, Series A} \textbf{43} (1986), 23--37.
\bibitem{chvatal} V. Chv\'atal. Intersecting families of edges in hypergraphs having the hereditary property. Hypergraph Seminar, {\em Lecture Notes in Mathematics}, vol. 411, Springer-Verlag, Berlin (1974), pp. 61--66.
\bibitem{daykin-kk} D. E. Daykin. A simple proof of the Kruskal-Katona theorem. {\em J. Combin. Theory, Series A} \textbf{17} (1974), 252--253.
\bibitem{daykin-ekr} D. E. Daykin. Note: Erd\H{o}s-Ko-Rado from Kruskal-Katona. {\em J. Combin. Theory, Series A} \textbf{17} (1974) 254--255.
\bibitem{dhm} D. E. Daykin, A. J. W. Hilton and D. Mikl\'os. Pairings from down-sets and up-sets in distributive lattices. {\em J. Combin. Theory, Series A} \textbf{34} (1983), 215--230.
\bibitem{dinur-friedgut} I. Dinur and E. Friedgut. Intersecting families are essentially contained in juntas. {\em Combin. Probab. Comput.} \textbf{18} (2009), 107--122.
\bibitem{dshardness} I. Dinur and S. Safra. On the hardness of approximating vertex cover. {\em Ann. Math.} \textbf{162} (2005), 439--485.
\bibitem{delsarte-73} P. Delsarte. An Algebraic Approach to the Association Schemes of Coding Theory. {\em Philips Research Reports}, Supplement No.\ 10 (1973).
\bibitem{def} M. Deza, P. Erd\H{o}s and P. Frankl. Intersection properties of systems of finite sets. {\em Proc. London
Math. Soc.} \textbf{36} (1978), 369--384.
\bibitem{df} M. Deza and P. Frankl. On the maximum number of permutations with given maximal or minimal distance. {\em J. Combin. Theory, Series A} \textbf{22} (1977), 352--360.
\bibitem{ellis-setwise} D. Ellis. Setwise intersecting families of permutations. {\em J. Combin. Theory, Series A} \textbf{119} (2012), 825--849.
\bibitem{ellis-dfs} D. Ellis. Stability for $t$-intersecting families of permutations. {\em J. Combin. Theory, Series A} \textbf{118} (2011), 208--227.
\bibitem{eff} D. Ellis, Y. Filmus and E. Friedgut. Triangle-intersecting families of graphs. {\em J. Eur. Math. Soc.} \textbf{14} (2012), 841--885.
\bibitem{efp} D. Ellis, E. Friedgut and H. Pilpel. Intersecting families of permutations. {\em J. Amer. Math. Soc.} \textbf{24} (2011), 649--682. 
\bibitem{ekn} D. Ellis, G. Kalai and B. Narayanan. On symmetric intersecting families. {\em Europ. J. Combin.} \textbf{86} (2020), 103094.
\bibitem{ekl-iso} D. Ellis, N. Keller and N. Lifshitz, On a biased edge isoperimetric inequality for the discrete cube. {\em J. Combin. Theory, Series A} \textbf{163} (2019), 118-162.
\bibitem{ekl-entropy} D. Ellis, N. Keller and N. Lifshitz. Stability for the Complete Intersection Theorem, and the
Forbidden Intersection Problem of Erd\H{o}s and S\'os. Preprint. arXiv:1604.06135.
\bibitem{ekl} D. Ellis, N. Keller and N. Lifshitz. Stability versions of Erd\H{o}s-Ko-Rado type theorems, via isoperimetry. {\em J. Eur. Math. Soc.} \textbf{21} (2019), 3857-3902.
\bibitem{ekl-prep} D. Ellis, G. Kindler and N. Lifshitz. Families of linear maps with a forbidden intersection. In preparation.
\bibitem{el} D. Ellis and N. Lifshitz. Approximation by juntas in the symmetric group, and forbidden intersection problems. To appear in {\em Duke Math. J.}
\bibitem{en} D. Ellis and B. Narayanan. On symmetric 3-wise intersecting families. {\em Proc. Amer. Math. Soc.} \textbf{145} (2017), 2843--2847.
\bibitem{er} D. Ellis and W. Raynaud. {\em Families of sets that are pairwise close}. PhD thesis of W. Raynaud, Chapter 4. Queen Mary, University of London, 2020. Available at \textit{https://qmro.qmul.ac.uk/xmlui/handle/123456789/68552}.
\bibitem{ekr} P. Erd\H{o}s, C. Ko and R. Rado, Intersection theorems for systems of finite sets, {\it
Quart. J. Math. Oxford (Series 2)} \textbf{12} (1961), 313--320. 
\bibitem{emc} P. Erd\H{o}s. A problem on independent $r$-tuples. {\em Ann. Univ. Sci. Budapest.} \textbf{8} (1965), 93--95.
 \bibitem{es} P. L. Erd\H{o}s and L. A. Sz\'ekely.  Erd\H{o}s-Ko-Rado theorems of higher order. In: {\em Numbers, Information and Complexity}, I. Alth\"ofer, N. Cai, G. Dueck, L. Khachatrian, M. S. Pinsker, A. S\'ark\"ozy, I. Wegener and Z. Zhang (Eds.), pp.\ 117-124. Springer, New York, 2000.
\bibitem{frankl-new} P. Frankl. Improved bounds for Erd\H{o}s' Matching Conjecture. {\em J. Combin Theory, Series A} \textbf{120} (2013), 1068--1072.
\bibitem{frankl-77} P. Frankl. On families of finite sets no two of which intersect in a singleton. {\em Bull. Austral.
Math. Soc.} \textbf{17} (1977), 125--134.
\bibitem{frankl-78-int} P. Frankl. On intersecting families of finite sets. {\em J. Combin. Theory, Series A} \textbf{24} (1978), 146--161.
\bibitem{frankl-other-range} P. Frankl. Proof of the Erd\H{o}s matching conjecture in a new range. {\em Isr. J. Math.} \textbf{222} (2017), 421--430.
\bibitem{frankl-sym} P. Frankl. Regularity conditions and intersecting hypergraphs. {\em Proc. Amer. Math.
Soc.} \textbf{82} (1981), 309--311.
\bibitem{ff} P. Frankl and Z. F\"uredi. Forbidding just one intersection. {\em J. Combin. Theory, Series A} \textbf{39}
(1985), 160--176.
\bibitem{ff-un} P Frankl and Z F\"uredi. On hypergraphs without two edges intersecting in a given number of vertices. {\em J. Combin. Theory, Series A}
\textbf{36} (1984), 230--236.
\bibitem{fk} P. Frankl and A. Kupavskii. The Erd\H{o}s Matching Conjecture and Concentration Inequalities. Preprint. arXiv:1806.08855.
\bibitem{fr87} P. Frankl and V. R\"odl. Forbidden intersections. {\em Trans. Amer. Math. Soc.} \textbf{300} (1987), 259--286.
\bibitem{frankl-survey} P. Frankl and N. Tokushige. Invitation to intersection problems for finite sets. {\em J. Combin. Theory, Series A}
\textbf{144} (2016), 157--211.
\bibitem{fw81} P. Frankl and R. M. Wilson. Intersection theorems with geometric consequences. {\em Combinatorica} \textbf{1} (1981), 357--368.
\bibitem{fwvs} P. Frankl and R. Wilson. The Erd\H{o}s-Ko-Rado theorem for vector spaces. {\em J. Combin. Theory, Series A} \textbf{43} (1986), 228--236.
\bibitem{friedgut-measure} E. Friedgut. On the measure of intersecting families, uniqueness and stability. {\em Combinatorica} \textbf{28} (2008), 503--528.
\bibitem{fkkk} E. Friedgut, J. Kahn, G. Kalai and N. Keller. Chv\'atal's conjecture and correlation inequalities. {\em J. Combin. Theory, Series A} \textbf{156} (2018), 22--43.
\bibitem{godsil-meagher-survey} C. Godsil and K. Meagher. {\em Erd\H{o}s-Ko-Rado theorems: Algebraic Approaches.} Cambridge University Press, Cambridge, 2015.
\bibitem{gm} C. Godsil and K. Meagher. A new proof of the Erd\H{o}s-Ko-Rado theorem for intersecting families of permutations. {\em Europ. J. Combin.} 30 (2009), 404--414.
\bibitem{grey} A. de Grey. The chromatic number of the plane is at least 5. {\em Geombinatorics} \textbf{28} (2018), 5--18.
\bibitem{grimmett} G. R. G. Grimmett. {\em Percolation}. 2nd edition, Springer-Verlag, Berlin 1999.
\bibitem{hilton-milner} A. J. W. Hilton and E. C. Milner. Some intersection theorems for systems of finite sets. {\em Quart. J. Math. Oxford, Series 2} \textbf{18} (1967), 369--384.
\bibitem{harper} L. Harper. Optimal numberings and isoperimetric problems on graphs. {\em J. Combin. Theory}
\textbf{1} (1966), 385--393.
\bibitem{hls} H. Huang, P.-S. Loh and B. Sudakov. The Size of a Hypergraph and its Matching Number. {\em Combin. Probab. Comput.} \textbf{21} (2012), 442--450.
\bibitem{isbell} J. R. Isbell. Homogeneous games, II. {\em Proc. Amer. Math. Soc.} \textbf{11} (1960), 159--161.
\bibitem{jt} J. R. Johnson and J. Talbot. \textit{G}-intersection theorems for matchings and other graphs. {\em Combin. Probab. Comput.} \textbf{17} (2008), 559--575.
\bibitem{kahn-kalai-borsuk} J. Kahn and G. Kalai. A counterexample to Borsuk's conjecture. {\em Bull. Amer. Math. Soc.} \textbf{29} (1993), 60--62.
\bibitem{kkl} J. Kahn, G. Kalai and N. Linial. The influence of variables on Boolean functions. {\em Proceedings of the 29th Annual Symposium on Foundations of Computer Science}, 1988, pp. 68--80.
\bibitem{kalai-conj} G. Kalai. Intersecting families of triangulations. {\em MathOverflow}, 2012.\\
\textit{https://mathoverflow.net/questions/114646/intersecting-family-of-triangulations}.
\bibitem{katona} G. O. H. Katona. A simple proof of the Erd\H{o}s -- Chao Ko -- Rado theorem. {\em J. Combin. Theory, Series B} \textbf{13} (1972), 183--184.
\bibitem{katona-t} G. O. H. Katona. Intersection theorems for systems of finite sets. {\em Acta Math. Acad. Sci. Hung.}
\textbf{15} (1964), 329--337.
\bibitem{keevash-long} P. Keevash and E. Long. Frankl-R\"odl-type theorems for codes and permutations. {\em Trans. Amer. Math. Soc.} \textbf{369} (2017), 1147--1162.
\bibitem{keller-lifshitz} N. Keller and N. Lifshitz, The junta method for hypergraphs, and the Erd\H{o}s-Chv\'atal simplex conjecture. {\em Adv. Math.}, to appear. arXiv:1707.02643.
\bibitem{kindler-safra} G. Kindler and S. Safra. Noise-resistant Boolean functions are juntas. Preprint. \textit{https://www.cs.huji.ac.il/$\sim$gkindler/papers/noise-stable-r-juntas.ps}.
\bibitem{kleitman} D. J. Kleitman. Extremal hypergraph problems. In: {\em Proceedings of the 7th British Combinatorial Conference} (B. Bollob\'as, ed.), pp. 44--65, Cambridge University Press, Cambridge, 1979.
\bibitem{kleitman-emc} D. J. Kleitman. Maximal number of subsets of a finite set no $k$ of which are pairwise disjoint. {\em J. Combin. Theory} \textbf{5} (1968), 157--163.
\bibitem{kneser} M. Kneser. Aufgabe 360, {\em Jahresbericht der D.M.V.} \textbf{58} (1955), p.\ 27.
\bibitem{kro} Y. Kohayakawa and V. R\"odl. Szemer\'edi's Regularity Lemma and Quasi-randomness. In: {\em Recent Advances in Algorithms and Combinatorics}, CMS Books Math. / Ouvrages Math. SMC, vol. 11, Springer, New York, 2003, pp. 289--351.
\bibitem{ksss} J . Koml\'os, A. Shokoufandeh, M. Simonovits and E. Szemer\'edi. The regularity
lemma and its applications in graph theory. In: {\em Theoretical Aspects of Computer Science}
(Tehran, 2000), Lecture Notes in Computer Science, vol.\ 2292, Springer, Berlin, 2002, pp. 84--112.
\bibitem{ks} J. Koml\'os and M. Simonovits. Szemer\'edi's regularity lemma and its applications in
graph theory. In: {\em Combinatorics, Paul Erdos is Eighty} Vol.\ 2 (Keszthely, 1993), Bolyai Soc. Math. Stud., vol. 2, Janos Bolyai Math. Soc., Budapest, 1996, pp. 295--352.
\bibitem{kr} C. Y. Ku and D. Renshaw. Erd\H{o}s-Ko-Rado theorems for permutations and set partitions. {\em J. Combin. Theory, Series A} \textbf{115} (2008), 1008--1020.
\bibitem{larman-rogers} D. G. Larman and C. A. Rogers. The realization of distances within sets in Euclidean space. {\em Mathematika} \textbf{19} (1972), 1--24.
\bibitem{lm} B. Larose and C. Malvenuto. Stable sets of maximal size in Kneser-type graphs. {\em Europ. J. Combin.} \textbf{25} (2004), 657--673.
\bibitem{leader-survey} I. Leader. Discrete isoperimetric inequalities. In: {\em Probabilistic Combinatorics and Its Applications}, B. Bollob\'as (Ed.), Amer. Math. Soc., Providence, 1991, pp.\ 57--80.
\bibitem{lsp-orig} C. H. Li, S. J. Song and V. R. T. Pantangi. Erd\H{o}s-Ko-Rado problems for permutation groups. Preprint, version 1 (2020). arXiv:2006.10339v1.
\bibitem{li} C. H. Li, S. J. Song and V. R. T. Pantangi. Erd\H{o}s-Ko-Rado problems for permutation groups. Preprint, version 2 (2021). arXiv:2006.10339v2.
\bibitem{ly} J. Liu and W. Yang. Set systems with restricted $k$-wise $\mathscr{L}$-intersections modulo a prime number. {\em Europ. J. Combin.} \textbf{36} (2014), 707--719.
\bibitem{long-personal} E. Long. Personal communication, 2016.
\bibitem{lovasz-kneser} L. Lov\'asz. Kneser's conjecture, chromatic number, and homotopy. {\em J. Combinatorial
Theory} \textbf{25} (1978), 319--324. 
\bibitem{lovasz} L. Lov\'asz. On the Shannon Capacity of a Graph. {\em IEEE Trans. Inf. Theory} \textbf{25} (1979), 1--7.
\bibitem{luczak} T. {\L}uczak. Randomness and regularity. In: {\em Proceedings of the International Congress of Mathematicians},
Vol.\ III , Eur. Math. Soc., Zurich, 2006, pp. 899--909.
\bibitem{mm} K. Meagher and L. Moura. Erd\H{o}s-Ko-Rado theorems for uniform set-partitions. {\em Electronic J. Combin.} \textbf{12} (2005), R.40.
\bibitem{mr} K. Meagher and A. S. Razafimahatratra. 2-intersecting permutations. Preprint. arXiv:2005.00139.
\bibitem{mrs} K. Meagher, A. S. Razafimahatratra and P. Spiga. On triangles in derangement graphs. Preprint. arXiv:2009.01086.
\bibitem{mspiga} K. Meagher and P. Spiga. An Erd\H{o}s-Ko-Rado theorem for the derangement graph of $\text{PGL}(2,q)$ acting on the projective line. {\em J. Combin. Theory, Series A} \textbf{118} (2011) 532--544.
\bibitem{mspiga2} K. Meagher and P. Spiga. An Erd\H{o}s-Ko-Rado Theorem for the Derangement Graph of $\text{PGL}_3(q)$ Acting on the Projective Plane. {\em SIAM J. Discrete Math.} \textbf{28} (2014), 918--941.
\bibitem{mst-2-trans} K. Meagher, P. Spiga and Pham Huu Tiep. An Erd\H{o}s-Ko-Rado theorem for finite 2-transitive groups. {\em Europ. J. Combin.} \textbf{55} (2016), 100--118.
\bibitem{mr} D. Mubayi and V. R\"odl. Specified Intersections. {\em Trans. Amer. Math. Soc.} \textbf{366} (2014), 491--504.
\bibitem{odonnell} R. O'Donnell. {\em Analysis of Boolean Functions.} Cambridge University Press, Cambridge, 2014.
\bibitem{olarte} J. A. Olarte, F. Santos, J. Spreer and C. Stump. The EKR property for flag pure simplicial complexes without boundary. {\em J. Combin. Theory, Series A} \textbf{172} (2020), 105205.
\bibitem{raigordskii} A. M. Raigordskii. On the chromatic number of a space. {\em Russ. Math. Surv. } \textbf{55} (2000), 351--352.
\bibitem{russell} P. Russell. Families intersecting on an interval. {\em Discrete Math.} \textbf{309} (2009), 2952--2956.
\bibitem{russo} L. Russo. An approximate zero-one law. {\em Z. Wahrsch. Verw. Gebiete} \textbf{61} (1982), 129--139.
\bibitem{schramm} O. Schramm. Illuminating sets of constant width. {\em Mathematika} \textbf{35} (1988), 180--189.
\bibitem{snevily} H. Snevily. A new result on Chv\'atal's conjecture. {\em J. Combin. Theory, Series A} \textbf{61} (1992), 137--141.
\bibitem{steiner} J. Steiner. {\em Gesammelte Werke.} K. Weierstrass (Ed.). Vol.\ 2. Berlin, 1882.
\bibitem{wang} D. L. Wang. On systems of finite sets with constraints on their unions and intersections. {\em J.
Combin. Theory, Series A} \textbf{23} (1977), 344--348.
\bibitem{wilson} R.M. Wilson. The exact bound in the Erd\H{o}s-Ko-Rado theorem. {\it Combinatorica} \textbf{4}
(1984), 247--257.








\end{thebibliography}
\end{document}